\documentclass[10pt]{amsart}
\usepackage{amssymb}

\usepackage{amsfonts,stmaryrd, amssymb,enumerate, euscript}
\usepackage{amsmath,amscd, mathtools}
\usepackage[dvips]{graphicx}
\usepackage{mathrsfs}
\usepackage[all]{xy}
\usepackage{color}

\newtheorem{Thm}{Theorem}[section]
\newtheorem{Lem}[Thm]{Lemma}
\newtheorem{Def}[Thm]{Definition}
\newtheorem{Cor}[Thm]{Corollary}
\newtheorem{Prop}[Thm]{Proposition}
\newtheorem{Ex1}[Thm]{Example}
\newtheorem{Rem1}[Thm]{Remark}

\newcommand{\DD}{\EuScript{D}}

\newcommand{\Ker}{\mathop{\mathrm{Ker}}\nolimits}\newcommand{\Modu}{\mathop{\mathrm{Mod}}\nolimits}

\newcommand{\op}{\mathop{\mathrm{op}}\nolimits}
\newcommand{\HH}{\mathop{\mathrm{HH}}\nolimits}
\newcommand{\Hu}{\mathop{\mathrm{H}}\nolimits}
\newcommand{\sg}{\mathop{\mathrm{sg}}\nolimits}

\newcommand{\HC}{\mathop{\mathrm{HC}}\nolimits}

\newcommand{\End}{\mathop{\mathrm{End}}\nolimits}
\newcommand{\Barr}{\mathop{\mathrm{Bar}}\nolimits}
\newcommand{\Map}{\mathop{\mathrm{Map}}\nolimits}
\newcommand{\Hom}{\mathop{\mathrm{Hom}}\nolimits}
\newcommand{\Ext}{\mathop{\mathrm{Ext}}\nolimits}
\newcommand{\Tor}{\mathop{\mathrm{Tor}}\nolimits}
\newcommand{\Tot}{\mathop{\mathrm{Tot}}\nolimits}
\newcommand{\id}{\mathop{\mathrm{id}}\nolimits}
\newenvironment{Rem}{\begin{Rem1}\rm}{\end{Rem1}}

\title[Tate-Hochschild cohomology ring]{The
Batalin-Vilkovisky structure on the Tate-Hochschild cohomology ring of
a group algebra}

\author{Yuming Liu, Zhengfang Wang and Guodong Zhou}

\address{Yuming Liu
\newline School of Mathematical Sciences
\newline Laboratory of Mathematics and Complex Systems
\newline Beijing Normal University
\newline Beijing 100875
\newline P.R.China}
\email{ymliu@bnu.edu.cn}

\address{Zhengfang Wang
\newline Beijing International Center for Mathematical Research (BICMR)
\newline Peking University
\newline Beijing 100871
\newline P.R.China}
\email{wangzhengfang@bicmr.pku.edu.cn}

\address{Guodong Zhou
 \newline School of Mathematical Sciences
 \newline Shanghai Key laboratory of PMMP
\newline East China Normal University
\newline Shanghai 200241
 \newline P.R.China}
  \email{gdzhou@math.ecnu.edu.cn}

\date{version of \today}

\newcommand{\lra}{\longrightarrow}

\newcommand{\ra}{\rightarrow}
\newcommand{\sdp}{\times\kern-.2em\vrule height1.1ex depth-.05ex}
\newcommand{\epi}{\lra \kern-.8em\ra}

\newcommand{\Z}{{\mathbb Z}}

 \normalbaselines
\thanks{
}

\begin{document}

\renewcommand{\thefootnote}{\alph{footnote}}
\renewcommand{\thefootnote}{\alph{footnote}}
\setcounter{footnote}{-1} 
\renewcommand{\thefootnote}{\alph{footnote}}
\setcounter{footnote}{-1} 

\begin{abstract} We determine the Batalin-Vilkovisky structure on the Tate-Hochschild cohomology of the group algebra $kG$ of a finite group $G$ in terms of the additive decomposition. In particular, we show that the Tate cohomology of $G$ is a Batalin-Vilkovisky subalgebra of the Tate-Hochschild cohomology of the group algebra $kG$, and that the Tate cochain complex of $G$ is a cyclic $A_{\infty}$-subalgebra of the Tate-Hochschild cochain complex of $kG$.   
\end{abstract}

\maketitle

\section*{Introduction}
 For any associative algebra $A$, Hochschild introduced in 1945 a graded group $\HH^*(A, A)$ defined as the cohomology of certain cochain complex $C^*(A, A)$, where $C^n(A, A)$ is the space of linear maps from $A^{\otimes n}$ to $A$. In the 1960's, when studying the deformation theory of associative algebras, Gerstenhaber observed that there is a rich algebraic structure on $\HH^*(A, A)$, called a Gerstenhaber algebra, consisting of the following date: 
\begin{enumerate}[(i)]
 \item $\HH^*(A, A)$ is a graded-commutative associative algebra via the cup product;
  \item $\HH^*(A, A)$ is endowed with  a graded  Lie bracket  (nowadays called {\it Gerstenhaber bracket}) of degree $-1$  so that it becomes  a graded Lie algebra (of degree $-1$); 
 \item The Gerstenhaber bracket is compatible with the cup product via the graded Leibniz rule. 
 \end{enumerate}

 During the past few decades, a new structure (the so-called {\it Batalin-Vilkovisky structure})  has been extensively studied in topology and mathematical physics, and recently it was introduced into algebra. Roughly speaking,  a Batalin-Vilkovisky (aka. BV) structure on Hochschild cohomology  is a square-zero operator (called {\it BV-operator}) of degree $-1$ such that the Gerstenhaber bracket is the obstruction of the BV-operator being a derivation with respect to the cup product. 
A typical example of a BV-algebra was given by Tradler \cite{Tradler2008} and Menichi \cite{Menichi2004} motivated from string topology. Namely,  the Hochschild cohomology ring of a finite dimensional symmetric algebra (e.g. the  group algebra of a finite group) is a BV-algebra. 


\bigskip
From the point of view of derived categories,   the $i$-th Hochschild cohomology group $\HH^i(A, A)$ of  an  algebra $A$ over a field $k$ is isomorphic to  the space of morphisms from $A$ to $A[i]$ in $\DD^b(A\otimes_k A^{\op})$, the bounded derived category of $A$-$A$-bimodules, where $[i]$ is the $i$-th shift functor in $\DD^b(A\otimes_k A^{\op})$.  As a generalization of  the Hochschild cohomology,  the Tate-Hochschild cohomology group $\widehat{\HH}^i(A, A) \ (i\in \mathbb Z)$  is defined as the space of morphisms from $A$ to $A[i]$ in the singularity category $\DD_{\sg}(A\otimes_k A^{\op})$ of  $A$-$A$-bimodules, where $[i]$ is the $i$-th shift functor in $\DD_{\sg}(A\otimes_k A^{\op})$. Recall that $\DD_{\sg}(A\otimes_k A^{\op})$ is the Verdier quotient of the bounded derived category $\DD^b(A\otimes_k A^{\op})$ by the full subcategory consisting of bounded complexes  of  projective $A$-$A$-bimodules, which was introduced by Buchweitz \cite{Buchweitz1986} and later independently by Orlov  \cite{Orlov}.   As a particular example that will be relevant in this paper, the Tate-Hochschild cohomology $\widehat{\HH}^*(kG, kG)$ of a finite abelian group $G$  is isomorphic to $kG\otimes \widehat{\mathrm H}^*(G, k)$ as graded associative algebras, where $\widehat{\mathrm H}^*(G, k)$ is the Tate cohomology of $G$ (cf. Section \ref{subsection-Tate-cohomology}).  The notion of Tate-Hochschild  cohomology  has been  studied in the literature, such as    \cite{Buchweitz1986, BJ2011, Nguyen2012}. 

As already mentioned,  the Hochschild cohomology $\HH^*(A, A)$ may be computed by the Hochschild cochain complex $C^*(A, A)$. 
In \cite{Wang2015} the second named author    constructed a complex, the so-called  ``singular Hochschild cochain complex" to compute the Tate-Hochschild cohomology  $\widehat{\HH}^*(A, A)$. 
Via this complex, the author in loc. cit.  showed that  there is  a Gerstenhaber structure on $\widehat{\HH}^*(A, A)$  extending the classical Gerstenhaber structure on $\HH^*(A, A)$ (cf. Theorem \ref{Theorem-Wang2015} below).  Generalizing the result of Tradler and Menichi, the author proved that the Gerstenhaber structure of the  Tate-Hochschild cohomology  of a symmetric algebra extends to the  BV structure on the Hochschild cohomology (cf. Theorem \ref{Theorem-BV-algebra} below). 

In a later work \cite{RiveraWang2017}, the authors studied the Tate-Hochschild cohomology of finite dimensional  differential graded (dg) symmetric algebras, from the point of view of string topology. Generalizing the classical Tate cochain complex of a finite group (cf. Section \ref{subsection-Tate-cohomology}),  the authors constructed  an analogous complex $\mathcal D^*(A, A)$ (called {\it Tate-Hochschild cochain complex}) computing the Tate-Hochschild cohomology of a symmetric algebra $A$. The negative part of $\mathcal D^*(A, A)$ is the Hochschild chain complex $C_*(A, A)$ with $\mathcal D^{-m-1}(A, A)=C_m(A, A) \ (m\geq 0)$; the non-negative part of $\mathcal D^*(A, A)$ is the Hochschild cochain complex $C^*(A, A)$ with $\mathcal D^m(A, A)=C^m(A, A)\ (m\geq 0)$; and the differential $\tau: C_0(A, A)\rightarrow C^0(A, A)$ in degree $-1$ is given by $a\mapsto \sum_{i} e_iaf_i$, where $\sum_i e_i\otimes f_i$ is the Casimir element of $A$ (cf. Section \ref{subsection-symmetric}). It is shown in loc. cit.  that there is a cyclic $A_{\infty}$-algebra structure $(m_1=\partial, m_2, m_3, \cdots)$   and an $L_{\infty}$-algebra  structure $(l_1=\partial, l_2, l_3, \cdots)$ on $\mathcal D^*(A, A)$ such that $m_i=0$ for $i>3$. Moreover, the restrictions of $m_2$ and $l_2$ to the non-negative part $\mathcal D^{\geq 0}(A, A)$ are respectively,   the usual cup product and the Gerstenhaber bracket on $C^*(A, A)$. 



\bigskip

{\it The aim of the present article is to describe explicit complex level formulas for  the BV structure on $\widehat{\HH}^*(A,A)$ in a special case where $A=kG$ is the group algebra of a finite group $G$ over a field $k$.}

\bigskip
It is well-known that the Hochschild (co)homology  of the group algebra of a finite group admits  a decomposition as vector spaces into a direct sum of group (co)homology spaces of centralizers of elements. More precisely, let $k$ be a field and $G$ a finite group. Then we have the following additive decompositions (cf. e.g. \cite[Theorem 2.11.2]{Benson}): 
$$
\begin{array}{rcl}
\HH^*(kG,kG) &\simeq&   \bigoplus\limits_{x\in X}\mathrm H^*(C_G(x),k),\\
\HH_*(kG, kG) & \simeq & \bigoplus\limits_{x\in X}\mathrm H_*(C_G(x), k)
\end{array}
$$ where $X$ is a set of representatives of conjugacy classes of elements of $G$ and $C_G(x)$ is the centralizer of $x \in X$. Siegel and Witherspoon  in \cite{SW1999} gave a formula for the cup product of the Hochschild cohomology $\HH^*(kG,kG)$ in terms of the above  additive decomposition, and  later, Bouc in \cite{Bouc2003} gave a quick proof of this formula using Green functors. In \cite{LZ2015}, the first and the third named authors lifted the additive decomposition of $\HH^*(kG,kG)$ to the complex level. More concretely,  the authors lifted the above isomorphism on $\HH^*(kG, kG)$  to a chain homotopy equivalence given by   two maps of complexes (cf. Theorem \ref{realization-Hochschild-cohomology} below)\begin{equation} \label{equation-comparison-map}
\xymatrix@C=0.000000001pc{
&C^*(kG, kG) \ar@<0.5ex>[rrrrr]^-{\iota^*}&&&&&\bigoplus\limits_{x\in X} C^*(C_G(x), k)\ar@<0.5ex>[lllll]^-{\rho^*}}
\end{equation}
such that $\iota^*\circ \rho^*=\id$ and $\id-\rho^*\circ \iota^*$ is homotopic to the zero map, where  $C^*(kG, kG)$ denotes the Hochschild cochain complex (see Section \ref{section2}) and  $C^*(C_G(x), k)$ the group cohomology complex (see Section \ref{subsection: remainder on group cohomology}).  As a result, they described explicitly how to transfer the  cup product, the Lie bracket and the BV-operator from the left-hand side to  the right-hand side at the complex level.   In the present article, we construct an explicit homotopy $s^*$ between $\id$ and $\rho^*\circ \iota^*$, namely $\delta^*\circ s^*+s^*\circ \delta^*=\id-\rho^*\circ \iota^*$ (cf. Theorem \ref{realization-Hochschild-cohomology}). Such a triple $( \rho^*,\iota^*, s^*)$ is called a {\it homotopy deformation retract}.


Combining the above two additive decompositions,  we obtain an  additive decomposition of the Tate-Hochschild cohomology $\widehat{\HH}^*(kG,kG)$.  That is,
\begin{equation*}\label{additive-decomposition-Hochschild}
\widehat{\HH}^*(kG,kG)\simeq \bigoplus_{x\in X}\widehat{\mathrm H}^*(C_G(x),k),
\end{equation*}
where we fix $X$ to be a set of representatives of conjugacy classes of elements of $G$ and $C_G(x)$ is the centralizer subgroup of $x\in G$, and where $\widehat{\mathrm H}^*(C_G(x),k)$ is the Tate  cohomology of $C_G(x)$ (see Section \ref{subsection-Tate-cohomology}).  In \cite{Nguyen2012}, Nguyen gave the cup product formula for the Tate-Hochschild cohomology $\widehat{\HH}^*(kG,kG)$ in terms of the above additive decomposition. In the present article, we shall  describe the cup product and the BV-operator on $\widehat{\HH}^*(kG,kG)$ in terms of the additive decomposition at the complex level.  By this, we mean that we will give an explicit formula for the cup product and the BV-operator on the complex $\mathcal D^*(kG, kG)$ and also give some explicit calculations in terms of the additive decomposition at the complex level. To achieve this,  we  extend the chain homotopy equivalence (\ref{equation-comparison-map}) to the following homotopy deformation retract (cf. Remark \ref{remark-deformation-retract2}), 
\begin{equation*}\xymatrix@C=0.000000001pc{
\ar@(lu,dl)_-{\widehat s}&\mathcal D^*(kG, kG) \ar@<0.5ex>[rrrrr]^-{\widehat{\rho}}&&&&&\bigoplus\limits_{x\in X} \widehat{ C}^*(C_G(x), k),\ar@<0.5ex>[lllll]^-{\widehat{\iota}}}
\end{equation*}
namely $\widehat{\rho}\circ\widehat{\iota}=\id$ and $\id-\widehat{\iota}\circ\widehat{\rho}=\partial\circ \widehat s+\widehat s\circ \partial$, where $\partial$ is the differential of $\mathcal D^*(kG, kG)$ and $\widehat{C}^*(C_G(x), k)$ is the Tate cochain complex of $C_G(x)$ (see Section \ref{subsection-Tate-cohomology}).  Via this homotopy deformation retract, we may  transfer the cup product  on the left hand side to the right hand side and thus  obtain a cup product formula at the complex level (cf. Remark \ref{remark-cup-product-negative}), as doing so for $\HH^*(kG, kG)$ in \cite[Section 7]{LZ2015}.  The BV-operator  preserves  each summand of the additive decomposition (cf. Section \ref{section-delta-operator}). As a consequence, we obtain that  $\widehat{\mathrm H}^*(G, k)$ is a BV subalgebra of $\widehat{\HH}^*(kG, kG)$ (cf. Corollary \ref{BV-sub-algebra}). 

We have an explicit formula for the cup product $\cup$ on $\mathcal D^*(kG, kG)$ (see  Definition \ref{Def-cup-product}). We observe that the restriction of $\cup$ to the negative part  $\mathcal D^{<0}(kG, kG)$  (i.e. the Hochschild chain complex $C_*(kG, kG)$)  is in general not compatible with the differential of $C_*(kG, kG)$ (cf. Remark \ref{remark-restriction-negative}).   
For this reason,  $\cup$ is not well-defined in the whole Hochschild homology  $\HH_*(kG, kG)$.  To deal with this issue, we shall consider the following truncated subcomplex of $\mathcal D^*(kG, kG)$,
$$\widetilde{C}_{*}(kG, kG)):\quad \cdots\xrightarrow{\partial_p} C_{p-1}(kG, kG) \xrightarrow{\partial_{p-1}} \cdots \xrightarrow{\partial_2}C_1(kG, kG) \xrightarrow{\partial_1} \Ker(\tau) \rightarrow 0.$$ 
It is clear that the  cohomology of $\widetilde{C}_{*}(kG, kG)$ is isomorphic to  the negative part $\widehat{\HH}^{<0}(kG, kG)$ of $\widehat{\HH}^*(kG, kG)$. Actually, it also coincides with the stable Hochschild homology $\HH^{st}_*(kG, kG)$ studied in \cite{HHKM2005, LiuZhouZimmermann2012} (cf. Remark \ref{remark-stable-hochschild} below).  
Similarly, we denote by  $\mathrm H_*^{st}(G, k)$ the negative part of the Tate cohomology $\widehat{\mathrm H}^*(G, k)$, namely $\mathrm H^{st}_{m}(G, k)=\widehat{\mathrm H}^{-m-1}(G, k) \ (m\geq 0)$. The additive decomposition of $\widehat{\HH}^*(kG, kG)$ yields an additive  decomposition 
$$\HH_*^{st}(kG, kG)\simeq \bigoplus_{x\in X} \mathrm H^{st}_*(C_G(x), k).$$ We prove that  $\HH_{-*-1}^{st}(kG, kG)$ is a BV-algebra (without unit) (cf. Theorem \ref{theorem-stable-Hochschild}). It is  well-known that there is an isomorphism between the Hochschild homology $\HH_*(kG, kG)$ and the singular homology $\mathrm H_*(LBG, k)$ of the free loop space $LBG$ of the classifying space $BG$ (cf. \cite[7.3.13 Corollary]{Lod}). We denote by $\mathrm H^{st}_*(LBG, k)$ the subspace of $\mathrm H_*(LBG, k)$ corresponding to $\HH_*^{st}(kG, kG)$  under the above isomorphism. Then we obtain that $\mathrm H^{st}_{-*-1}(LBG, k)$, equipped with the  $S^1$-action and the product transferred from the cup product on $\HH_*^{st}(kG, kG)$,   is a BV-algebra (see Corollary \ref{cor-bv-free-loop}).

To the best of the authors' knowledge,  it is still an open question whether there is a BV$_{\infty}$-algebra structure (cf. \cite{Kau, TZ}) on $\mathcal D^*(kG, kG)$.  What we have done in the present article is only the first step toward understanding this higher algebraic structure on $\mathcal D^*(kG, kG)$  and its  behavior in terms of the additive decomposition.  By one of our results, namely that    $\widehat{\mathrm H}^*(G, k)$ is a BV subalgebra of  $\widehat{\HH}^*(kG, kG)$, it is interesting to ask whether the Tate cochain complex $\widehat{C}^*(G, k)$ is a  BV$_{\infty}$-subalgebra of $\mathcal D^*(kG, kG)$ (cf. Remark \ref{BV-infinity-algebra}). 
With this ultimate goal, we could give a partial  result in the present article. We show that  the Tate cochain  complex $\widehat{C}^*(G,k)$ is a cyclic  $A_\infty$-subalgebra of the Tate-Hochschild cochain complex $D^*(kG,kG)$ 
(cf. Theorem \ref{theorem-A-infinity}). In particular, we obtain an isomorphism of cyclic $A_{\infty}$-algebras between $\mathcal D^*(kG, kG)$ and $kG\otimes \widehat{C}^*(G, k)$ when $G$ is a finite abelian group (cf. Corollary \ref{cor-abelian-A}).  
Further problems along this direction will be explored in future research. 

\bigskip

This paper is organized as follows. In Section \ref{section-reminder-Hochschild}, we recall
some notions and algebraic structures on Hochschild (co)homology and  Tate-Hochschild cohomology. In Section \ref{section2}, we study  algebraic structures on the Tate-Hochschild cochain complex $\mathcal D^*(kG, kG)$ for a finite group $G$, including the explicit description of the cyclic  $A_{\infty}$-algebra structure. In Section \ref{section3}, we recall the notions of cohomology and Tate cohomology of finite groups,  including the Tate cochain complex $\widehat{C}^*(G, k)$ (computing the Tate cohomology).
 In Section 4, we lift explicitly the additive decomposition of the Tate-Hochschild cohomology of a group algebra to the complex level.  We also prove that the Tate cochain complex $\widehat{C}^*(G, k)$ is a cyclic $A_{\infty}$-subalgebra of $\mathcal D^*(kG, kG)$. We explain in Section 5 the cup product formula in $\widehat{\HH}^*(kG,kG)$ which appeared in \cite{Nguyen2012} and give a new proof for it  using Green functors, following Bouc. In Section 6, we deal with the BV-operator $\widehat{\Delta}$ and the Lie bracket in $\widehat{\HH}^*(kG,kG)$. In particular, we show that the operator $\widehat{\Delta}$ preserves each summand of the additive decomposition, and that $\widehat{\mathrm H}^*(G, k)$ is indeed a BV subalgebra of $\widehat{\HH}^*(kG,kG)$. In Section 7, we use our results to compute the BV structure of the Tate-Hochschild cohomology for symmetric group of degree 3 over a field $k$ of characteristic $3$. In Appendix A, we provide a proof scattered in literature of the fact that the Connes' $B$-operator is trivial in the group homology $\mathrm H_*(G,k)$.
\bigskip

\textbf{Acknowledgement}
The first named author was partially supported by NCET Program from MOE of China and by NSFC (No.11331006). The second named author was partially supported by NSFC (No.11871071) and he  would like to thank the Schools of Mathematical Sciences at the East China Normal University and the Beijing Normal University for their hospitality during his visit. The third named author was partially supported by NSFC  (No.11671139) and by STCSM (No.13dz2260400).

We are very grateful to the referees for valuable suggestions and comments, which  have  led to substantial changes and significant improvement on the presentation of this paper.

\bigskip

\section{Reminder on Hochschild (co)homology and Tate-Hochschild cohomology}\label{section-reminder-Hochschild}
Throughout this paper, we fix a field $k$.  All group algebras denoted by $kG$ or $kH$,  and all algebras denoted by $A$ in the sequel, and their modules we considers as such, will be assumed to be  finite dimensional. We shall write $\otimes$ for $\otimes_k$, the tensor product over the field $k$,  for two $k$-vector spaces $V$ and $W$, write $\mathrm{Hom}(V, W)$ for $\mathrm{Hom}_k(V, W)$,   the set of $k$-linear maps from $V$ to $W$. 

In this section we recall the definition and algebraic structure of the Tate-Hochschild cohomology of an associative $k$-algebra. For more details, we refer the reader to \cite{Wang2015,LZ2015} and the references therein.

\subsection{Hochschild (co)homology of algebras}\label{subsection-hochschild-cohomology}

Let $A$ be a finite dimensional $k$-algebra. Denote the enveloping algebra $A\otimes_k A^{\op}$ of $A$ by $A^e$.

Let us first recall the construction of the Hochschild (co)chain complexes, using the normalized bar resolution  of $A$.
Denote by $\overline{A}$ the quotient $k$-vector space $A/(k\cdot1)$.
The normalized bar resolution $(\Barr_*(A),d_*)$ of $A$ is a free
resolution of $A$ as $A$-$A$-bimodules, where
$$\Barr_{-1}(A)=A,\quad \text{ and for }n\geq 0, \quad \Barr_n(A)=A\otimes  \overline{A}^{\otimes n}  \otimes A,$$
and the differential is defined as follows:  for $n\geq 1$,
$$d_n: \Barr_n(A)\rightarrow \Barr_{n-1}(A)$$ sends $a_0\otimes  \overline{a_{1, n}} \otimes a_{n+1}$ to
$$a_0a_1\otimes \overline{a_{2, n}}\otimes a_{n+1}+\sum_{i=1}^{n-1}(-1)^ia_0\otimes \overline{a_{1, i-1}} \otimes \overline{a_ia_{i+1}}\otimes \overline{a_{i+1, n}}
  \otimes a_{n+1}+(-1)^n a_0\otimes \overline{a_{1, n-1}}\otimes a_na_{n+1},$$ and for $n=0$, 
$$d_0: \Barr_0(A)=A\otimes A\rightarrow A, \quad a_0\otimes a_1\mapsto a_0a_1.$$ 
Here for simplicity we write $\overline{a_{i, j}}:=\overline{a_{i}}\otimes \overline{a_{i+1}}\otimes \cdots \otimes \overline{a_{j}}  \ (i\leq j), $ and when $n=0$, $  \overline{A}^{\otimes n}:=k$.    This complex  is exact as there exists a contracting homotopy: for $p\geq 0$
$$s_p: \Barr_p(A)\to \Barr_p(A), a_0\otimes \overline{a_{1, p}}\otimes a_{p+1}\mapsto 1\otimes  \overline{a_{0, p}}\otimes a_{p+1}$$ and $s_{-1}: A\to \Barr_0(A), a\mapsto 1\otimes a$. Note that each $s_p\ (p\geq -1)$ is a morphism of right $A$-modules. 


Recall that the \textit{Hochschild cochain complex} $(C^*(A,A),\delta^*)$ is defined as follows:
$$C^n(A,A)=\Hom_{A^e}(\Barr_n(A),A)\simeq  \Hom( \overline{A}^{\otimes n},A), \quad \text{for }n\geq 0,$$
  and the differential is given by
$$\delta^n: C^n(A,A)\rightarrow C^{n+1}(A,A), \quad \varphi\mapsto
\delta^n(\varphi),$$  where $\delta^n(\varphi)$ sends $\overline{a_{1,n+1}}\in \overline{A}^{\otimes (n+1)}$ to
$$a_1\varphi(\overline{a_{2,n+1}})+\sum_{i=1}^{n}(-1)^i\varphi(\overline{a_{1,i-1}}\otimes \overline{a_ia_{i+1}}\otimes \overline{a_{i+2,n+1}})+(-1)^{n+1}\varphi(\overline{a_{1,n}})a_{n+1}.$$ In degree zero, the differential map $\delta^0: A\rightarrow \Hom(\overline{A},A)$ is given by
$$\delta^0(x)(\overline a)=ax-xa  \quad (\mbox{for } x\in A \mbox{ and }
\overline a\in \overline{A}).$$
For any $n\geq 0$, the $n$-th
\textit{Hochschild cohomology group} of $A$ is defined to be the cohomology group
$$\HH^n(A,A)=H^n(C^*(A,A))\simeq \Ext_{A^e}^n(A,A).$$
Recall that the \textit{Hochschild chain complex} $(C_*(A,A),\partial_*)$ is defined as follows:
$$C_n(A,A)=A\otimes_{A^e}\Barr_n(A)\simeq  A\otimes  \overline{A}^{\otimes n}, \quad \text{for }n\geq 0,$$
  and, for $n\geq 2$,  the differential $\partial_n: A\otimes  \overline{A}^{\otimes n}\rightarrow A\otimes  \overline{A}^{\otimes (n-1)}$ sends  
$a_0\otimes\overline{a_{1,n}}$ to $$ a_0a_1\otimes\overline{a_{2,n}}+\sum_{i=1}^{n-1}(-1)^i a_0\otimes\overline{a_{1,i-1}}\otimes\overline{a_ia_{i+1}}\otimes \overline{a_{i+2,n}}+(-1)^{n}a_na_0\otimes \overline{a_{1,n-1}},$$
and in  degree one, the differential   $\partial_1: A\otimes  \overline{A} \rightarrow A$ is given by
$$\partial_1(a_0\otimes \overline{a_1})=a_0a_1-a_1a_0  \quad (\mbox{for } a_0\in A \mbox{ and }
\overline{ a_1}\in \overline{A}).$$
For all $n\geq 0$, the $n$-th
\textit{Hochschild homology group} of $A$ is defined to be the homology group
$$\HH_n(A,A)=H_n(C_*(A,A))\simeq \Tor^{A^e}_n(A,A).$$

Recall that the bounded  derived category $\DD^b(A)$  is the  triangulated category obtained from the homotopy category of bounded complexes of finitely generated $A$-modules by inverting all quasi-isomorphisms.  The Hochschild cohomology groups $\HH^*(A)$ can be realized as $$\HH^n(A, A)=\Hom_{\DD^b(A^e)}(A, A[n]) \quad n\geq 0,$$
where $\DD^b(A^e)$ is the bounded derived category of $A^e$ and $[n]$ denotes the $n$-th shift functor in $\DD^b(A^e)$ (cf. e.g. \cite{Wei}). We end this subsection with a remark.

\begin{Rem} \label{Remark: contructing resolutions from bar resolution}

Let $M$ be a left $A$-module. Then $\Barr_*(A)\otimes_A M$ is a free resolution of $M$. In fact, this complex is exact with the contracting homotopy $\{s_p\otimes_A\id,\  p\geq -1\}$ since $\{s_p, \ p\geq -1\}$ are homomorphisms of right $A$-modules.

The similar result holds for right $A$-modules.

\end{Rem}
\subsection{Tate-Hochschild cohomology}\label{subsection-symmetric}
Let $A$ be a finite dimensional $k$-algebra. 
The singularity category $\DD_{\sg}(A)$ of  $A$ is defined to be the Verdier quotient $$ \DD_{\sg}(A) =\DD^b(A)/\mathrm{per}  A, $$
where
$\mathrm {per}  A$  is the   the bounded homotopy category of finitely generated projective $A$-modules.

Recall that the $i$-th ($i\in\mathbb Z$)  Tate-Hochschild cohomology group $\widehat{\HH}^i(A, A)$ of $A$ is defined as the space $\Hom_{\DD_{\sg}(A^e)}(A, A[i]),$ where $[i]$ denotes the $i$-th shift functor in $\DD_{\sg}(A^e)$. Clearly, the quotient functor from $\DD^b( A^e)$ to $\DD_{\sg}(A^e)$ induces a natural morphism
\begin{equation*}
\rho: \HH^*(A, A)\rightarrow \widehat{\HH}^*(A, A).
\end{equation*}

To compute the Tate-Hochschild cohomology $\widehat{\HH}^*(A, A)$, the second named author   constructed a complex $C_{\sg}^*(A, A)$ (called {\it singular Hochschild cochain complex})  in \cite[Section 3.1]{Wang2015}. Roughly speaking, it is  a colimit of the inductive system consisting of Hochschild cochain complexes with coefficients in the bimodules of noncommutative differential forms. On $C_{\sg}^*(A, A),$ the author  constructed a cup product $\cup$ and a Lie bracket $[\cdot, \cdot]$,   which leads to the following result. 
\begin{Thm}$($\cite[Corollary 5.3]{Wang2015}$)$\label{Theorem-Wang2015}
The Tate-Hochschild cohomology $\widehat{\HH}^*(A, A)$, equipped with the cup product $\cup$ and the Lie bracket $[\cdot, \cdot]$, is a Gerstenhaber algebra. Moreover, the above map $\rho: \HH^*(A, A)\rightarrow \widehat{\HH}^*(A, A)$  is a morphism of Gerstenhaber algebras.
\end{Thm}

In the case of a self-injective algebra $A$ over a field $k$, the Tate-Hochschild cohomology  agrees with the Tate cohomology defined in \cite{Buchweitz1986}.  We have the following descriptions of the Tate-Hochschild cohomology $\widehat{\HH}^*(A, A)$.
\begin{Prop} \label{singularHochschild-selfinjective} $($\cite[Corollary 6.4.1]{Buchweitz1986}$)$ Let A be a self-injective algebra over a field
k.  Denote $\Hom_{A^e}(A,A^e)$ by $A^{\vee}$. Then
\begin{enumerate}[(i)]
\item $\widehat{\HH}^n(A,A)\simeq
\HH^n(A,A)$ for all $n > 0$,
\item $\widehat{\HH}^n(A,A)\simeq
\HH_{-n-1}(A^{\vee},A)$ for all $n < -1$,
\item $\widehat{\HH}^0(A,A)\simeq
 \underline{\Hom}_{A^e}(A,A)$, $\widehat{\HH}^{-1}(A,A)\simeq
 \underline{\Hom}_{A^e}(A,\Omega_{A^e}(A))$, and there is an exact sequence
$$0\rightarrow \widehat{\HH}^{-1}(A,A)\rightarrow  A^{\vee}\otimes_{A^e}A\stackrel{\sigma}{\rightarrow} \Hom_{A^e}(A,A)\rightarrow \widehat{\HH}^0(A,A)\rightarrow 0,$$ where the map $\sigma$ is given by $\sigma(f\otimes a)(a') = f(a')\cdot a$ for $a,a'\in A$ and $f\in A^{\vee}$. Here $\underline{\Hom}_{A^e}(-, -)$ denotes the homomorphism space in the stable category $A^e$-$\underline{\Modu}$ and $\Omega_{A^e}$ is the syzygy functor over $A^e$-$\underline{\Modu}$. 
\end{enumerate}
\end{Prop}

Now we specialize $A$ to be a symmetric algebra. Symmetric algebras are self-injective and include group algebras of finite groups. Recall that a symmetric algebra is a finite dimensional $k$-algebra $A$ such that
there is a symmetric non-degenerate associative bilinear form $\langle\cdot,
\cdot\rangle: A\times A\rightarrow k$ (called the symmetrizing form), or equivalently, $A\simeq
A^*=\mathrm{\Hom}_k(A, k)$ as $A$-$A$-bimodules. Note that we can choose an $A$-$A$-bimodule isomorphism (denoted by $t$) as follows: $t(a)=\langle a,\cdot\rangle$ for $a\in A$.
This isomorphism $t$ induces the following isomorphism
$$t\otimes \id: A\otimes_kA\rightarrow A^*\otimes_kA\simeq \End_k(A)$$
$$\quad \quad \quad \quad a\otimes b \mapsto t(a)\otimes b \mapsto (x\mapsto t(a)(x)b).$$
Following Brou\'{e} (see \cite{Broue2009}), we call the element $(t\otimes \id)^{-1}(\id):=
\sum_ie_i\otimes f_i \in A\otimes_kA$  the Casimir element of $A$. It follows
from \cite[Proposition 3.3]{Broue2009} that  the Casimir element induces an isomorphism $$A\simeq A^{\vee}=\Hom_{A^e}(A,A^e), \quad a \mapsto \sum_i e_ia\otimes f_i$$ as $A$-$A$-bimodules, where we identify $\Hom_{A^e}(A, A^e)$ as 
$$(A\otimes A)^A:=\{\sum_i a_i\otimes b_i \in A\otimes_k A| \sum_{i} aa_i\otimes b_i=\sum_i a_i\otimes b_ia \ \mbox{for any $a\in A$}\}.$$ Hence, if $A$ is a symmetric algebra and $n < -1$, then, by Proposition \ref{singularHochschild-selfinjective} (ii), the Tate-Hochschild cohomology $\widehat{\HH}^n(A,A)$ is isomorphic to the usual Hochschild homology:
$$\widehat{\HH}^n(A,A)\simeq
\Tor_{-n-1}^{A^e}(A^{\vee},A)\simeq
\Tor_{-n-1}^{A^e}(A,A)=\HH_{-n-1}(A,A).$$
Moreover, for $n=-1, 0$ we have the following interesting observations.
\begin{Rem}\label{remark-stable-hochschild}
By Proposition \ref{singularHochschild-selfinjective} (iii), the $0$-th Tate-Hochschild cohomology $\widehat{\HH}^0(A,A)$  is a quotient of the $0$-th Hochschild cohomology $\HH^0(A,A)$ and coincides with the stable center $Z^{st}(A)=Z(A)/Z^{pr}(A)$ (cf. \cite[Section 2]{LiuZhouZimmermann2012}). Moreover, the map $\sigma: A^{\vee}\otimes_{A^e}A\rightarrow  \Hom_{A^e}(A,A)$ in Proposition \ref{singularHochschild-selfinjective} (iii) is identified with the trace map $$\tau: \HH_0(A,A)=A/[A,A]\rightarrow  \HH^0(A,A)=Z(A), \quad  a+[A,A]\mapsto \sum_ie_iaf_i,$$ where $\Ker(\tau)={Z^{pr}(A)}^\perp/[A,A]$ is equal to the so-called $0$-th stable Hochschild homology $\HH_0^{st}(A)$ (cf. \cite[Section 4]{HHKM2005}, \cite[Section 2 and 3]{LiuZhouZimmermann2012}). Thus, in this case, the $-1$-th Tate-Hochschild cohomology $\widehat{\HH}^{-1}(A,A)$ is a subspace of the $0$-th Hochschild homology $\HH_0(A,A)$ and coincides with the $0$-th stable Hochschild homology $\HH^{st}_0(A)$ (cf. \cite[Section 4]{HHKM2005}, \cite[Section 2 and 3]{LiuZhouZimmermann2012}).
\end{Rem}

Therefore,  $\widehat{\HH}^*(A, A)$ is a ``combination" of the Hochschild cohomology $\HH^*(A, A)$ and the Hochschild homology $\HH_*(A, A)$.
We can summarize the above results by means of the following diagram:
$$\xymatrix@R=1.2pc{
& & & & \HH^0\ar@{->>}[d] & \HH^1\ar@<-0.5ex>@{=}[d] & \HH^2\ar@<-0.5ex>@{=}[d]& \cdots\\
\cdots & \widehat{\HH}^{-3}\ar@<-1ex>@{=}[d]& \widehat{\HH}^{-2}\ar@<-1ex>@{=}[d] & \widehat{\HH}^{-1}\ar@<-0.5ex>@{_{(}->}[d] & \widehat{\HH}^{0} & \widehat{\HH}^{1} & \widehat{\HH}^{2} & \cdots \\
\cdots & \HH_{2} & \HH_{1} & \HH_{0}\ar@{->}[uur]_{\tau} & & & &
}$$

In \cite[Section 6.4]{Wang2015}, the author constructed a complex (called {\it Tate-Hochschild cochain complex})
$$\mathcal D^*(A, A):= (\cdots \xrightarrow{\partial_2} C_1(A, A)\xrightarrow{\partial_1} C_0(A, A) \xrightarrow{\tau} C^0(A, A)\xrightarrow{\delta^0} C^1(A, A) \xrightarrow{\delta^1}\cdots),$$
  to compute $\widehat{\HH}^*(A, A)$ for a symmetric algebra $A$, where $\partial_*$ (resp. $\delta^*$) is the  differential of $C_*(A, A)$ (resp.  $C^*(A, A)$) (see Section \ref{subsection-hochschild-cohomology}); and $\tau(x)=\sum_{i} e_ixf_i$. Here $\sum_i(e_i\otimes f_i)$ is the Casimir element. Note that the  bilinear form  $\langle\cdot, \cdot\rangle$ on $A$  defines a non-degenerate bilinear form (still denoted by $\langle\cdot, \cdot\rangle$)
$$\langle\cdot,\cdot\rangle: \mathcal{D}^{*}(A,A)\times \mathcal{D}^{*}(A,A)\rightarrow k$$ on $\mathcal D^*(A, A):$
For $\alpha\in C^m(A, A)$ and $\beta=a_0\otimes \overline{a_{1, n}}  \in C_{n}(A, A),$
we define
$$\langle \beta, \alpha\rangle=\langle \alpha, \beta\rangle:=
\begin{cases}
\langle \alpha(\overline{a_{1, n}}), a_{0}\rangle & \mbox{if $m=n$},\\
0 & \mbox{otherwise}.
\end{cases}
$$

\begin{Rem}\label{duality-in-Tate-Hochschild-cohomology} In fact,  this bilinear form $\langle\cdot,\cdot\rangle$ is induced by  the duality between $C_*(A, A)$ and $C^*(A, A)$ defined in  \cite[Lemma 2.9]{KoenigLiuZhou}. 
Note that  $\langle\cdot, \cdot\rangle$ descends to $\widehat{\HH}^*(A, A)$ since it is compatible with the differential of $\mathcal D^*(A, A)$ (cf. Lemma \ref{lemma-pairing-compatibility} below).    
In particular, we have a non-degenerate
bilinear form between $\widehat{\HH}^0(A,A)\simeq
 Z^{st}(A)$ and $\widehat{\HH}^{-1}(A,A)\simeq
 \HH_0^{st}(A)$ (cf. \cite[Theorem 2.15 (3)]{KoenigLiuZhou}).
\end{Rem}

The following result shows that $\mathcal D^*(A, A)$ has  a  rich algebraic structure. 
\begin{Thm}\label{theorem-A-infinity-algebra}$($\cite[Theorem 6.3 and Proposition 6.5]{RiveraWang2017}$)$ Let $A$ be a symmetric $k$-algebra. Then there is a cyclic $A_{\infty}$-algebra structure $(m_1=\partial, m_2, m_3, \cdots)$   and an $L_{\infty}$-algebra  structure $(l_1=\partial, l_2, l_3, \cdots)$ on $\mathcal D^*(A, A)$ such that $m_i=0$ for $i>3$, where the restrictions of $m_2$ and $l_2$ to the nonnegative part $\mathcal D^{\geq 0}(A, A)$ are respectively,   the usual cup product and Gerstenhaber bracket on $C^*(A, A)$. \end{Thm}
\begin{Rem}
We have simple and explicit formulas for the $A_{\infty}$-products since $m_i=0$ for $i>3$. But the formulas for the $L_{\infty}$-brackets $l_i$  are in general very complicated and messy.  In Section \ref{section2}, we write down the explicit formulas for the $A_{\infty}$-products $m_i$ on $\mathcal D^*(kG, kG)$.  In Theorem \ref{theorem-A-infinity} below, we prove that   the Tate cochain complex  $\widehat{C}^*(G, k)$  is a cyclic $A_{\infty}$-subalgebra of $\mathcal D^*(kG, kG)$. \end{Rem}

Recall that the Connes' $B$-operator on the Hochschild chain complex $C_*(A, A)$ is defined as
\begin{equation*}
B(a_0\otimes  \overline{a}_{1,m})=\sum^{m}_{i=0} (-1)^{mi}1\otimes  \overline{a_{i, m}}\otimes  \overline{a_0}\otimes  \overline{a_{1, i-1}}.
\end{equation*}
Tradler in \cite{Tradler2008} and Menichi \cite{Menichi2004} showed  that the Hochschild cohomology $\HH^*(A,A)$ of a symmetric algebra $A$ is a  BV-algebra whose BV-operator $\Delta$ is the dual of the Connes' $B$-operator with respect to the bilinear form $\langle\cdot,\cdot\rangle$. That is,
$$\langle \Delta(f)( \overline{a_{1, m}}), a_0\rangle=\langle B(a_0\otimes  \overline{a_{1, m}}), f\rangle.$$

Generalizing the above  result, we have the following result. 
\begin{Thm}\label{Theorem-BV-algebra}$($\cite[Theorem 6.17]{Wang2015}\cite[Corollary 6.7]{RiveraWang2017}$)$
Let $A$ be a symmetric $k$-algebra.  Then the Gerstenhaber algebra $(\widehat{\HH}^*(A, A), \cup, [\cdot, \cdot])$ is a BV-algebra whose BV-operator $\widehat{\Delta}$ is given by
\begin{equation*}
\widehat{\Delta}^i:=
\begin{cases}
\Delta^i  &  \mbox{for $i> 0$,}\\
0 & \mbox{for $i=0$,}\\
B_{-i-1} & \mbox{for $i\leq -1$.}
\end{cases}
\end{equation*}\end{Thm}

\section{Tate-Hochschild
cohomology of a group algebra}\label{section2}

Let $k$ be a field, $G$ a finite group and $kG$ the group algebra. Recall that  $kG$ is a symmetric algebra with the symmetrizing form:
$$\langle g, h\rangle=1\ \mbox{if $gh=1$ and}\  \langle g, h\rangle=0 \  \mbox{otherwise}$$
for  all $g,h\in G$. In particular, $\sum_{g\in G}g^{-1}\otimes g$ is a Casimir element of $kG$.  Thus from Section \ref{subsection-symmetric}, we have that the  Tate-Hochschild cohomology  $\widehat{\HH}^*(kG,kG)$  is a ``combination" of the Hochschild cohomology $\HH^*(kG,kG)$ and the Hochschild homology $\HH_*(kG,kG)$.


The Hochschild (co)chain complexes of $kG$ have the following simple descriptions.  For a set $X$, we denote by $k[X]$ the $k$-vector space spanned by the elements in $X$. In particular, we have $kG=k[G]$. Note that  $\overline{kG}$  can be identified with the $k$-vector space $k[\overline{G}]$, where $\overline{G}=G-\{1\}$.
When $n=0$, $\overline{G}^{\times n}$ denotes a one-point set and $k[\overline{G}^{\times n}]:=k$.  For simplicity, we write $(g_{1, n})$ for $(g_1, g_2, \cdots, g_n)\in G^{\times n}$.

The normalized bar resolution $(\Barr_*(kG),d_*)$ of $kG$  has the form  (throughout we just write all the maps on the base elements)
$$\Barr_{-1}(kG)=kG, \text{ and for }n\geq 0, \quad \Barr_n(kG)=k[G\times\overline{G}^{\times n}\times G],$$
$$d_0: \Barr_0(kG)=k[G\times G]\rightarrow kG, \quad (g_0,  g_1)\mapsto g_0g_1,\ \   \text{ and for }n\geq 1,$$
$$d_n: \Barr_n(kG)\rightarrow \Barr_{n-1}(kG), \ \  (g_0,\overline{g_{1, n}},  g_{n+1})\mapsto
\sum_{i=0}^{n}(-1)^i(g_0,  \overline{g_1}, \cdots, \overline{g_ig_{i+1}},
\cdots, \overline{g_n}, g_{n+1}).$$   Here $k[G\times\overline{G}^{\times n}\times G]$ denotes the $k$-vector space spanned by the elements in the Cartesian product $G\times\overline{G}^{\times n}\times G.$
We always use the normalized bar resolution (except in Appendix A) since it greatly simplifies the computations. From now on,  we just write $g$ for its image $\overline{g}$ in $\overline{G}$.

  Recall that the Hochschild cochain complex $(C^*(kG,kG),\delta^*)$ is defined as follows:
$$C^n(kG,kG)=\Hom_{(kG)^e}(\Barr_n(kG),kG)\simeq  \Hom_k(k[\overline{G}^{\times n}],kG)\simeq \Map(\overline{G}^{\times n},kG), \quad \text{for }n\geq 0,$$
where $\Map(\overline{G}^{\times n},kG)$ denotes the set of maps from $\overline{G}^{\times n}$ to $kG$, and the differential is given by
$$\delta^n: \Map(\overline{G}^{\times n},kG)\rightarrow \Map(\overline{G}^{\times (n+1)},kG), \quad \varphi\mapsto
\delta^n(\varphi),$$  where $\delta^n(\varphi)$ sends $g_{1,n+1}\in \overline{G}^{(n+1)}$ to
$$g_1\varphi(g_{2,n+1})+\sum_{i=1}^{n}(-1)^i\varphi(g_{1,i-1},g_ig_{i+1},g_{i+2,n+1})+(-1)^{n+1}\varphi(g_{1,n})g_{n+1}.$$ In degree zero, the differential map $\delta^0: kG\rightarrow \Map(\overline{G},kG)$ is given by
$$\delta^0(x)(g)=gx-xg  \quad (\mbox{for } x\in kG \mbox{ and }
g\in \overline{G}).$$

Recall that the Hochschild chain complex $(C_*(kG,kG),\partial_*)$ is defined as follows:
$$C_n(kG,kG)=kG\otimes_{(kG)^e}\Barr_n(kG)\simeq  k[G\times\overline{G}^{\times n}], \quad \text{for }n\geq 0,$$
where $k[G\times\overline{G}^{\times n}]$ denotes the $k$-vector space spanned by the elements in $G\times\overline{G}^{\times n}$, and the differential is given by
$$\partial_n: k[G\times\overline{G}^{\times n}]\rightarrow k[G\times\overline{G}^{\times {(n-1)}}],$$  $$(g_0,g_{1,n})\mapsto (g_0g_1,g_{2,n})+\sum_{i=1}^{n-1}(-1)^i(g_0,g_{1,i-1},g_ig_{i+1},g_{i+2,n})+(-1)^{n}(g_ng_0,g_{1,n-1}).$$ In degree one, the differential map $\partial_1: k[G\times \overline{G}]\rightarrow kG$ is given by
$$\partial_1(g_0,g_1)=g_0g_1-g_1g_0  \quad (\mbox{for } g_0\in G \mbox{ and }
g_1\in \overline{G}).$$

From Section \ref{subsection-symmetric}, the Tate-Hochschild cohomology $\widehat{\HH}^*(kG,kG)$ can be computed by the following Tate-Hochschild cochain complex $\mathcal{D}^*(kG,kG)$:
\begin{equation}\label{equation-tate-complex}
\cdots\stackrel{\partial_2}{\rightarrow}k[G\times\overline{G}]\stackrel{\partial_1}{\rightarrow}kG\stackrel{\tau}{\rightarrow}kG\stackrel{\delta_0}{\rightarrow} \Map(\overline{G},kG)\stackrel{\delta_1}{\rightarrow}\cdots,
\end{equation}
where the differential $\tau$ (from degree $-1$ component to degree $0$ component) is defined to be the trace map $x\mapsto \sum_{g\in G}gxg^{-1}$. Notice that $\sum_{g\in G}g\otimes g^{-1}$ is a Casimir element of  $kG$.

Since we have an algebra isomorphism
$(kG)^e\simeq k(G\times G)$ given by $g_1\otimes g_2\mapsto
(g_1,g_2^{-1})$, we can identify
each $kG$-$kG$-bimodule
 $M$ as a left $k(G\times G)$-module by the action $(g_1,g_2)\cdot x=g_1xg_2^{-1}$, or as a right $k(G\times G)$-module by the action $x\cdot(g_1,g_2)=g_2^{-1}xg_1$. In the following, we shall view the
bar resolution, the Hochschild co(chain) complexes for the group algebra $kG$ in terms of
$k(G\times G)$-modules. Consequently, $$\HH^n(kG,kG)\simeq \Ext_{k(G\times G)}^n(kG,kG),$$
where the $k(G\times G)$-module structure on both $kG$ by the action $(g_1,g_2)\cdot x=g_1xg_2^{-1}$, and $$\HH_n(kG,kG)\simeq \Tor^{k(G\times G)}_n(kG,kG),$$
where the first $kG$ has a right $k(G\times G)$-module structure by the action $x\cdot (g_1,g_2)=g_2^{-1}xg_1$, and the second $kG$ has a left $k(G\times G)$-module structure by the action $(g_1,g_2)\cdot x=g_1xg_2^{-1}$.

\bigskip

Now we recall from \cite{Wang2015} the (generalized) cup product on $\mathcal D^*(kG, kG)$. 
\begin{Def}\label{Def-cup-product}
Let $\alpha\in \mathcal{D}^n(kG,kG)$ and $\beta\in \mathcal{D}^m(kG,kG)$. Then the {\it (generalized) cup product} $\alpha\cup\beta$ is defined  by the following six cases:

{\it Case 1.} $n\geq 0, m\geq 0$. Then $\alpha\in C^n(kG,kG)$,
$\beta\in C^m(kG,kG)$, and the cup product $\alpha\cup \beta\in C^{n+m}(kG,kG)=\mathcal{D}^{n+m}(kG,kG)$ is
the same as the usual cup product on $C^*(kG, kG)$:
$$\alpha\cup \beta: \overline{G}^{\times n+m}\rightarrow kG,  \quad g_{1,n+m}\mapsto \alpha(g_{1,n})\beta(g_{n+1,n+m}).$$

{\it Case 2.} $n\leq -1, m\leq -1$. Then $\alpha=(g_0,g_{1,s})\in C_s(kG,kG)$ with $s=-n-1\geq 0$,
$\beta=(h_0,h_{1,t})\in C_t(kG,kG)$ with $t=-m-1\geq 0$, and the cup product $\alpha\cup \beta\in C_{s+t+1}(kG,kG)=\mathcal{D}^{n+m}(kG,kG)$ is
defined by
$$\alpha\cup \beta=\sum_{g\in G}(gh_0,h_{1,t},g^{-1}g_0,g_{1,s})\in k[G\times\overline{G}^{\times s+t+1}].$$
This product in $C_*(kG, kG)$ is originally defined in \cite[Theorem 6.1]{Abb} inspired from string topology.

{\it Case 3.} $n\geq 0, m\leq -1$ and $n+m\leq -1$. Then $\alpha\in C^n(kG,kG)$,
$\beta=(h_0,h_{1,t})\in C_t(kG,kG)$ with $t=-m-1\geq 0$, and the cup product $\alpha\cup \beta\in C_{t-n}(kG,kG)=\mathcal{D}^{n+m}(kG,kG)$ is
the same as the usual cap product $\cap$ (which induces an action of Hochschild cohomology on Hochschild homology):
$$\alpha\cup \beta=(\alpha(h_{t-n+1,t})h_0,h_{1,t-n})\in k[G\times\overline{G}^{\times t-n}].$$

{\it Case 4.} $n\geq 0, m\leq -1$ and $n+m\geq 0$. Then $\alpha\in C^n(kG,kG)$,
$\beta=(g_0,g_{1,t})\in C_t(kG,kG)$ with $t=-m-1\geq 0$, and the cup product $\alpha\cup \beta\in C^{n-t-1}(kG,kG)=\mathcal{D}^{n+m}(kG,kG)$ is
defined as the following generalized cap product:
$$\alpha\cup \beta: \overline{G}^{\times n-t-1}\rightarrow kG,  \quad h_{1,n-t-1}\mapsto \sum_{g\in G}\alpha(h_{1,n-t-1},g^{-1},g_{1,t})g_0g.$$

{\it Case 5.} $n\leq -1, m\geq 0$ and $n+m\leq -1$. Then $\alpha=(g_0,g_{1,s})\in C_s(kG,kG)$ with $s=-n-1\geq 0$,
$\beta\in C^m(kG,kG)$, and the cup product $\alpha\cup \beta\in C_{s-m}(kG,kG)=\mathcal{D}^{n+m}(kG,kG)$ is
the following cap product $\cap$ from the right side:
$$\alpha\cup \beta=(g_0\beta(g_{1,m}),g_{m+1,s})\in k[G\times\overline{G}^{\times s-m}].$$

{\it Case 6.} $n\leq -1, m\geq 0$ and $n+m\geq 0$. Then $\alpha=(g_0,g_{1,s})\in C_s(kG,kG)$ with $s=-n-1\geq 0$,
$\beta\in C^m(kG,kG)$, and the cup product $\alpha\cup \beta\in C^{m-s-1}(kG,kG)=\mathcal{D}^{n+m}(kG,kG)$ is
defined as the following generalized cap product from the right side:
$$\alpha\cup \beta: \overline{G}^{\times m-s-1}\rightarrow kG,  \quad h_{1,m-s-1}\mapsto \sum_{g\in G}gg_0\beta(g_{1,s},g^{-1},h_{1,m-s-1}).$$
\end{Def}

\begin{Rem}\label{remark-signs-change}
Since the definition of the  cup product $\cup$ in this paper  is different  from that in \cite{Wang2015}, in order   to make   the following identity still hold in $\mathcal D^*(kG, kG)$ (cf. Lemma \ref{lemma-pairing-compatibility}), $$\partial(\alpha\cup \beta)=\partial(\alpha)\cup \beta+(-1)^{m} \alpha\cup \partial(\beta), \quad \mbox{for $\alpha\in \mathcal D^m(kG, kG)$ and $\beta\in \mathcal D^n(kG, kG),$}$$ we have to change the signs of the differential in the negative part $\mathcal D^{<0}(kG, kG)$. That is, the new differential $\partial'$ on $\mathcal D^*(kG, kG)$ is given as follows: 
$$\partial'_m(\alpha)=\begin{cases}
(-1)^{m}\partial_m(\alpha) & \mbox{for $\alpha\in \mathcal D^{m}(kG, kG)$ and $m< 0$,}\\
 \delta(\alpha)  & \mbox{for  $\alpha\in \mathcal D^{m}(kG, kG)$ and  $m\geq 0.$}
\end{cases}$$
By Section \ref{subsection-symmetric}, there is a non-degenerate bilinear form on $\mathcal D^*(kG, kG)$ (induced by the symmetrizing form $\langle\cdot, \cdot\rangle$ on kG)
$$\langle\cdot,\cdot\rangle: \mathcal{D}^{*}(kG,kG)\times \mathcal{D}^{*}(kG, kG)\rightarrow k$$
For $\alpha\in C^m(kG, kG)$ and $\beta=(g_0, g_{1, n}) \in C_{n}(kG, kG),$
we define
$$\langle \beta, \alpha\rangle=\langle \alpha, \beta\rangle:=
\begin{cases}
\langle \alpha(g_{1, n}), g_{0}\rangle & \mbox{if $m=n$},\\
0 & \mbox{otherwise}.
\end{cases}
$$
 As usual, we call an element $\alpha\in \mathcal D^n(kG,kG)$ homogeneous of degree $n$, and its degree will be denoted by $|\alpha|$. In particular, $|\alpha|=-m-1$ for $\alpha\in C_m(kG, kG)=\mathcal D^{-m-1}(kG, kG)$. \end{Rem}
\begin{Lem}\label{lemma-pairing-compatibility}
The following identities hold in the complex $(\mathcal D^*(kG, kG), \partial')$
$$\langle \partial'(\alpha), \beta\rangle =(-1)^{|\alpha|+1} \langle \alpha, \partial'(\beta)\rangle, \quad
\langle \alpha\cup \beta, \gamma\rangle =\langle \alpha, \beta\cup \gamma\rangle,$$
$$\partial'(\alpha\cup \beta)=\partial'(\alpha)\cup\beta+(-1)^{|\alpha|} \alpha\cup \partial'(\beta)$$
for homogeneous elements  $\alpha, \beta, \gamma\in\mathcal D^*(kG, kG).$
\end{Lem}
\begin{proof}
The first equality follows from a straightforward computation. Let us verify the second identity. We have the following two cases.
\begin{enumerate}[(i)]
\item For $\phi\in C^m(kG, kG), \psi\in C^n(kG, kG)$ and $\alpha:=(g_0, g_{1, m+n})\in C_{m+n}(kG, kG)$, we have
\begin{equation*}
\begin{split}
\langle \phi\cup \psi, \alpha\rangle=\langle g_0,  \phi(g_{1, m})\psi(g_{m+1, m+n})\rangle=\langle \phi, (\psi(g_{m+1, m+n})g_0, g_{1, m})\rangle=\langle \phi, \psi\cup \alpha\rangle\\
\langle \phi\cup \psi, \alpha\rangle=\langle g_0,  \phi(g_{1, m})\psi(g_{m+1, m+n})\rangle=\langle (g_0\phi(g_{1, m}), g_{m+1, m+n}), \psi\rangle=\langle\alpha\cup\phi,\psi\rangle.
\end{split}
\end{equation*}
This implies that $\langle \phi\cup \psi, \alpha\rangle=\langle \phi, \psi\cup \alpha\rangle=\langle\alpha\cup\phi,\psi\rangle$.
\item For $\alpha=(g_0, g_{1, r})\in C_r(kG, kG), \beta=(h_0, h_{1, t})\in C_t(kG, kG)$ and $\phi\in C^{r+t+1}(kG, kG)$, we have
\begin{equation*}
\begin{split}
\langle \alpha\cup \beta, \phi \rangle=\sum_{g\in G} \langle gh_0, \phi(h_{1, t}, g^{-1}g_0, g_{1, r})\rangle
=\sum_{g\in G}\langle (g_0, g_{1, r}), gh_0\phi(h_{1, t}, g^{-1}, ?)\rangle=\langle \alpha, \beta\cup \phi\rangle\\
\langle \alpha\cup \beta, \phi \rangle=\sum_{g\in G} \langle gh_0, \phi(h_{1, t}, g^{-1}g_0, g_{1, r})\rangle
=\sum_{g\in G}\langle (h_0, h_{1, t}), \phi(?, g^{-1}g_0, g_{1, r})g\rangle=\langle \beta,  \phi\cup \alpha\rangle
\end{split}
\end{equation*}
where we need to  use the identity  $\sum_{g\in G} (gg_0, g^{-1})=\sum_{g\in G}(g, g_0g^{-1})$ in $k[G\times \overline{G}]$ for $g_0\in G$. 
\end{enumerate}
This verifies the second identity. From the first two identities,  to  verify the third identity,   it is sufficient to consider the following two cases.
\begin{enumerate}[(i)]
\item For $\phi, \psi\in C^*(kG, kG)$, it is well-known that
$$\delta(\phi\cup \psi)=\delta(\phi)\cup \psi+(-1)^{|\phi|}\phi\cup \delta(\psi).
$$
In this case, the third identity holds since $\partial'=\delta$ for $C^*(kG, kG)$.
\item   For $\alpha=(g_0, g_{1, s})\in C_{s}(kG, kG)$ and $\beta=(h_0, h_{1, t})\in C_t(kG, kG)$, we have
\begin{equation*}
\begin{split}
\partial(\alpha\cup \beta)=&\sum_{g\in G} \partial((gh_0, h_{1, t}, g^{-1}g_0, g_{1, s}))\\
=&\alpha\cup\partial(\beta)+(-1)^{|\beta|}\partial(\alpha)\cup \beta.
\end{split}
\end{equation*}
Thus $\partial'(\alpha\cup \beta)=(-1)^{|\alpha|}\alpha\cup \partial'(\beta)+\partial'(\alpha)\cup \beta$ since $\partial'(\alpha)=(-1)^{|\alpha|}\partial(\alpha)$.
\end{enumerate}
This proves the lemma. \end{proof}

As a consequence, the (generalized)  cup product $\cup$ on $\mathcal D^*(kG, kG)$ induces a graded-commutative associative product (still denoted by $\cup$) over $\widehat{\HH}^*(kG,kG),$ which coincides with the Yoneda product in the singularity category $\DD_{\sg}((kG)^e) $ (cf. \cite{Wang2015}).
We remind that, contrary to the Hochschild cochain complex case, the above cup product $\cup$ on the Tate-Hochschild cochain complex is not associative, but it is associative up to homotopy (cf. \cite{RiveraWang2017}). From \cite[Theorem 6.3]{RiveraWang2017}, it follows that the cup product extends to an $A_{\infty}$-algebra  structure $(m_1, m_2, m_3, \cdots)$ on $\mathcal{D}^*(kG, kG)$ with $m_1=\partial', m_2=\cup$ and $m_i=0$ for $i>3$ (cf. Theorem \ref{theorem-A-infinity-algebra}). The formula for $m_3$ is described as follows.
\begin{enumerate}[(i)]
\item If either $\phi, \varphi, \psi\in C^*(kG, kG)$ or $\phi, \varphi, \psi\in C_*(kG, kG)$, 
then $m_3(\phi, \varphi, \psi)=0$.
\item If $\alpha, \beta \in C_*(kG, kG)$ and $\phi \in C^{*}(kG, kG)$, then $m_3(\alpha, \beta, \phi)=0=m_3(\phi, \alpha, \beta)$.
\item If $\alpha \in C_{*}(kG, kG)$ and $\phi, \varphi \in C^{*}(kG,kG)$, then $m_3(\phi, \varphi,\alpha)=0=m_3(\alpha, \phi, \varphi)$.
\item For $\phi\in C^{m}(kG, kG), \varphi\in C^{n}(kG, kG)$ and $\alpha=(g_0, g_1, \cdots, g_r) \in C_{r}(kG, kG)$,
\begin{itemize}
\item if $r+2\leq m+n$, then $m_3(\phi,\alpha,\varphi) \in C^{m-r+n-2}(kG, kG)$ is defined by
\begin{eqnarray*}
\lefteqn{
m_3(\phi, \alpha, \varphi)(h_1, \cdots, h_{m-r+n-2})=}\\
&&\sum_{g\in G}\sum_{j=1}^{\min\{n, r+1\}}(-1)^{m+r+j-1} \phi(h_{1, m-r+j-2},  g, g_{j, r})g_0
\varphi(g_{1, j-1}, g^{-1}, h_{m-r+j-1, m-r+n-2}), \\
\end{eqnarray*}
\item if $r+2>m+n$, then $m_3(\phi,\alpha,\varphi)=0.$
\end{itemize}
\item For $\alpha=(g_0, g_{1, r}) \in C_{r}(kG, kG), \beta=(h_0, h_{1, s})\in C_{s}(kG, kG)$ and $\phi\in C^{m}(kG, kG),$ \begin{itemize}
\item if  $m-1\leq r+s$, then
\begin{eqnarray*}
m_3(\alpha,\phi, \beta)=\sum_{g\in G}\sum_{j=0}^s(-1)^{n-j} (g_{0}\phi(g_{1, m-s+j-1}, g, h_{j+1, s}) h_0, h_{1, j}, g^{-1}, g_{m-s+j, r}),
\end{eqnarray*}
\item if $m-1>r+s$, then $m_3(\alpha,\phi, \beta)=0$.
\end{itemize}
\end{enumerate}

Therefore we  have the following identity, for $\alpha_1, \alpha_2, \alpha_3\in \mathcal D^*(kG, kG)$,
\begin{eqnarray*}
\alpha_1\cup(\alpha_2\cup \alpha_3)-(\alpha_1\cup \alpha_2)\cup \alpha_3\lefteqn{=\partial'(m_3(\alpha_1, \alpha_2, \alpha_3))+m_3(\partial'(\alpha_1), \alpha_2, \alpha_3)}\\
&+(-1)^{|\alpha_1|} m_3(\alpha_1, \partial'(\alpha_2), \alpha_3)+(-1)^{|\alpha_1|+|\alpha_2|}m_3(\alpha_1, \alpha_2, \partial'(\alpha_3)).
\end{eqnarray*}
From \cite[Proposition 6.5]{RiveraWang2017},  it follows that the $A_{\infty}$-algebra structure is compatible with  the non-degenerate  bilinear form  $\langle\cdot, \cdot\rangle$ in the following sense:
\begin{equation}\label{equation-cyclic-pairing}\langle \alpha_0, m_k(\alpha_1, \cdots, \alpha_k)\rangle=(-1)^{|\alpha_0|(2-k)+k}\langle m_k(\alpha_0, \cdots, \alpha_{k-1}), \alpha _k\rangle\end{equation}
for any $\alpha_i\in \mathcal D^*(kG, kG),\ 0\leq  i\leq  k$. Such $A_{\infty}$-algebra is called {\it cyclic}. 
In particular, Formula (iv) is dual to Formula (v) in the sense of Equation (\ref{equation-cyclic-pairing}). In other words, we may get one from the other by Equation (\ref{equation-cyclic-pairing}).





Moreover, one can define a Lie bracket $[\cdot,\cdot]$ on $\widehat{\HH}^*(kG,kG)$ such that $(\widehat{\HH}^*(kG,kG), \cup, [\cdot,\cdot])$ becomes a Gerstenhaber algebra,
that is,  for homogeneous elements $\alpha,\beta,\gamma$ in
$\widehat{\HH}^*(kG,kG)$, the following three conditions hold:
\begin{itemize}
\item  $(\widehat{\HH}^*(kG,kG), \cup)$ is an associative algebra and it is graded
commutative, that is, the cup product $\cup$ is an associative
multiplication and satisfies $\alpha\cup
\beta=(-1)^{|\alpha||\beta|}\beta\cup \alpha$;

\item  $(\widehat{\HH}^*(kG,kG), [\cdot,\cdot])$ is a graded
Lie algebra of degree $-1$, that is, the Lie bracket $[\cdot,\cdot]$ satisfies $[\alpha,\beta]=-(-1)^{(|\alpha|-1)(|\beta|-1)}[\beta,\alpha]$ and the graded Jacobi identity $$(-1)^{(|\alpha|-1)(|\gamma|-1)}[[\alpha,\beta],\gamma]+(-1)^{(|\beta|-1)(|\alpha|-1)}[[\beta,\gamma],\alpha]+(-1)^{(|\gamma|-1)(|\beta|-1)}[[\gamma,\alpha],\beta]=0;$$

\item Poisson rule: $[\alpha\cup \beta, \gamma]=[\alpha, \gamma]\cup \beta + (-1)^{|\alpha|(|\gamma|-1)}\alpha\cup[\beta,
\gamma]$.
\end{itemize}

The Lie bracket $[\cdot,\cdot]$ is a generalization of the Gerstenhaber bracket $[\cdot,\cdot]$ in Hochschild cohomology, and we can write it down explicitly  at the complex level.   Since $[\cdot,\cdot]$ is determined by  the cup product $\cup$ and the BV-operator $\widehat{\Delta}$ in $\widehat{\HH}^*(kG,kG)$ (see below), we refrain from giving a formula at the complex level. The interested reader can refer the details to the paper \cite{Wang2015}.

\bigskip

Let us now define the BV-operator $\widehat{\Delta}$ in $\widehat{\HH}^*(kG,kG)$. At the complex level, $\widehat{\Delta}$ is the Connes' $B$-operator $B$ for the negative part $\mathcal{D}^{<0}(kG,kG)=C_{*}(kG,kG)$, is the operator $\Delta$ for the positive part $\mathcal{D}^{>0}(kG,kG)=C^{>0}(kG,kG)$, and $\widehat{\Delta}: \mathcal{D}^{0}(kG,kG)\rightarrow \mathcal{D}^{-1}(kG,kG)$ is zero. More precisely, if $n\leq -1$, then $\widehat{\Delta}=B: k[G\times\overline{G}^{\times s}]\rightarrow k[G\times\overline{G}^{\times {s+1}}]$ (let $s=-n-1$) is given by
$$\alpha=(g_0,g_{1,s})\mapsto \widehat{\Delta}(\alpha)=\sum_{i=0}^s(-1)^{is}(1,g_{i,s},g_{0,i-1});$$
if $n\geq 1$, then $\widehat{\Delta}=\Delta: \Map(\overline{G}^{\times n},kG)\rightarrow \Map(\overline{G}^{\times (n-1)},kG)$ maps any $\alpha: \overline{G}^{\times n}\rightarrow kG$ to $\Delta(\alpha):  \overline{G}^{\times (n-1)}\rightarrow kG$ such that
$$\Delta(\alpha)(g_{1,n-1})=\sum_{g_n\in G}\sum_{i=1}^n(-1)^{i(n-1)}\langle \alpha(g_{i,n-1}, g_n,
 g_{1,i-1}),1\rangle g_n^{-1},$$
 where $\langle\cdot, \cdot\rangle$ is the symmetrizing form on $kG$ (cf. \cite[Section 8]{LZ2015}). 
It is easy to verify that the operator
$\widehat{\Delta}: \mathcal{D}^{*}(kG,kG)\rightarrow \mathcal{D}^{*-1}(kG,kG)$ is a chain map (with $\widehat{\Delta}^2=0$) and therefore induces an operation
(still denoted by $\widehat{\Delta}$) in  $\widehat{\HH}^*(kG,kG)$. It turns out that the Gerstenhaber algebra
$(\widehat{\HH}^*(kG,kG), \cup, [\cdot, \cdot])$ together with the operator $\widehat{\Delta}$
is a Batalin-Vilkovsky algebra (BV-algebra), that is, in
addition to be a Gerstenhaber algebra, $(\widehat{\HH}^*(kG,kG),\widehat{\Delta})$ is
a complex and
$$[\alpha,\beta]=-(-1)^{(|\alpha|-1)|\beta|}(\widehat{\Delta}(\alpha\cup\beta)-\widehat{\Delta}(\alpha)\cup\beta-(-1)^{|\alpha|}\alpha\cup\widehat{\Delta}(\beta))$$
for all homogeneous elements $\alpha, \beta\in \widehat{\HH}^*(kG,kG)$ (cf. \cite{Wang2015}).

\begin{Rem}\label{sign-convention} The signs in the
definition of a BV-algebra depend on the choice of the definitions
of cup product and Lie bracket. If we define $\alpha{\cup}
'\beta=(-1)^{|\alpha||\beta|}\alpha\cup\beta$ and
${\Delta}'(\alpha)=(-1)^{(|\alpha|-1)}\Delta(\alpha)$,
then we get
$$[\alpha,\beta]=(-1)^{|\alpha|}({\Delta}'(\alpha{\cup}'\beta)-
{\Delta}'(\alpha){\cup}'\beta-(-1)^{|\alpha|}\alpha{\cup}'{\Delta}'(\beta)),$$
which is the equality in the usual definition of a BV-algebra (see,
for example \cite{Getzler1994,Menichi2004}). We choose the sign
convention from \cite{Tradler2008} because of our convention of the
definitions of cup product and Connes' $B$-operator in the
Hochschild (co)homology theory.
\end{Rem}


\section{Reminder on cohomology and Tate cohomology of finite groups}\label{section3}

In this section we recall some notions on Tate  cohomology of finite groups. For the details, we refer the reader to \cite[Chapter VI]{Brown}.

\subsection{Group (co)homology}\label{subsection: remainder on group cohomology}

Let $k$ be a field, $G$ a finite group and $kG$ the group algebra. Let $M$ be a left $kG$-module. Then the cohomology of $G$ with coefficients in $M$ is defined to be
$$\Hu^p(G, M)=\Ext^p_{kG}(k, M), \ p\geq 0,$$
and the homology of $G$ with coefficients in $M$ is defined to be
$$\Hu_p(G, M)=\Tor_p^{kG}(k, M),\  p\geq 0,$$ where $k$ is  the left trivial $kG$-module in $\Ext^p_{kG}(k, M)$ and is the  right trivial $kG$-module in $\Tor_p^{kG}(k, M)$. 

By Remark \ref{Remark: contructing resolutions from bar resolution}, the complex $P_*:=\Barr_*(kG)\otimes_{kG}
k$ is the \textit{standard   resolution} of the trivial $kG$-module $k$.  So there exist canonical complexes computing group (co)homology.

Recall that the \textit{group cohomology  complex} $(C^*(G, M),\delta^*)$ is defined as follows:
$$C^n(G, M)=\Hom_{kG}(\Barr_n(kG)\otimes_{kG}
k, M)\simeq  \Hom_{kG}(k[\overline{G}^{\times n}], M)\simeq \Map(\overline{G}^{\times n}, MG), \quad \text{for }n\geq 0,$$
  and the differential is given by
$$\delta^n: \Map(\overline{G}^{\times n}, M)\rightarrow \Map(\overline{G}^{\times (n+1)}, M), \quad \varphi\mapsto
\delta^n(\varphi),$$  where $\delta^n(\varphi)$ sends $g_{1,n+1}\in \overline{G}^{n+1}$ to
$$g_1\varphi(g_{2,n+1})+\sum_{i=1}^{n}(-1)^i\varphi(g_{1,i-1},g_ig_{i+1},g_{i+2,n+1})+(-1)^{n+1}\varphi(g_{1,n}).$$ In degree zero, the differential map $\delta^0: M\rightarrow \Map(\overline{G},M)$ is given by
$$\delta^0(x)(g)=gx-x  \quad (\mbox{for } x\in M \mbox{ and }
g\in \overline{G}).$$

We can consider  $M$ as a right $kG$-module via $x\cdot g=g^{-1}x, x\in M, g\in G.$ Then $\Tor_*^{kG}(k, M)\cong \Tor_*^{kG}(M, k),$ where we use the right $kG$-module  $M$  in $\Tor_*^{kG}(M, k)$. 
Notice  that  $\Tor_*^{kG}(M, k)$ can by computed by  the \textit{group homology   complex} $(C_*(G, M),\partial_*)$, which is   defined as follows:
$$C_n(G,M)=M\otimes_{kG}\Barr_n(kG)\otimes_{kG} k \simeq  M\otimes k[\overline{G}^{\times n}], \quad \text{for }n\geq 0,$$
 and the differential $\partial_n: M\otimes k[\overline{G}^{\times n}]\rightarrow M\otimes k[\overline{G}^{\times (n-1)}], n\geq 2$ is given by
  $$x\otimes g_{1,n} \mapsto x\cdot g_1\otimes (g_{2,n})+\sum_{i=1}^{n-1}(-1)^i x\otimes (g_{1,i-1},g_ig_{i+1},g_{i+2,n})+(-1)^{n} x\otimes (g_{1,n-1}),$$  and in degree one, the differential map $\partial_1: M\otimes k[\overline{G}]\rightarrow M$ is given by
$$\partial_1(x\otimes g_1)=x\cdot g_1-x  \quad (\mbox{for } x\in M \mbox{ and }
g_1\in \overline{G}). $$

\bigskip

Let us explain  conjugation maps, restriction maps and corestriction maps on group (co)homology, as we shall need them in Section \ref{section-cup-product}.  For more details, we refer the reader to the textbook \cite{Evens}. 

 Let $G$ be  a finite group and $M$ a $kG$-module. Let $Q_*$ be a projective resolution of $k$ as a $kG$-module.
\begin{itemize}
\item[(1)] For any $g\in G$ and $H\leq G$, write ${}^g\!H=gHg^{-1}$.   The conjugation map $g^*: \mathrm H^*(H,M)\longrightarrow \mathrm H^*({}^g\!H,M)$    is induced by the map
     $$g^*: \Hom_{kH}(Q_*, M)\to \Hom_{kH}(Q_*, M),\quad \varphi \mapsto {}^g\!\varphi,$$
     where  ${}^g\!\varphi(x)=g\varphi(g^{-1}x), x\in Q_*;$

\item[(2)] For $H\leq G$,  the restriction map $res^G_H: \mathrm H^*(G, M)\longrightarrow \mathrm H^*(H, M)$  is induced by the natural inclusion map
     $\Hom_{kG}(Q_*, M)\to \Hom_{kH}(Q_*, M)$, as homomorphisms of $kG$-modules    are necessarily homomorphisms of $kH$-modules;

\item[(3)] For  $H\leq G$ the corestriction map $cor^G_{H}: \mathrm H^*(H, M)\longrightarrow \mathrm H^*(G,M)$ is induced from the map
$$\Hom_{kH}(Q_*, M)\to \Hom_{kG}(Q_*, M),\quad \varphi\mapsto \sum_{t\in T} {}^t\!\varphi,$$
 where $T$ is a complete set of representatives of the left cosets of the subgroup $H$ in $G$;

 \item[(1')] For any $g\in G$,  the conjugation map $g_*: \mathrm H_*(H,M)\longrightarrow \mathrm H_*({}^g\!H,M)$    is induced by the map
     $$g_*:  M\otimes_{kH}Q_* \to M\otimes_{k {}^g\!H}Q_*, \quad x\otimes_{kH} y\mapsto xg^{-1}\otimes_{k{}^g\!H} gy;$$

\item[(2')] For $H\leq G$,  the restriction map $res^G_H: \mathrm H_*(G, M)\longrightarrow \mathrm H_*(H, M)$ is     is induced by the map
     $$M\otimes_{kG}Q_* \to M\otimes_{k H}Q_*,  \quad x\otimes_{kG} y \mapsto \sum_{t\in T} xt\otimes_{kH}t^{-1}x,$$
     where $T$ is a complete set of representatives of the left cosets of the subgroup $H$ in $G$;

\item[(3')] For  $H\leq G$,  the corestriction map $cor^G_{H}: \mathrm H_*(H, M)\longrightarrow \mathrm H_*(G,M)$ is induced from the natural quotient map
$M\otimes_{kH}Q_* \to M\otimes_{k G}Q_*.$
\end{itemize}

\subsection{ Tate cohomology of groups}
\label{subsection-Tate-cohomology}

Applying the duality functor $( )^*=\Hom_k(-,k)$ to the standard resolution $P_*=\Barr_*(kG)\otimes_{kG} k$,  we get a ``backwards projective resolution" $\Hom(\Barr_*(kG)\otimes_{kG}
k, k)$ of $k$. By splicing together $\Barr_*(kG)\otimes_{kG}
k$ and $\Hom(\Barr_*(kG)\otimes_{kG}
k), k)$ we get a complete resolution of the trivial module $k$:
$$\xymatrix@R=1.5px{
\cdots\ar[r]& P_2\ar[r] & P_1\ar[r] & P_0\ar[rr]\ar@{->>}[dr] &  & {P_0}^*\ar[r] & {P_1}^* \ar[r]& \cdots, \\
& & & & k\ar@{^{(}->}[ur] & & &
}$$
where  the (left) $kG$-module structure over $P_n^*=\Hom_k(P_n,k)$ is given by $(g\varphi)(x)=\varphi(g^{-1}x)$ for any $\varphi\in P_n^*$ and $x\in P_n$.
We denote this complete resolution by $F_*$, where $F_n=P_n$ for $n\geq 0$ and $F_n=P_{-n-1}^*$ for $n\leq -1$. Let $U$ be any (left) $kG$-module. Applying the functor $\Hom_{kG}(-,U)$ to $F_*$, we get a cochain complex, denoted by $\widehat{C}^*(G, U)$ (called {\it Tate cochain complex}  of $G$): 
\begin{enumerate}
\item The nonnegative part $\widehat{C}^{\geq 0}(G, U)$ is exactly the group cohomology complex $C^*(G, U)$  with coefficients in $U$.

\item For each $n\leq -1$ (let $s=-n-1\geq 0$), we notice that there is a natural isomorphism
$$U\otimes_{kG}P_s\simeq \Hom_{kG}(P_s^*,U),\quad u\otimes_{kG} x\mapsto (\alpha\mapsto \sum_{g\in G} \alpha(gx)gu),$$
 where in $U\otimes_{kG} P_s$ we consider $U$ as a right $kG$-module by the action $ug=g^{-1}u$, and where $P_s^*$ is viewed as a left $kG$-module.
 It follows that, for $n\leq -1$ (let $s=-n-1\geq 0$), $$\widehat{C}^n(G, U) =  \Hom_{kG}(P_s^*,U)\simeq U\otimes_{kG}P_s \simeq U\otimes_{kG} (kG\otimes k(\overline{G}^{\times s}))\simeq
U\otimes k[\overline{G}^{\times s}]$$  and the differential
is given by
$$\partial_s(x, g_{1,s})=(g_1^{-1}x, g_{2,s})+\sum_{i=1}^{s-1}(-1)^i(x, g_{1,i-1},g_ig_{i+1},g_{i+2,s})+(-1)^{s}(x, g_{1,s-1})$$
for all $s\geq 1$, $x\in U$ and $g_1,\cdots, g_s\in \overline{G}$.
Therefore the negative part $C^{<0}(G, U)$ is exactly the group homology complex $C_*(G, U)$  with coefficients in $U$ (Here we consider $U$ as a right $kG$-module).

\item For $n=-1$ (or $s=0$), the differential $\delta_{-1}: C_0(G, U)=U\rightarrow U=C^0(G, U)$
is given by $u\mapsto (\sum_{g\in G}g)u$ for $u\in U$. 
\end{enumerate}

 The Tate cohomology of $G$ with coefficients in $U$ is defined to be the (co)homology group $$\widehat{\mathrm H}^n(G,U)=H^n(\widehat{C}^*(G, U))=H^n(\Hom_{kG}(F_*,U)).$$
We have the following descriptions of the Tate cohomology $\widehat{\mathrm H}^n(G,U)$:
\begin{enumerate}[(i)]
\item  $\widehat{\mathrm H}^n(G,U)\simeq
\mathrm H^n(G,U):= \Ext^n_{kG}(k,U)$ for all $n > 0$,

\item $\widehat{\mathrm H}^n(G,U)\simeq \mathrm H_{-n-1}(G,U):=
\Tor_{-n-1}^{kG}(U, k)$ for all $n < -1$,

\item there is an exact sequence
$$0\rightarrow \widehat{\mathrm H}^{-1}(G,U)\rightarrow \mathrm H_{0}(G,U) \stackrel{\alpha}\rightarrow \mathrm  H^0(G,U)\rightarrow \widehat{\mathrm H}^{0}(G,U)\rightarrow 0.$$ 
Denote the sum $\sum_{g\in G}g$ by $N$. Then the map $\alpha$ is the so-called norm map: $$\mathrm H_{0}(G,U)=U_G\rightarrow U^G=\mathrm H^0(G,U),\quad u\mapsto Nu.$$
\end{enumerate}
Therefore,  $\widehat{\mathrm H}^*(G,U)$ is a ``combination" of the group cohomology $\mathrm H^*(G,U)$ and the group homology $\mathrm H_*(G,U)$. We can summarize the above results by means of the following diagram (cf. \cite[VI. 4]{Brown}):
$$\xymatrix@R=1.3pc{
& & & & \mathrm H^0\ar@<-0.3ex>@{->>}[d] & \mathrm H^1\ar@<-0.5ex>@{=}[d] & \mathrm H^2\ar@<-0.5ex>@{=}[d] & \cdots\\
\cdots & \widehat{\mathrm H}^{-3}\ar@{=}[d]& \widehat{\mathrm H}^{-2}\ar@{=}[d] & \widehat{\mathrm H}^{-1}\ar@{_{(}->}[d] & \widehat{\mathrm H}^0 & \widehat{\mathrm H}^1 & \widehat{\mathrm H}^2 & \cdots \\
\cdots & \mathrm H_{2} & \mathrm H_{1} & \mathrm H_{0}\ar@{->}[uur]_{\alpha} & & & &
}$$

\medskip

Of particular interest to us is the case when $U=k$, the trivial $kG$-module. From now on, we always refer to Tate cohomology of a group algebra $kG$ as $\widehat{\mathrm H}^*(G,k)$, unless stated otherwise.

\begin{Rem}\label{-1-and-0-Tate-group-cohomology} If the characteristic of $k$ divides the order of $G$, then the map $\alpha: \mathrm H_{0}(G,k) \rightarrow  \mathrm H^0(G,k)$ is zero and we have that $\widehat{\mathrm H}^{-1}(G,k)\simeq \mathrm H_{0}(G,k)$ and $\widehat{\mathrm H}^{0}(G,k)\simeq \mathrm H^{0}(G,k)$. Otherwise,  the map $\alpha: \mathrm H_{0}(G,k) \rightarrow  \mathrm H^0(G,k)$ is an isomorphism and we have that $\widehat{\mathrm H}^p(G, k)=0$ for all $p\in \mathbb Z$. 
\end{Rem}

\section{Lifting  the additive decomposition to the complex level}\label{section4}

Let $k$ be a field and $G$ a finite group. Then the Tate-Hochschild cohomology $\widehat{\HH}^*(kG,kG)$ admits an additive
decomposition (as $k$-vector spaces):
$$\widehat{\HH}^*(kG,kG)\simeq \bigoplus_{x\in X}\widehat{\mathrm H}^*(C_G(x),k),$$
where $X$ is a set of representatives of conjugacy classes of elements of $G$ and $C_G(x)$ is the centralizer subgroup of $x\in G$, and where $\widehat{\mathrm H}^*(C_G(x),k)$ is the  Tate  cohomology of $C_G(x)$ (cf. \cite[Section 5]{Nguyen2012}). In this section, we give an explicit construction of the additive decomposition at the complex level. We deal with this task in three cases:

\medskip\begin{enumerate}[]
\item {\it The first case:} $n>0$. In this case, $$\widehat{\HH}^n(kG,kG)\simeq \HH^n(kG,kG)\simeq \bigoplus_{x\in X}\mathrm H^n(C_G(x),k)= \bigoplus_{x\in X}\Ext^n_{kC_G(x)}(k,k);$$
\item {\it The second case:} $n<-1$ (let $s=-n-1>0$). In this case, $$\widehat{\HH}^n(kG,kG)\simeq \HH_s(kG,kG)\simeq \bigoplus_{x\in X}\mathrm H_s(C_G(x),k)= \bigoplus_{x\in X}\Tor_s^{kC_G(x)}(k,k);$$
\item {\it The third case:} $n = 0, -1$. In this case, $$\widehat{\HH}^0(kG,kG)\simeq \bigoplus_{x\in X}\widehat{\mathrm H}^0(C_G(x),k), \quad \widehat{\HH}^{-1}(kG,kG)\simeq \bigoplus_{x\in X}\widehat{\mathrm H}^{-1}(C_G(x),k).$$
\end{enumerate}

\bigskip
{\it The first case.} This is just the usual additive decomposition of the Hochschild cohomology $\HH^*(kG,kG)$ (except in degree zero component) and its lifting (to the complex level) has been done by the first named author and the third named author in \cite{LZ2015}. The idea is as follows. Cibils and Solotar \cite{CS1997} constructed a subcomplex of the
Hochschild cochain complex $C^*(kG, kG)$ for each conjugacy class, and then
they showed that for a finite abelian group, the subcomplex is isomorphic to the group cohomology complex (cf. Section \ref{subsection: remainder on group cohomology}). This was generalized to any finite group in \cite{LZ2015}: for each conjugacy class, this complex computes the group cohomology of the corresponding centralizer
subgroup. Let us briefly recall the construction there. For simplicity,  we denote by $\mathcal{H}_*$ and $\mathcal{H}^*$ the Hochschild chain complex $C_*(kG,kG)$ and the Hochschild cochain complex $C^*(kG,kG)$ respectively.

Recall from Section \ref{section2} that
the Hochschild cohomology $\HH^*(kG,kG)$ of the group algebra $kG$ can
be computed by the Hochschild cochain  complex $C^*(kG, kG)$:
$$(\mathcal{H}^*)\quad \quad \quad \quad 0\longrightarrow kG\stackrel{\delta^0}{\longrightarrow} \Map(\overline{G},kG)\stackrel{\delta^1}{\longrightarrow} \cdots \longrightarrow \Map(\overline{G}^{\times n},kG)
\stackrel{\delta^n}{\longrightarrow} \cdots,$$ where the differential is
given by $$\delta^0(x)(g)=gx-xg  \quad (\mbox{for } x\in kG \mbox{ and }
g\in \overline{G})$$ and (for $\varphi: \overline{G}^{\times
n}\longrightarrow kG$ and $g_1,\cdots, g_{n+1}\in \overline{G}$)
$$\quad  \quad \quad \delta^n(\varphi)(g_{1,n+1})=g_1\varphi(g_{2,n+1})+\sum_{i=1}^{n}(-1)^i\varphi(g_{1,i-1},g_ig_{i+1},g_{i+2,n+1})+(-1)^{n+1}\varphi(g_{1,n})g_{n+1}.$$
Let $X$ be a complete set of
representatives of the conjugacy classes in the finite group $G$.
For $x\in X$, $C_x=\{gxg^{-1}|g\in G\}$ is the conjugacy class
corresponding to $x$ and $C_G(x)=\{g\in G|gxg^{-1}=x\}$ is the
centralizer subgroup. Now take a conjugacy class $C_x$ and define
$$\mathcal{H}_x^0=k[C_x],  \mbox{ and for }n\geq 1,$$
$$\mathcal{H}_x^n=\{\varphi: \overline{G}^{\times
n}\longrightarrow kG\ | \ \varphi(g_1,\cdots, g_{n})\in k[g_1\cdots
g_{n}C_x]\subset kG, \forall g_1,\cdots, g_n\in \overline{G}\},$$
where $g_1\cdots g_{n}C_x$ denotes the subset of $G$ obtained  by multiplying
$g_1\cdots g_{n}$ on $C_x$ and $k[g_1\cdots g_{n}C_x]$ is the
$k$-subspace of $kG$ spanned by this set. Note that we have $g_1\cdots g_{n}C_x=C_xg_1\cdots g_{n}$
 and $k[g_1\cdots g_{n}C_x]=k[C_xg_1\cdots g_{n}]$.
Let
$\mathcal{H}_x^*=\bigoplus_{n\geq 0}\mathcal{H}_x^n$. Cibils and
Solotar observed that
$\mathcal{H}_x^*$ is a subcomplex of $\mathcal{H}^*$ and
$\mathcal{H}^*=\bigoplus_{x\in X}\mathcal{H}_x^*$ (see \cite[Page 20, Proof of the theorem]{CS1997}).

\begin{Lem}\label{isomorphism-Hochschild-cohomology}  The complex $\mathcal{H}_x^*$ is  isomorphic to the complex  $$\Hom_{kC_G(x)}(\Barr_*(kG)\otimes_{kG}k,k),$$ which computes the
group cohomology $\mathrm H^*(C_G(x), k)$ of $C_G(x)$. More concretely, there is an isomorphism of complexes:
$$
\begin{array}{rcl}
\mathcal{H}_x^* &\rightarrow&
\Hom_{kC_G(x)}(\Barr_*(kG)\otimes_{kG}k,k),\\
(\varphi_x: \overline{G}^{\times n}\rightarrow kG)&\mapsto &(\widehat{\varphi}_x:
S_x\times \overline{G}^{\times n}\rightarrow k), \quad
\widehat{\varphi}_x(\gamma_{i, x}, g_{1,n})= a_{i, x}
\end{array}
$$
where we write $\varphi_x(g_{1, n})g_n^{-1}\cdots g_1^{-1}=\sum_{i=1}^{n_x} a_{i, x} x_i\in kC_x$; we fix a right coset decomposition of $C_G(x)$ in $G$: $$G=C_G(x)\gamma_{1, x}\cup \cdots \cup C_G(x)\gamma_{n_x, x}$$ and thus $C_x=\{\gamma_{1, x}^{-1}x\gamma_{1, x}, \cdots, \gamma_{n_x, x}^{-1}x\gamma_{n_x, x}\}$. Here  we write $x_i=\gamma_{i, x}^{-1}x\gamma_{i, x}$ and  $S_x=\{\gamma_{1, x}, \cdots, \gamma_{n_x, x}\}, $ and we take $\gamma_{1, x}=1$ and $x_1=x$.  The inverse is given by
$$
\begin{array}{rcl}
\Hom_{kC_G(x)}(\Barr_*(kG)\otimes_{kG}k,k) &\rightarrow&
\mathcal{H}_x^*,\\
(\widehat{\varphi}_x: S_x\times\overline{G}^{\times n}\rightarrow k)&\mapsto &(\varphi_x:
\overline{G}^{\times n}\rightarrow kG)), \
\varphi_x(g_{1,n})= \sum_{i=1}^{n_x}\widehat{\varphi}_x(\gamma_{i,x},g_{1, n})x_ig_1\cdots g_n.
\end{array}
$$
 Passing to the cohomology, we have
$H^*(\mathcal{H}_x^*)\simeq \mathrm H^*(C_G(x),k)$.
\end{Lem}
\begin{proof}
This follows from the first five steps in \cite[Page 9]{LZ2015}.
\end{proof}

To compare the two complexes $\Hom_{kC_G(x)}(\Barr_*(kG)\otimes_{kG}k,k)$ and  $$C^*(C_G(x), k)=\Hom_{kC_G(x)}(\Barr_*(kC_G(x))\otimes_{kC_G(x)} k, k),$$ we need the following comparison maps defined in \cite[Page 16]{LZ2015}. The comparison map
$$\iota: \Barr_*(kC_G(x))\otimes_{kC_G(x)}k\rightarrow \Barr_*(kG)\otimes_{kG}k$$ is just defined as  the inclusion map
$\iota: k[C_G(x)\times \overline{C_G(x)}^{\times
n}]\hookrightarrow k[G\times \overline{G}^{\times n}].$
The comparison map $$\rho: \Barr_*(kG)\otimes_{kG}k\longrightarrow
\Barr_*(kC_G(x))\otimes_{kC_G(x)}k$$ is defined as follows:
$$
\begin{array}{rcl}
\rho_{-1}: k&\longrightarrow& k, \quad 1\longmapsto 1,\\
\rho_{0}: kG &\longrightarrow& kC_G(x), \quad h\gamma_{i,x}\longmapsto h,
\mbox{ for  }h\in  C_G(x),\\
\rho_{1}: k[G\times \overline{G}]&\longrightarrow& k[C_G(x)\times \overline{C_G(x)}], \quad (h\gamma_{i,x}, g_1)\longmapsto
(h, h_{i_1}),
\end{array}
$$
$$\mbox{ where $h\in C_{G}(x)$ and 
} \gamma_{i,x}g_1=h_{i_1}\gamma_{s^1_i,x}\mbox{ for }h_{i_1}\in
\overline{C_G(x)},$$
$$\rho_{n}: k[G\times \overline{G}^{\times n}]\longrightarrow k[C_G(x)\times
 \overline{C_G(x)}^{\times n}], \quad (h\gamma_{i,x}, g_{1, n})\longmapsto (h, h_{i_1}, \cdots,
h_{i_n}),$$ where $ h_{i_1},\cdots, h_{i_n}\in \overline{C_G(x)}$
 are determined by the sequence $\{g_1,\cdots, g_n\}$ as follows:
$$\gamma_{i,x}g_1=h_{i_1}\gamma_{s_i^1,x}, \quad \gamma_{s_i^1,x}g_2=h_{i_2}
\gamma_{s_i^2,x}, \quad \cdots,
 \quad \gamma_{s_i^{n-1},x}g_n=h_{i_n}\gamma_{s_i^n,x}.$$ 

\begin{Rem}\label{remark-homotopy-deformation}
Notice that $\rho\circ \iota=\id$. There is a homotopy $s: \Barr_*(kG)\otimes_{kG} k\rightarrow \Barr_*(kG)\otimes_{kG} k$ between $\id$ and $\iota\circ \rho$. For $(h\gamma_{i, x}, g_{1, n}) \in k[G\times \overline{G}^{\times n}]$, we define 
$$s(h\gamma_{i, x}, g_{1, n}):=(h, \gamma_{i, x}, g_{1, n})+\sum_{j=1}^n(-1)^j (h, h_{i_1}, \cdots, h_{i_j}, \gamma_{s_{i}^j, x}, g_{j+1}, \cdots, g_n),$$
 where $ h_{i_1},\cdots, h_{i_n}\in \overline{C_G(x)}$
 are determined by the sequence $\{g_1,\cdots, g_n\}$ as follows:
$$\gamma_{i,x}g_1=h_{i_1}\gamma_{s_i^1,x}, \quad \gamma_{s_i^1,x}g_2=h_{i_2}
\gamma_{s_i^2,x}, \quad \cdots,
 \quad \gamma_{s_i^{n-1},x}g_n=h_{i_n}\gamma_{s_i^n,x}.$$  By a straightforward computation, we get that $\id-\iota\circ \rho=(d\otimes_{kG}k)\circ s+s\circ (d\otimes_{kG}k),$ where  $d$ is the differential of $\Barr_*(kG)$ (cf. Section \ref{section2}). As a consequence, we get a homotopy deformation retract of complexes of (left) $kC_G(x)$-modules
 \begin{equation*}
 \xymatrix@C=0.000000001pc{
\ar@(lu,dl)_-{s}&\mathcal \Barr_*(kG)\otimes_{kG} k \ar@<0.5ex>[rrrrr]^-{\rho}&&&&&\bigoplus\limits_{x\in X} \Barr_*(kC_G(x))\otimes_{kC_G(x)}k. \ar@<0.5ex>[lllll]^-{\iota}}
\end{equation*}
That is, we have $\rho\circ \iota=\id$ and $ \id-\iota\circ\rho=(d\otimes_{kG}k)\circ s+s\circ (d\otimes_{kG}k).$ We remark that our construction of the homotopy $s$ is inspired from \cite[Definition 3.4]{HLL18}. 
\end{Rem} 

Applying the functor $\Hom_{kC_G(x)}(-, k)$ to the above homotopy deformation retract in Remark \ref{remark-homotopy-deformation} and then composing with the isomorphism in Lemma \ref{isomorphism-Hochschild-cohomology}, we obtain the following homotopy deformation retract of complexes  for any $x\in X$, \begin{equation*}\xymatrix@C=0.000000001pc{
\ar@(lu,dl)_-{s^x}&\mathcal H^*_x \ar@<0.5ex>[rrrrr]^-{\iota^x}&&&&& C^*(C_G(x), k).\ar@<0.5ex>[lllll]^-{\rho^x}}
\end{equation*}
Here the surjection $\iota^x$ is given by 
 \begin{equation*}
 \begin{split}
\iota^x: \mathcal{H}_x^* \rightarrow C^*(C_G(x), k), \quad 
[\varphi_x: \overline{G}^{\times n}\rightarrow kG]\mapsto [\widehat{\varphi}_x:
\overline{C_G(x)}^{\times n}\rightarrow k],
\end{split}
\end{equation*}
$$\mbox{with} \quad \widehat{\varphi}_x(h_{1,n})=a_{1,x}, \mbox{ where }\
\varphi_x(h_{1,n})h_{n}^{-1}\cdots
h_1^{-1}=\sum_{i=1}^{n_x}a_{i,x}x_i\in kC_x.$$ In other words,
 $\widehat{\varphi}_x(h_{1,n})$ is just the coefficient of
$x$ in $\varphi_x(h_{1,n})h_{n}^{-1}\cdots h_1^{-1}\in
kC_x$. 
The map $\rho^x$ is given   by
 \begin{equation*}
 \begin{split}
 \rho^x:  C^*(C_G(x), k)&\rightarrow \mathcal H_x^*,\quad 
[\widehat{\varphi}_x:
\overline{C_G(x)}^{\times n}\rightarrow k]\mapsto
[\varphi_x: \overline{G}^{\times n}\rightarrow kG], \end{split}
\end{equation*}
$$\mbox{with} \quad \varphi_x\in\mathcal{H}_x^n,\quad \mbox{and} \quad 
\varphi_x(g_{1,n})=\sum_{i=1}^{n_x}\widehat{\varphi}_x(h_{i_1},\cdots,
h_{i_n})x_ig_1\cdots g_n,$$
where $h_{i_1},\cdots, h_{i_n}\in \overline{C_G(x)}$  are determined by the sequence  $\{g_1,\cdots,
g_n\}$ as follows:
$$\gamma_{i,x}g_1=h_{i_1}\gamma_{s_i^1,x}, \quad \gamma_{s_i^1,x}g_2=h_{i_2}\gamma_{s_i^2,x}, \quad \cdots, \quad
 \gamma_{s_i^{n-1},x}g_n=h_{i_n}\gamma_{s_i^n,x}.$$
 The homotopy $s^x$ is given by: For $(\varphi_x: \overline{G}^{\times n} \rightarrow kG)\in \mathcal H_x^n$, we define $s^x(\varphi_x) \in \mathcal H_x^{n-1}$ as 
$$s^x(\varphi_x)(g_{1, n-1})=\sum_{j=0}^{n-1}\sum_{i=1}^{n_x} (-1)^ja^{1}_{i, j} x_i g_1\cdots g_{n-1},$$ 
where the coefficients $a_{i, j}^1$ are determined by the following identity (when $j=0$, we set $\gamma_{s_i^0, x}=\gamma_{i, x}$)
$$\varphi_x(h_{i_1}, \cdots, h_{i_j}, \gamma_{s_i^j, x}, g_{j+1}, \cdots,g_{n-1})g_{n-1}^{-1}\cdots g_1^{-1}\gamma_{i, x}^{-1}=\sum_{k=1}^{n_x} a_{i, j}^k x_k$$
since we have $h_{i_1}h_{i_2}\cdots h_{i_j}\gamma_{s_i^j, x}g_{j+1}\cdots g_{n-1}=\gamma_{i, x}g_1\cdots g_{n-1}$ for any $0\leq  j\leq n-1$.

Therefore,   we get a lifting of the additive decomposition of $\HH^*(kG,kG)$ at the complex level. 
 \begin{Thm}\label{realization-Hochschild-cohomology}  $($$\mathrm{cf}.$ \cite[Theorem 6.3]{LZ2015}$)$  Let $k$ be a field and $G$ a finite group. Consider the additive
decomposition of Hochschild cohomology algebra of the group algebra
$kG$:
$$\HH^*(kG,kG)\simeq \bigoplus_{x\in X}\mathrm H^*(C_G(x),k)$$
where $X$ is a set of representatives of conjugacy classes of elements of $G$ and $C_G(x)$ is the centralizer of $G$.
Then the above additive decomposition  lifts to a homotopy deformation retract of complexes 
\begin{equation*}\xymatrix@C=0.000000001pc{
\ar@(lu,dl)_-{s^*}&C^*(kG, kG) \ar@<0.5ex>[rrrrr]^-{\iota^*}&&&&&\bigoplus\limits_{x\in X}  C^*(C_G(x), k),\ar@<0.5ex>[lllll]^-{\rho^*}}
\end{equation*}
where $\iota^*=\sum_{x\in X} \iota^x,\ \rho^*=\sum_{x\in X} \rho^x$, and $s^*=\sum_{x\in X}s^x$. 
\end{Thm}

Notice that the homotopy $s^*$ in the above theorem is induced from the homotopy $s$ in Remark \ref{remark-homotopy-deformation}, which is not contained in \cite[Theorem 6.3]{LZ2015}.  

\bigskip
{\it The second case.} This is just the usual additive decomposition of the Hochschild homology $\HH_*(kG,kG)$ (except in degree zero component). We use a similar idea as in the first case to give a lifting of $\HH_*(kG,kG)$ to the complex level.

Recall from Section \ref{section2} that the Hochschild homology $\HH_*(kG,kG)$ of the group algebra $kG$ can
be computed by the Hochschild chain complex $C_*(kG, kG)$:
$$(\mathcal{H}_*)\quad \quad \quad \quad \cdots \longrightarrow k[G\times\overline{G}^{\times s}]\stackrel{\partial_s}{\longrightarrow} \cdots \longrightarrow k[G\times\overline{G}]\stackrel{\partial_1}{\longrightarrow} kG\longrightarrow 0,$$ where the differential is
given by $$\partial_s(g_0,g_{1,s})=(g_0g_1,g_{2,s})+\sum_{i=1}^{s-1}(-1)^i(g_0,g_{1,i-1},g_ig_{i+1},g_{i+2,s})+(-1)^{s}(g_sg_0,g_{1,s-1}),$$
$$\partial_1(g_0,g_1)=g_0g_1-g_1g_0  \quad (\mbox{for } g_0\in G \mbox{ and }
g_1\in \overline{G}).$$
We denote
$$\mathcal{H}_{x,0}=k[C_x],  \mbox{ and for }s\geq 1,$$
$$\mathcal{H}_{x,s}=k[(g_s^{-1}\cdots g_1^{-1}u,g_{1,s}) \ | \ u\in C_x, g_1,\cdots, g_s\in \overline{G}].$$
Let
$\mathcal{H}_{x,*}=\bigoplus_{s\geq 0}\mathcal{H}_{x,s}$. It is easy to verify that $\mathcal{H}_{x,*}$ is a subcomplex of $\mathcal{H}_*$ and $\mathcal{H}_{*}=\bigoplus_{x\in X}\mathcal{H}_{x,*}$.
\begin{Rem}
We obtain this decomposition of $C_*(kG, kG)=\mathcal H_*$ motivated from Cibils-Solatar's decomposition of $\mathcal H^*$, but this decomposition has already appeared in  \cite[7.4.4 Proposition]{Lod}. In fact, the complex $k[\Gamma_*(G, x)]$ in \cite[7.4.4 Proposition]{Lod},  which is constructed as certain  subcyclic set of the cyclic bar construction,   coincides with  $\mathcal H_{x, *}$. We thank an anonymous referee for pointing out this to us.
\end{Rem}

\begin{Lem}\label{isomorphism-Hochschild-homology} The complex $\mathcal{H}_{x,*}$ is  isomorphic to  $k\otimes_{kC_G(x)}\Barr_*(kG)\otimes_{kG}k$, which computes the group homology $\mathrm H_*(C_G(x), k)$ of $C_G(x)$. More concretely, there is an isomorphism of complexes:
$$
\begin{array}{rcl}\mathcal{H}_{x,s}&\longrightarrow&
k\otimes_{kC_G(x)}\Barr_s(kG)\otimes_{kG}k\simeq k\otimes_{kC_G(x)}k[G\times \overline{G}^{\times s}],\\
(g_s^{-1}\cdots g_1^{-1}g_0^{-1}xg_0,g_{1,s})&\longmapsto& 1\otimes_{kC_G(x)}g_{0,s},
\end{array}
$$
 and its inverse is given by
 $$
\begin{array}{rcl}
k\otimes_{kC_G(x)}k[G\times  \overline{G}^{\times s}]&\longrightarrow&
\mathcal{H}_{x,s},\\
1\otimes_{kC_G(x)}g_{0,s}&\longmapsto& (g_s^{-1}\cdots g_1^{-1}g_0^{-1}xg_0,g_{1,s}).
\end{array}
$$
\end{Lem}

\begin{proof} The differential in the complex $\mathcal{H}_{x,*}$ is induced from $\mathcal{H}_*$, while the differential in the complex $k\otimes_{kC_G(x)}k[G\times \overline{G}^{\times s}]$ is given by $$1\otimes_{kC_G(x)}g_{0,s}\longmapsto 1\otimes_{kC_G(x)}((g_0g_1,g_{2,s})+\sum_{i=1}^{s-1}(-1)^i(g_0,g_{1,i-1},g_ig_{i+1},g_{i+2,s})+(-1)^{s}g_{0,s-1}).$$
It is straightforward to check that the given maps commute with the above differentials. Passing to the homology, we have
$H_*(\mathcal{H}_{x,*})\simeq \mathrm H_*(C_G(x),k)= \Tor_*^{kC_G(x)}(k,
k)$ since $\Barr_*(kG)\otimes_{kG}k$ is a projective resolution of $k$ as $kC_{G}(x)$-modules.
\end{proof}

 Applying the functor $k\otimes_{kC_G(x)}-$ to the homotopy deformation retract in Remark \ref{remark-homotopy-deformation} and then composing with the isomorphism in Lemma \ref{isomorphism-Hochschild-homology}, we obtain the following homotopy deformation retract for any $x\in X$, 
 \begin{equation*}\xymatrix@C=0.000000001pc{
\ar@(lu,dl)_-{s_x}&\mathcal H_{x, *} \ar@<0.5ex>[rrrrr]^-{\rho_x}&&&&& C_*(C_G(x), k).\ar@<0.5ex>[lllll]^-{\iota_x}}
\end{equation*}
Here the injection $\iota_x$ is given by 
\begin{equation*} 
 \begin{split}
 \iota_x:  C_*(C_G(x),k)\stackrel{\sim}{\longrightarrow} &\mathcal H_{x, *},\\
[\widehat{\alpha}_x=(h_1, \cdots, h_n)\in k[\overline{C_G(x)}^{\times n}]\longmapsto &[\alpha_x=(h_n^{-1}\cdots h_1^{-1}x,h_{1,n})\in\mathcal{H}_{x,n}].
\end{split}
\end{equation*}
and the surjection $\rho_x$ is given by 
\begin{equation*}
 \begin{split}
 \rho_x: \mathcal H_{x, *} & \rightarrow C_*(C_G(x), k)\\
[\alpha_x=(g_n^{-1}\cdots g_1^{-1}g_0xg_0^{-1},g_{1,n})\in\mathcal{H}_{x,n}]&\longmapsto  [\widehat{\alpha}_x=(h_{i_1}, \cdots, h_{i_n})\in k[\overline{C_G(x)}^{\times n}]]\end{split}
\end{equation*}
where $h_{i_1},\cdots, h_{i_n}\in \overline{C_G(x)}$ are determined by the following sequence:
$$g_0=h\gamma_{i, x}, \quad \gamma_{i,x}g_1=h_{i_1}\gamma_{s_i^1,x}, \quad \gamma_{s_i^1,x}g_2=h_{i_2}\gamma_{s_i^2,x}, \quad \cdots, \quad
 \gamma_{s_i^{n-1},x}g_n=h_{i_n}\gamma_{s_i^n,x}.$$
 The homotopy $s_x$ is given as follows: For $\alpha_x=(g_n^{-1}\cdots g_1^{-1}g_0^{-1}xg_0,g_{1,n})\in\mathcal{H}_{x,n}$, $$
s_x(\alpha_x)=\sum_{j=0} ^n(-1)^j(g_n^{-1}\cdots g_1^{-1} g_0^{-1}x h, h_{i_1}, \cdots, h_{i_j}, \gamma_{s_i^j, x}, g_{j+1}, \cdots, g_n)\in \mathcal H_{x, n+1}, $$ 
when $j=0$, we set $\gamma_{s_i^0, x}=\gamma_{i, x}.$

Therefore,  we get  a lifting of the additive decomposition of $\HH_*(kG, kG)$ at the complex level.
 \begin{Thm}\label{realization-Hochschild-homology}  Let $k$ be a field and $G$ a finite group. Consider the additive
decomposition of Hochschild homology of the group algebra
$kG$:
$$\HH_*(kG,kG)\simeq \bigoplus_{x\in X}\mathrm H_*(C_G(x),k)$$
where $X$ is a set of representatives of conjugacy classes of elements of $G$ and $C_G(x)$ is the centralizer of $G$. 
Then, the above additive decomposition lifts to a homotopy deformation retract of complexes
\begin{equation*}\xymatrix@C=0.000000001pc{
\ar@(lu,dl)_-{s_*}&C_*(kG, kG) =\bigoplus\limits_{x\in X}\mathcal H_{x, *}\ar@<0.5ex>[rrrrr]^-{\rho_*}&&&&&\bigoplus\limits_{x\in X}  C_*(C_G(x), k).\ar@<0.5ex>[lllll]^-{\iota_*}}
\end{equation*}
where $\iota_*=\sum_{x\in X} \iota_x,\  \rho_*=\sum_{x\in X} \rho_x$ and $s_*=\sum_{x\in X} s_x$.  
\end{Thm}

\bigskip
{\it The third case.} Recall from Proposition \ref{singularHochschild-selfinjective}  that we have the following exact sequence
$$0\rightarrow \widehat{\HH}^{-1}(kG,kG)\rightarrow \HH_0(kG,kG) \stackrel{\tau}{\rightarrow}  \HH^0(kG,kG)\rightarrow \widehat{\HH}^0(kG,kG)\rightarrow 0.$$
By Remark \ref{remark-stable-hochschild}, the map $\tau$ is lifted to
$$\tau: C_0(kG, kG)=kG \rightarrow C^0(kG, kG)=kG, \ \ h\mapsto \sum_{g\in G} ghg^{-1}. $$
From the above two cases, it follows that the additive decompositions: $$\HH_0(kG,kG)\simeq \bigoplus_{x\in X}\mathrm H_0(C_G(x),k), \quad \HH^0(kG,kG)\simeq \bigoplus_{x\in X}\mathrm H^0(C_G(x),k)$$
are lifted to
\begin{equation}\label{equation-C0}
\begin{array}{rcl}
kG=C_0(kG, kG)\simeq \bigoplus_{x\in X}k[C_x] &\xrightarrow{} &\bigoplus_{x\in X} C_0(C_G(x), k)=\bigoplus_{x\in X}k_x, \\
x_i \in k[C_x] & \mapsto & 1_x \\
x & \mapsfrom & 1_x
\end{array}\end{equation} and
\begin{equation}\label{equation-C1}
\begin{array}{rcl}
kG=C^0(kG, kG)=\bigoplus_{x\in X} k[C_x] &\xrightarrow{} &\bigoplus_{x\in X} C^0(C_G(x), k)=\bigoplus_{x\in X} k_x.\\
x \in k[C_x] & \mapsto&1_x\in k_x\\
x_i \in k[C_x] & \mapsto & 0 \ \ \mbox{if $x_i\neq x$}\\
 \sum_{i=1}^{n_x}x_i&\mapsfrom & 1_x\in k_x
\end{array}
\end{equation}
where $k_x$ is the one-dimensional vector space indicated by $x$ and $1_x$ the unit of $k_x$.  For each $x\in X$, we also have an
exact sequence
$$0\rightarrow \widehat{\mathrm H}^{-1}(C_G(x),k)\rightarrow \mathrm H_{0}(C_G(x),k) \stackrel{\alpha_x}\rightarrow  \mathrm H^0(C_G(x),k)\rightarrow \widehat{\mathrm H}^{0}(C_G(x),k)\rightarrow 0,$$
where the map $\alpha_x$ is lifted to the map
$$\alpha_x: C_0(C_G(x), k) =k\rightarrow C^0(C_G(x), k)=k,\quad 1\mapsto |C_G(x)|.$$
Note that we have the following commutative diagram
\begin{equation*}
\xymatrix@R=1.3pc{
C_0(C_G(x), k)\ar@{_(->}[d] \ar[r]^-{\alpha_x} & C^0(C_G(x), k)\ar@{_(->}[d]\\
C_0(kG, kG) \ar[r]^-{\tau} & C^0(kG, kG)
}
\end{equation*}
where the vertical injections are defined in (\ref{equation-C0}) and (\ref{equation-C1}), respectively. This implies  that the restriction of the trace map $\tau$ to $C_0(C_G(x), k)$ is $\alpha_x$ for any $x\in X$. Thus we have the additive decompositions
$$\widehat{\HH}^{-1}(kG,kG)\simeq \bigoplus_{x\in X}\widehat{\mathrm H}^{-1}(C_G(x),k), \quad \widehat{\HH}^0(kG,kG)\simeq \bigoplus_{x\in X}\widehat{\mathrm H}^0(C_G(x),k).$$

 \bigskip
As a conclusion of the above three cases,   we get the following additive decomposition \begin{equation}\label{equation-tate-additive}\widehat{\HH}^*(kG, kG) \simeq  \bigoplus_{x\in X}\widehat{\mathrm H}^*(C_G(x), k).\end{equation}
Since we note that the trace map $\tau$ restricts to $\tau_x: \mathcal H_{x, 0}\to \mathcal H_x^{0}$ for any $x\in X$, we get a subcomplex: $$\widehat{\mathcal H}_x^*: \quad \cdots \rightarrow \mathcal H_{x, 1}\rightarrow \mathcal H_{x, 0}\xrightarrow{\tau_x}\mathcal H_x^0\rightarrow \mathcal H_x^1\rightarrow \cdots$$  of $\mathcal D^*(kG, kG)$.  It is clear  that $\mathcal D^*(kG, kG)= \bigoplus_{x\in X} \widehat{\mathcal H}_x^*$ as complexes.

\begin{Rem}\label{remark-deformation-retract2}
By  Theorems \ref{realization-Hochschild-cohomology} and \ref{realization-Hochschild-homology}, we obtain 
a homotopy deformation retract \begin{equation*}
 \xymatrix@C=0.000000001pc{
\ar@(lu,dl)_-{\widehat s}&\mathcal D^*(kG, kG) \ar@<0.5ex>[rrrrr]^-{\widehat{\rho}}&&&&&\bigoplus\limits_{x\in X} \widehat{ C}^*(C_G(x), k), \ar@<0.5ex>[lllll]^-{\widehat{\iota}}}
\end{equation*}
where, for $m\geq 0$, we have  $$\widehat{\iota}^{m}=\rho^m, \ \widehat{\iota}^{-m-1}=\iota_m; \quad \widehat{\rho}^m=\iota^m,\  \widehat{\rho}^{-m-1}=\rho_{m}; \quad 
\widehat{s}^m=s^m,\  \widehat{s}^{-m-1}=s_{m}.$$ 
This homotopy deformation retract should play a crucial role in the future study of the behavior of  the higher algebraic structures on $\mathcal D^*(kG, kG)$ in terms of the additive decomposition. 

Taking $x=1$, we obtain a split  inclusion of complexes $ \widehat{\iota}_{x=1}: \widehat{C}^*(G, k)\hookrightarrow \mathcal D^*(kG, kG)$ given as follows: 
$$\begin{array}{rcl}
C^*(G, k)&\hookrightarrow& C^*(kG, kG),\\
(\varphi: \overline{G}^{\times n}\longrightarrow k)&\longmapsto& (\psi: \overline{G}^{\times n}\longrightarrow kG), \quad \psi(g_{1,n})=\varphi(g_{1,n})g_1\cdots g_{n};\\
C_*(G, k)&\longrightarrow& C_*(kG, kG),\\
(g_{1,n}) &\longmapsto& (g_n^{-1}\cdots g_1^{-1}x,g_{1,n}).\end{array}$$
In particular, this induces an inclusion $\widehat{\iota}_{x=1}:\widehat{\mathrm H}^*(G, k)\hookrightarrow \widehat{\HH}^*(kG, kG)$. 
\end{Rem}

We define a left $kG$-module ${_ckG}$ as follows. As a vector space
${_ckG}=kG$ and the action of $G$ on $kG$ is given by conjugation: $g\cdot x=gxg^{-1}$ for any $g\in G$ and $x\in kG$. Note that we have a $kG$-module decomposition ${_ckG} = \bigoplus_{x\in X}kC_x$, where $C_x$ denotes the conjugacy class of $x$. 

\begin{Prop}\label{isomorphism-of-complexes}
We have an isomorphism of complexes $\rho: \mathcal{D}^*(kG, kG)\rightarrow \widehat{C}^*(G, {_ckG})$. As a result, we can present the isomorphisms in the additive decomposition  as follows: $$\widehat{\HH}^*(kG, kG) \simeq \widehat{\mathrm H}^*(G, {_{c}kG}) \simeq \bigoplus_{x\in X}\widehat{\mathrm H}^*(G, kC_x) \simeq \bigoplus_{x\in X}\widehat{\mathrm H}^*(C_G(x), k).$$
\end{Prop}
\begin{proof} (Compare to \cite[Remark 6.2]{LZ2015})
Let us construct the morphism of complexes
$\rho: \mathcal{D}^*(kG, kG)\rightarrow \widehat{C}^*(G, {_ckG})$ as follows.
For $m\geq 0$ and $\phi\in \mathcal{D}^m(kG, kG)\simeq \Map(\overline{G}^{\times m},kG)$, we define $$\rho(\phi)\in \widehat{C}^m(G, {_ckG})\simeq \Map(\overline{G}^{\times m},kG),\quad (g_{1, m})\mapsto\phi(g_{1, m})g_m^{-1}\cdots g_1^{-1}.$$  In fact, for each $x\in X$, $\rho$ restricts to an isomorphism $\rho_x: \mathcal{H}_x^m\xrightarrow{\simeq} \Map(\overline{G}^{\times m},kC_x).$ Similarly, for $m\geq 0$ and $(h, g_{1,m})\in
\mathcal{D}^{-m-1}(kG, kG)\simeq k[G\times\overline{G}^{\times {m}}]$, we define $$\rho((h, g_{1,m}))=(hg_1\cdots g_{m}, g_{1, m})\in \widehat{C}^{-m-1}(G, {_ckG})\simeq k[G\times  \overline{G}^{\times {m}}].$$
In fact, for each $x\in X$, $\rho$ restricts to an isomorphism $\rho_x: \mathcal{H}_{x,m}\xrightarrow{\simeq} k[C_x\times \overline{G}^{\times m}]$. It is easy to check that $\rho$ is a morphism of complexes. Note that $\rho$ is an isomorphism with inverse $\rho^{-1}$   given by $$\rho^{-1}(\psi)(g_{1, m}) =\psi(g_{1, m})g_1\cdots g_m,\quad \mbox{for any $\psi\in \widehat{C}^m(G, {_ckG});$}$$ $$\rho^{-1}((h, g_{1,m}))=(hg_{m}^{-1}\cdots g_1^{-1}, g_{1,m}), \quad \mbox{for any $(h, g_{1, m})\in \widehat{C}^{-m-1}(G, {_ckG})$}$$ for $m\geq 0$. Thus the isomorphism $\rho$ induces the first two isomorphisms in this proposition.  The third isomorphism in the proposition follows from the following isomorphisms of complexes
$$kC_x\otimes_{kG}\Barr_*(kG)\otimes_{kG}k \simeq k\otimes_{kC_G(x)} \Barr_*(kG)\otimes_{kG}k, $$
$$\Hom_{kG}(\Barr_*(kG)\otimes_{kG} k, kC_x)\simeq \Hom_{kC_G(x)}(\Barr_*(kG)\otimes_{kG}k, k).$$ We remark that the above two isomorphisms  are induced by  $$\Hom_{kC_G(x)}(kG, k)\simeq kC_x\simeq k\otimes_{kC_G(x)} kG, \quad (\gamma_{j, x}\mapsto \delta_{i, j}) \mapsfrom \gamma_{i, x}^{-1}x \gamma_{i, x}\mapsto 1\otimes_{kC_G(x)}\gamma_{i, x}.$$  Here we fix a right coset decomposition of $C_G(x)$ in $G$: $G=C_G(x)\gamma_{1, x}\cup \cdots \cup C_G(x)\gamma_{n_x, x}$ and thus $C_x=\{\gamma_{1, x}^{-1}x\gamma_{1, x}, \cdots, \gamma_{n_x, x}^{-1}x\gamma_{n_x, x}\}$. This proves the proposition.
\end{proof}

\begin{Rem}\label{duality-and-additive-decomposition}
Note that the non-degenerate
bilinear pairing on $\mathcal{D}^*(kG, kG)$ (cf. Remark \ref{duality-in-Tate-Hochschild-cohomology}) induces  non-degenerate
bilinear pairings (for each $n\in \mathbb Z$):
$$\langle\cdot, \cdot\rangle: \widehat{\mathcal H}_x^n \times \widehat{\mathcal H}_{y}^{-n-1}\rightarrow k \quad \mbox{for any $x, y \in X$ such that $x^{-1}\in C_y$}.$$
In particular,  we  have 
$\langle \alpha, \beta\rangle=0$ for any $\alpha \in \widehat{\mathcal H}_x^*$ and $\beta \in \widehat{\mathcal H}_y^* \ (x, y\in X)$ such that $x^{-1}\notin C_{y}$. 
\end{Rem}

Recall that we have a cyclic $A_{\infty}$-algebra structure $(\langle\cdot, \cdot\rangle, \partial', \cup, m_3, m_i=0 \ (i>3))$ on $\mathcal D^*(kG, kG)$ (see Section \ref{section2}). Notice that the Tate cochain complex $\widehat C^*(G, k)$ can be seen as a subcomplex of  $\mathcal D^*(kG, kG)$ under the inclusion $\widehat{\iota}_{x=1}: \widehat{C}^*(G,k)\hookrightarrow \mathcal D^*(kG, kG)$ (cf.  Remark \ref{remark-deformation-retract2}). 
\begin{Thm}\label{theorem-A-infinity}
The Tate cochain complex $\widehat C^*(G, k)$ is a cyclic $A_{\infty}$-subalgebra of $\mathcal D^*(kG, kG)$,  and moreover  $\widehat{\mathcal H}^*_x$ is an $A_{\infty}$-module of $\widehat{C}^*(G, k)$ for each $x\in X$, under the  decomposition $\mathcal D^*(kG, kG)= \bigoplus_{x\in X}\widehat{\mathcal H}^*_x$. 
\end{Thm}

\begin{proof}
By Proposition \ref{isomorphism-of-complexes}, we have an isomorphism of complexes
$$\widehat{\iota}_{x=1}=\rho^{-1}|_{x=1}: \widehat C^*(G, k)\xrightarrow{\simeq}\widehat{\mathcal H}_1^*\ (\subset \mathcal D^*(kG, kG)).$$ 
For any $\alpha\in \widehat{\mathcal H}_1^m$ and $\beta\in \widehat{\mathcal H}_x^n\ (x\in X)$, by the definition of $\cup$ in Section \ref{section2}, it is easy  to check that
$$\alpha\cup \beta, \ \beta\cup \alpha \in \widehat{\mathcal H}_x^{m+n}.$$ 
This shows that $\cup$ on $\mathcal D^*(kG, kG)$ restricts to $\widehat{\mathcal H}_1^*$ and $\widehat{\mathcal H}_x^*$ is a module of $\widehat{C}^*(G, k)$. 
It remains to verify that 
$m_3(\alpha, \beta, \gamma)\in \widehat{\mathcal H}^*_x$ for $\alpha, \beta\in \widehat{\mathcal H}_1^*$ and $\gamma\in \widehat{\mathcal H}^*_x$. Recall from Section \ref{section2}, we only need to consider  the two nontrivial cases for $m_3$.   The first case is as follows. Let $\phi \in \mathcal H^m_{1}, \alpha=(g_0, g_{1, r})\in \mathcal H_{1, r}$ and $ \varphi \in \mathcal H^n_x$ for $r+2\leq m+n$.
We have 
\begin{eqnarray*}
\lefteqn{
m_3(\phi, \alpha, \varphi)(h_1, \cdots, h_{m-r+n-2})=}\\
&&\sum_{g\in G}\sum_{j=1}^{\min\{m, n, r\}}(-1)^{m+r+j-1} \phi(h_{1, m-r+j-2},  g, g_{j, r})g_0
\varphi(g_{1, j-1}, g^{-1}, h_{m-r+j-1, m-r+n-2}). \\
\end{eqnarray*}
Since we have that  $\phi(y_1, \cdots, y_m)\in k[y_1y_2\cdots y_m]$ for any $y_i \in G \ (1\leq i\leq m)$, $g_0g_1\cdots g_r=1$, and $\varphi(y_1, \cdots, y_n)\in k[y_1y_2\cdots y_nC_x]$ for any $y_i \in G \ (1\leq i\leq n)$,  $$\phi(h_{1, m-r+j},  g, g_{j, r})g_0
\varphi(g_{1, j-1}, g^{-1}, h_{m-r+j+1, m-r+n})\in k[h_1\cdots h_{m-r+n}C_x].$$ Thus $m_3(\phi, \alpha, \psi)\in \mathcal H_x^{m-r+n}.$ The second case can be verified in a similar way.  
\end{proof}
\begin{Cor}\label{cor-abelian-A}
Let $G$ be a finite abelian group. Then we have an isomorphism $$\mathcal D^*(kG, kG)\simeq kG\otimes \widehat{C}^*(G, k)$$ as cyclic $A_{\infty}$-algebras.
\end{Cor}
\begin{Rem}The cyclic $A_{\infty}$-algebra structure $(\langle\cdot,\cdot\rangle', m_1',m_2', \cdots)$ on $kG\otimes \widehat{C}^*(G, k)$ is defined as follows.
$$\langle g_1\otimes \alpha_1, g_2\otimes \alpha_2\rangle'=\langle g_1, g_2\rangle\langle \alpha_1, \alpha_2\rangle,$$
$$m'_p((g_1\otimes \alpha_1), \cdots, (g_p\otimes \alpha_p))=g_1\cdots g_p\otimes m_p(\alpha_1, \cdots, \alpha_p)$$
for any $g_i\otimes \alpha_i\in kG\otimes \widehat{C}^*(G, k)\ (i=1, \cdots, p).$

\end{Rem}
\begin{proof}[Proof of Corollary \ref{cor-abelian-A}]
From Proposition \ref{isomorphism-of-complexes}, we have an isomorphism $\rho: \mathcal D^*(kG, kG)\xrightarrow{\simeq} kG\otimes \widehat C^*(G, k)$, which induces an isomorphism  $\rho_x: \widehat{\mathcal H}^*_x\xrightarrow{\simeq} kx\otimes \widehat{C}^*(G, k)$ for any $x\in G$.   It is easy to check that  $\rho$ respects the cup product. Let us prove that $\rho$ respects the product $m_3$ as well. We need to check the following two cases. For the first case,  let $\phi \in \mathcal H^m_{x}, \ \psi \in \mathcal H^n_y$ and $\alpha=(g_0, g_1, \cdots, g_r)\in \mathcal H_{z, r}$ for $x, y, z\in G$ and $m+n\geq r$. By the above isomorphisms $\rho_x\ (x\in X)$, we have that $\phi=x \overline{\phi}, \ \psi=y\overline{\psi}$  for some $\overline{\phi}\in\mathcal H_{1}^m, \overline{\psi}\in \mathcal H_{1}^n$, and  $g_0\cdots g_r=z.$ Thus
\begin{equation*}
\begin{split}
&m_3(\phi, \alpha, \psi)(h_1, \cdots, h_{m-r+n})=\\
&\sum_{g\in G}\sum_{j=1}^{\min\{m, n, r\}}(-1)^{\kappa_{ij}} \phi(h_{1, m-r+j},  g, g_{j, r})g_0
\psi(g_{1, j-1}, g^{-1}, h_{m-r+j+1, m-r+n})\in k[h_1\cdots h_{m-r+n}C_{xyz}].
\end{split}
\end{equation*}
This implies that $m_3(\phi, \alpha, \psi)\in \mathcal H_{xyz}^{m-r+n}$. For the second case, let $\alpha\in \mathcal H_{x, r}, \ \beta\in \mathcal H_{y, s}$ and $\phi\in \mathcal H_{z}^m$ for $m-1\leq r+s$. By Equation (\ref{equation-cyclic-pairing}) , we have
$\langle m_3(\alpha, \phi, \beta), \psi\rangle=\langle \alpha,  m_3(\phi, \beta, \psi)\rangle$ for any $\psi\in \widehat{\mathcal H}_{w}^{*}\  (w\in X)$. It follows from Remark \ref{duality-and-additive-decomposition} that $m_3(\alpha, \phi, \beta)\in \widehat{\mathcal H}^*_{xyz}$ since we have $m_3(\phi, \beta, \psi)\in \widehat{\mathcal H}_{yzw}^*$ by the first case.  This proves the corollary.
\end{proof}

\bigskip

\section{The cup product formula and a new proof via Green functors}\label{section-cup-product}

In this section, we describe Nguyen's cup product formula for the Tate-Hochschild cohomology algebra $\widehat{\HH}^*(kG,kG)$ in terms of the additive decomposition and  provide a new proof via Green functors, following Bouc.

 \subsection{The cup product formula}
We shall state the cup product formula at the cohomology level analogous to the result of  Siegel-Witherspoon \cite{SW1999}. In fact, this has been done by Nguyen in \cite{Nguyen2012}.

  Let $X=\{g_1=1,g_2,\cdots, g_r\}$ be a complete set of representatives of conjugacy classes of
elements of $G$ and $H_i:=C_G(g_i)$ is the centralizer subgroup of $G$.
Let  $\gamma_i:\widehat{\mathrm H}^*(H_i,
 k)\to \widehat{\HH}^*(kG,kG) $  be the split injection appearing in  the additive decomposition in  Proposition \ref{isomorphism-of-complexes}
$$\widehat{\HH}^*(kG,kG)\simeq \mathrm{\widehat{\Hu}}^*(G,{_ckG}) \simeq \bigoplus_{g_i\in X} \mathrm{\widehat{\Hu}}^*(H_i,k).$$

We now state the cup product formula in terms of the above  additive decomposition of $\widehat{\HH}^*(kG,kG)$. For $i,j\in \{1,\cdots,r\}$, let $D$ be a set of double coset representatives for $H_i\backslash G/H_j$. Recall that for each $x\in D$, there is a unique $k=k(x)$ such that $g_k={}^yg_i{}^{yx}g_j$ for some $y\in G$.
\begin{Thm} \label{cup-product-formula}  $($\cite[Theorem 5.5]{Nguyen2012}$)$ Let $\alpha\in \widehat{\mathrm H}^*(H_i,k)$, $\beta\in \widehat{\mathrm H}^*(H_j,k)$. Then
$$\gamma_i(\alpha)\cup \gamma_j(\beta)=\sum_{x\in D}\gamma_k(cor^{H_k}_{W}(res^{{{}^y\!H_i}}_{W}y^*(\alpha)\cup res^{{}^{yx}\! H_j}_{W}{(yx)}^*(\beta))),$$
where $W=W(x)={}^{yx}\!H_j\bigcap {}^y\!H_i$.
\end{Thm}

Note that conjugation maps, restriction maps and corestriction maps appeared  in the above formula have been recalled in Section \ref{subsection: remainder on group cohomology}. The cup product formula for Hochschild cohomology $\HH^*(kG,kG)$ was given by  Siegel and Witherspoon  in \cite{SW1999}.


From the above cup product formula it is clear that the Tate cohomology algebra $\widehat{\mathrm H}^*(G, k)$ can be seen as a graded subalgebra of $\widehat{\HH}^*(kG,kG)$ and $$\widehat{\HH}^*(kG,kG)\simeq \bigoplus_{x\in X}\widehat{\mathrm H}^*(C_G(x), k)=\widehat{\mathrm H}^*(G,
k)\oplus(\bigoplus_{x\in X- \{1\}}\widehat{\mathrm H}^*(C_G(x), k))$$ is an
isomorphism of graded $\widehat{\mathrm H}^*(G, k)$-modules (see also Theorem \ref{theorem-A-infinity}).


\subsection{Mackey functors}

There are at least three different (but
equivalent) points of view for Mackey functors and Green functors.  We shall recall two of them; for more details we refer the reader to \cite{Bouc1997, Bouc2003a, Bouc2003}. We will see that  the cup product formula from the previous section follows essentially from the equivalence between two definitions of Green functors.

 Let $k$   be a field and  $G$ be a finite group.
\begin{Def}\label{Def-Mackey-1}  A Mackey functor for $G$ over $k$ is given by the following data:

For an arbitrary subgroup $H$ of $G$, we are  given a $k$-module $M(H)$ and homomorphisms of $k$-modules
$$t^{K}_{H}:M(H)\rightarrow M(K),$$
$$r^{K}_{H}:M(K)\rightarrow M(H),$$
$$c_{g,H}:M(H)\rightarrow M({}^{g}\!H),$$ for $H\leq  K\leq  G, g\in G$, where ${}^g\!H=gHg^{-1}$.
They satisfy the following conditions:
\begin{itemize}
\item[(1)] (transitivity) If $H\subseteq K\subseteq L$  are subgroups of $G$, then
$$t^{L}_{K}t^{K}_{H}=t^{L}_{H}, r^{K}_{H}r^{L}_{K}=r^{L}_{H};$$
if $g, g'\in G,H\leq G$, then $c_{g',{}^{g}\!H}c_{g,H}=c_{g'g,H}.$

\item[(2)](compatibility)  If $H\subseteq K$  are subgroups of $G$ and  $g\in G$, then
$$c_{g,K}t^{K}_{H}=t^{{}^{g}\!K}_{{}^{g}\!H}c_{g,H},$$
$$c_{g,H}r^{K}_{H}=r^{{}^{g}\!K}_{{}^{g}\!H}c_{g,K}.$$

\item[(3)] (triviality) If $H$ is a subgroup of $G$, then
$$t^{H}_{H}=r^{H}_{H}=\mathrm{Id}_{M(H)};$$
moreover, if $g\in H$, $c_{g, H}=\mathrm{Id}_{M(H)}$.

\item[(4)](Mackey axiom) If $H\subseteq K\supseteq L$ are subgroups of $G$, then
$$r^{K}_{H}t^{K}_{L}=\sum_{u\in [H\setminus K/L]}t^{H}_{H\cap {}^{u}L}c_{u,H^{u}\cap L}r^{L}_{H^{u}\cap L},$$
where $[H\setminus K/L]$ is a set of representatives of the double cosets of $K$ modulo $H$ and $L$ and $H^u=u^{-1}Hu$. \end{itemize}
\end{Def}
The maps $t^K_H$ are called transfers or traces and the maps $r^K_H$ are called restrictions.

The cohomology of finite groups is a Mackey functor. In fact, fix a finite group $G$, for $H\leq G$, define $M(H)=\mathrm H^*(H,k)$; for $H\leq  K\leq  G, g\in G$, the maps  $t^{K}_{H}, r^{K}_{H}, c_{g,H}$
are respectively the maps  $cor^K_H, res^K_H, g^*$ recalled in Section \ref{subsection: remainder on group cohomology}.  The Tate cohomology of finite groups is also a Mackey functor with similarly defined maps.


Let us introduce the second definition of Mackey functors and this definition uses the category $G$-$\mathrm{set}$ of finite  left $G$-sets. Recall that a {\it bivariant functor} from $G$-$\mathrm{set}$ to the category $k$-$\mathrm{Mod}$ of $k$-modules is a pair of functors
$(\tilde{M}_{*},\tilde{M}^{*})$ from  $G$-$\mathrm{set}$ to $k$-$\mathrm{Mod}$, where $\tilde{M}_{*}$ is covariant and $\tilde{M}^{*}$ is contravariant, such that the pair of functors coincide on objects. That is,  for any $G$-set $X$, the two $k$-modules $\tilde{M}_{*}(X)$ and $\tilde{M}^{*}(X)$ are the same (denoted by $\tilde{M}(X)$).

\begin{Def}\label{Def-Mackey-2}
A Mackey functor for $G$ over $k$ is a bivariant functor 
$(\tilde{M}_{*},\tilde{M}^{*})$ from  $G$-$\mathrm{set}$ to $k$-$\mathrm{Mod}$ satisfying  the following conditions:
\begin{itemize}
\item[(1)](additivity)  If $X$ and $Y$ are finite $G$-sets, $i_{X}$ and $i_{Y}$ are respectively the inclusion map of $X$ and $Y$ to the disjoint union $X\sqcup Y$, then the maps
 $$\tilde{M}(X)\oplus \tilde{M}(Y)\stackrel{(\tilde{M}_{*}(i_{X}),\tilde{M}_{*}(i_{Y}))}{\longrightarrow}\tilde{M}(X\sqcup Y)$$ and $$\tilde{M}(X\sqcup Y)
\stackrel{\left
(\begin{array}{c}
\tilde{M}^{*}(i_{X})\\
\tilde{M}^{*}(i_{Y})
\end{array}\right)}
{\longrightarrow}\tilde{M}(X)\oplus \tilde{M}(Y)$$ are isomorphisms inverse one to each other.

\item[(2)](cartesian squares) If $$\begin{array}[c]{ccc}
X&\stackrel{a}{\rightarrow}&Y\\
\downarrow\scriptstyle{b}&&\downarrow\scriptstyle{c}\\
T&\stackrel{d}{\rightarrow}&Z
\end{array}$$ is a pullback (or equivalently, cartesian square) of finite $G$-sets, then we have the equality $$\tilde{M}_{*}(b)\tilde{M}^{*}(a)=\tilde{M}^{*}(d)\tilde{M}_{*}(c).$$
\end{itemize}
\end{Def}
Notice that the property for cartesian squares in the second definition (i.e. Definition \ref{Def-Mackey-1}) implies  the Mackey formula in the first definition (i.e. Definition \ref{Def-Mackey-2}).

Morphisms between Mackey functors are natural transformations between bivariant functors and compositions of morphisms are just compositions of natural transformations.

The equivalence between the two definitions can be explained as follows.

 Give a Mackey functor $M$ in the sense of the first definition,    for a finite $G$-set $X$ and each $x\in X$, write $G_x\leq G$ its stabilizer. Then define $\tilde{M}(X)=(\oplus_{x\in X} M(G_x))^G$, where the $G$-action on $\oplus_{x\in X} M(G_x)$ is given by $g\cdot \alpha= c_{g, G_x}(\alpha)\in  M({}^gG_x)=M(G_{gx})$ for  $x\in X, g\in G, \alpha\in M(G_x)$. It is not difficult to verify that $\tilde{M}$ is a Mackey functor in the sense of the second definition. Usually we take $[G\backslash X]$  a set of representatives of the orbits of $G$ in $X$, and  write $\tilde{M}(X)=\oplus_{x\in [G\backslash X]} M(G_x)$.

Conversely, given a Mackey functor $\tilde{M}$ in the sense of the second definition, then for  a subgroup $H\leq G$, the set of left cosets $G/H$ can be considered as a transitive $G$-set with left multiplication as $G$-action, then define $M(H)=\tilde{M}(G/H)$, and this is a Mackey functor in the sense of the first definition.

In the following, we will not distinguish between $M$ and $\tilde{M}$, and write only $M$.

Return to the example of the Mackey functor given by Tate cohomology of finite groups. For a finite $G$-set $X$, define $M(X)=\widehat{\mathrm H}^*(G, k[X])$;   for a morphism of $G$-sets $f: X\to Y$, the   map $M_*(f): M(X)\to M(Y)$ is the usual map $\widehat{\mathrm H}^*(G, k[X])\to \widehat{\mathrm H}^*(G, k[Y])$, since Tate cohomology is covariant in the coefficients. Obviously for $H\leq G$, $M(G/H)=\widehat{\mathrm H}^*(G, k[G/H])=\widehat{\mathrm H}^*(G, kG\otimes_{kH}k)\simeq \widehat{\mathrm H}^*(H, k)$.

Let $\Gamma$ be a finite $G$-set. For a Mackey functor $M$ for $G$ over $k$, the \textit{Dress construction}   gives a
new Mackey functor $M_{\Gamma}$ defined as follows: for any
$G$-set $X$, $M_{\Gamma}(X)=M(X\times \Gamma)$. It is not
difficult to see that $M_\Gamma$ is a Mackey functor (\cite[1.2]{Bouc1997}).

\subsection{Green functors}
A Green functor A is a Mackey functor ``with a compatible
ring structure''. There is also a definition of Green functors in terms of
$G$-sets (\cite[2.2]{Bouc1997}).

\begin{Def} A Green functor for a finite group $G$  over $k$ is a Mackey functor $A$, such that for each subgroup $H\leq G$, $A(H)$ has a structure of $k$-algebra. We ask that  the Mackey functor structure and the algebra structure satisfy the following conditions:
\begin{itemize}

\item[(1)]If $H\subseteq K$ are subgroups of $G$, for arbitrary $g\in G$, the $k$-module homomorphisms $r^{K}_{H}$ and $c_{g,H}$ are homomorphisms of $k$-algebras.

\item[(2)](Frobenius identity) If $H\subseteq K$ are subgroups of $G$,  for arbitrary elements $a\in A(H),b\in A(K),$ we have
$$b\cdot (t^{K}_{H}a)=t^{K}_{H}((r^{K}_{H}b)\cdot a),$$
$$(t^{K}_{H} a)\cdot b=t^{K}_{H}(a\cdot (r^{K}_{H}b)).$$
\end{itemize}
Let $A,B$ be two Green functors, and let $f:A\longrightarrow B$ be a morphism between Mackey functors. If for each subgroup $H\leq G$, the map $f_{H}:A(H)\longrightarrow B(H)$ is a homomorphism of $k$-algebras, then $f$ is a morphism of Green functors.
\end{Def}

Similarly, Green functors have a definition via $G$-sets.
\begin{Def} A Green functor over $k$ is a Mackey functor $A$, such that  for each pair of finite $G$-sets $X,Y$, there is a homomorphism of $k$-modules $A(X)\times A(Y)\to A(X\times Y), (a,b)\longrightarrow a\times b$ satisfying the following conditions:
\begin{itemize}
\item[(1)](bifunctoriality) If $f:X\longrightarrow X'$ and $g:Y\longrightarrow Y'$ are homomorphisms of finite $G$-sets, then there exist the following commutative diagrams£»
$$\begin{array}[c]{ccc}
A(X)\times A(Y)&\stackrel{\times}{\rightarrow}&A(X\times Y)\\
\downarrow\scriptstyle{A_{*}(f)\times A_{*}(g)}&&\downarrow\scriptstyle{A_{*}(f\times g)}\\
A(X')\times A(Y')&\stackrel{\times}{\rightarrow}&A(X'\times Y')
\end{array}$$ and
$$\begin{array}[c]{ccc}
A(X)\times A(Y)&\stackrel{\times}{\rightarrow}&A(X\times Y)\\
\uparrow\scriptstyle{A^{*}(f)\times A^{*}(g)}&&\uparrow\scriptstyle{A^{*}(f\times g)}\\
A(X')\times A(Y')&\stackrel{\times}{\rightarrow}&A(X'\times Y')
\end{array}$$

\item[(2)](associativity) If $X,Y,Z$ are finite $G$-sets, we have the following commutative diagram:
$$\begin{array}[c]{ccc}
A(X)\times A(Y)\times A(Z)&\stackrel{Id_{A(X)}\times (\times)}{\rightarrow}&A(X)\times A(Y\times Z)\\
\downarrow\scriptstyle{(\times)\times Id_{A(Z)}}&&\downarrow\scriptstyle{\times}\\
A(X\times Y)\times A(Z)&\stackrel{\times}{\rightarrow}&A(X\times Y\times Z)
\end{array}$$
where $((X\times Y)\times Z)\simeq X\times Y\times Z\simeq X\times (Y\times Z)$.

\item[(3)] Let $\bullet$ be the $G$-set with one element. Then there exists an element $\varepsilon_{A}\in A(\bullet)$ such that for each finite  $G$-set $X$ and arbitrary element $a\in A(X)$, we have
$$A_{*}(p_{X})(a\times \varepsilon_{A})=a=A_{*}(q_{X})(\varepsilon_{A}\times a),$$
where $p_{X}$(resp. $q_{X})$ is the projection from $X\times  \bullet$ (resp. $ \bullet \times X)$ to $X$.
\end{itemize}
\end{Def}

Let $A,B$ be Green functors for the group $G$ over $k$. Then a morphism between them is a morphism of Mackey functors $f:A\longrightarrow B$ such that for any finite $G$-sets $X,Y$, the following diagram is commutative:
$$\begin{array}[c]{ccc}
A(X)\times A(Y)&\stackrel{\times}{\rightarrow}&A(X\times Y)\\
\downarrow\scriptstyle{f_{X}\times f_{Y}}&&\downarrow\scriptstyle{f_{X\times Y}}\\
B(X)\times B(Y)&\stackrel{\times}{\rightarrow}&B(X\times Y).
\end{array}$$

Let ${}_cG$ be the $G$-set $G$ with
conjugation action.  Recall that a \textit{crossed
$G$-monoid} $\Gamma$ is a  $G$-monoid with a
$G$-monoid map from $\Gamma$ to ${}_cG$.   Let $A$ be a Green functor over $k$ for $G$ and $\Gamma$ a crossed $G$-monoid. Then
  Bouc proved that the Dress construction $A_\Gamma$ is a
Green functor (\cite[Theorem 5.1]{Bouc2003}).
Notice that in this case, $A(\Gamma)=A_\Gamma(\bullet)$, the evaluation at the trivial $G$-set $\bullet$ of $A_\Gamma$, has a new $k$-algebra structure: for $a, b\in A(\Gamma)$, define their product $a\times_\Gamma b$  to be $A_*(\mu_\Gamma)(a\times b)$, where $\mu_\Gamma: \Gamma \times \Gamma\to \Gamma$ is the multiplication of the $G$-monoid $\Gamma$.

\subsection{A new proof of the cup product formula}

Now we explain how the result in \cite{Bouc2003} gives a quick proof of the main result of \cite{Nguyen2012}, that is, the cup product formula for Tate-Hochschild cohomology.


Let us recall Bouc's result. In the following statement, we write the action of $G$ on $\Gamma$ as ${}^g\gamma$.
\begin{Thm}\label{Thm:Green functor}\cite[Theorem 6.1 and Corollary 6.2]{Bouc2003}
 Let $A$ be a Green functor for $G$ over $k$ and $\Gamma$ a crossed $G$-monoid.  Then
 $$A_\Gamma(\bullet)=A(\Gamma)=\big(\bigoplus_{\gamma\in \Gamma} A(G_\gamma)\big)^G$$
 and the $\gamma$-component of the product of $a, b\in A(\Gamma)$ is
 $$(a\times_\Gamma b)_\gamma=\sum_{(\alpha, \beta)\in G_\gamma\backslash (\Gamma\times \Gamma), \alpha\beta=\gamma} t^{G_\gamma}_{G_{(\alpha, \beta)}} \big(  (r^{G_\alpha}_{G_{(\alpha,\beta)}} a_\alpha) \times (r^{G_\beta}_{G_{(\alpha,\beta)}} b_\beta)\big)  $$

Taking  a set of orbit representatives $[G  \backslash \Gamma ]$,  there is an isomorphism of $k$-modules
$$A(\Gamma)\simeq \bigoplus_{\gamma\in [G  \backslash \Gamma ]}  A(G_{\gamma}) $$
where $[G\backslash \Gamma]$ is a set of representatives of the orbits of $G$ in $\Gamma$. With this notation, the product of $a\in A(G_{\gamma})$ and $b\in A(G_{\delta})$ is equal to
$$\bigoplus_{\epsilon\in [G\backslash \Gamma]} \bigoplus_{\omega\in [G_{\gamma}\backslash G\slash G_{\delta}]}
t^{G_{\epsilon}}_{G_{{}^{g(\omega, \epsilon)}\! \gamma}\cap G_{{}^{g(\omega, \epsilon)  \omega} \delta}}  c_{g(\omega, \epsilon),G_\gamma\cap G_{{}^\omega\delta} }\big( r^{G_\gamma}_{G_\gamma\cap G_{{}^\omega\delta}}a \cdot r^{G_{{}^\omega\delta}}_{G_\gamma\cap G_{{}^\omega\delta}}b\big)$$
where $g(\omega, \epsilon)=g$ is an element of the unique class $G_{\epsilon}g $ in $G_{\epsilon}\backslash G$ such that ${}^g\!(\gamma {}^\omega\!\delta)=\epsilon$.

\end{Thm}

Let $k$ be a field and $G$ a finite group.  Then $G$ acts by conjugation on itself and denote by ${}_cG$ this $G$-set. Then it is well known that the Tate cohomology $A=\widehat{\mathrm H}^*(G,
k[?])$ sending a finite $G$-set $X$ to $\widehat{\mathrm H}^*(G,
k[X])$ is a Green functor.  Consider the crossed $G$-monoid
$\Gamma=({}_cG, u)$ where $u: {}_cG \to {}_cG$ is the trivial
homomorphism of $G$-monoids sending each element to the unit in
${}_cG$. Then one verifies easily that for the Dress construction we get
$A_{\Gamma}=\widehat{\mathrm H}^*(G, k[?\times {}_cG])$.
Remark that by \cite[Section 4]{Nguyen2012}, $A(\Gamma)=A_\Gamma(\bullet)=\widehat{\mathrm H}^*(G,
k[{}_cG])$, together with the $k$-algebra structure  $a\times_\Gamma b$  defined above,   is isomorphic to the Tate-Hochschild cohomology ring $\widehat{\HH}^*(kG,kG).$

 Nguyen \cite{Nguyen2012} considered the additive decomposition for the Tate-Hochschild cohomology ring of a group algebra. Her proof is similar to that of \cite{SW1999}. Notice that Bouc's above result also applies to this situation and yields a new proof of \cite[Theorem 5.5]{Nguyen2012} just as is explained in \cite[Page 421]{Bouc2003}. Let us explain the details.

 As a $G$-set, the $G$-orbits in ${}_cG$ are just conjugacy classes. Let $X=\{g_1=1,g_2,\cdots, g_r\}$ be a complete set of representatives of conjugacy classes of
elements of $G$ and denote the centralizer subgroup $C_G(g_i)$ of $G$ by $H_i$ for $1\leq i\leq r$.  As $G$-sets, ${}_cG$ is isomorphic to $\coprod_{i=1}^r G/H_i$, and $$A(G/H_i)=\widehat{\mathrm H}^*(G,
k[G/H_i])=\widehat{\mathrm H}^*(G,
kG\otimes_{kH_i} k])\simeq \widehat{\mathrm H}^*(H_i).$$  So we have an isomorphism of graded vector spaces
$$\widehat{\HH}^*(kG)\simeq \widehat{\mathrm H}^*(G, k[{}_cG])\simeq \bigoplus_{i=1}^r\widehat{\mathrm H}^*(H_i).$$

Fix $i,j\in \{1,\cdots,r\}$. Let $D$ be a set of double coset representatives for $H_i\backslash G/H_j$. Recall that for each $x\in D$, there is a unique $k=k(x)$ such that $g_k={}^y(g_i{}^{x}g_j)$ for some $y\in G$.
Now Theorem \ref{Thm:Green functor} gives the cup product formula:
$$\begin{array}{rcl}\gamma_i(\alpha)\cup \gamma_j(\beta)&=&\sum_{x\in D} \gamma_k (cor^{H_k}_{{}^y\!H_i\cap {}^{yx}\!H_j}  y^*(res^{H_i}_{H_i\cap {}^x\!H_j}\alpha\cup res^{{}^{x}\!H_j}_{H_i\cap {}^x\!H_j}(\beta)))\\
&=&\sum_{x\in D}\gamma_k(cor^{H_k}_{W}(res^{{^y\!H_i}}_{W}y^*(\alpha)\cup res^{{}^{yx\!}H_j}_{W}{(yx)}^*(\beta))),\end{array}$$
where $W=W(x)={}^{yx}H_j\bigcap {}^yH_i$. This is exactly  \cite[Theorem 5.5]{Nguyen2012}.

\begin{Rem}\label{remark-cup-product-negative}
The above cup product formula on $\widehat{\HH}^*(kG, kG)$  can be lifted to a cup product formula at the complex level by the homotopy deformation retract in Remark \ref{remark-deformation-retract2}\footnote{There should have an $A_{\infty}$-product formula at the complex level using the Homotopy Transfer Theorem. Here we only consider the product formula for $m_2$. The higher product formulas will be explored in future research.}.  Recall that the cup product formula in the nonnegative part $\mathcal D^{\geq 0}(kG, kG)$  in terms of the additive decomposition at the complex level was obtained in \cite[Section 7]{LZ2015}. Now, we describe the cup product formula in the negative part $\mathcal D^{<0}(kG,kG)$ as follows.  Let $X$ be the fixed set of representatives of conjugacy classes of elements of $G$. Recall that for any $z\in X$, we have fixed a right coset decomposition of $C_G(z)$ in $G$: $G=C_G(z)\gamma_{1, z}\cup \cdots \cup C_G(z)\gamma_{n_z, z}$. 
Let $\alpha_x:= (g_{1, s})\in k[\overline{C_G(x)}^{\times s}]=C_s(C_G(x), k)$ and $\alpha_y:=(h_{1, t}) \in k[\overline{C_G(y)}^{\times t}]=C_t(C_G(y), k)$ for $x, y\in X$.  By Theorem \ref{realization-Hochschild-homology}, we get that 
$$\iota_*(\alpha_x)=(g_s^{-1}\cdots g_{1}^{-1}x, g_{1, s})\in C_*(kG, kG);$$ $$\iota_*(\alpha_y)=(h_t^{-1}\cdots h_1^{-1}y, h_{1, t})\in C_*(kG, kG).$$ This yields $$\iota_*(\alpha_x)\cup \iota_*(\alpha_y)=\sum_{g\in G} (gh_t^{-1}\cdots h_1^{-1} y, h_{1, t}, g^{-1} g_{s}^{-1}\cdots g_1^{-1}x, g_{1,s})\in C_{s+t+1}(kG, kG).$$
Therefore, we have the following cup product formula: $$\rho_*(\iota_*(\alpha_x)\cup \iota_*(\alpha_y))=\sum_{z\in X}(\alpha_x\cup \alpha_y)_z\in \bigoplus_{z\in X} C_{s+t+1}(C_G(z), k),$$
where for a fixed $z\in X$,  
$$(\alpha_x\cup \alpha_y)_z=\sum_{g\in I_z} (k_{i_1}, \cdots, k_{i_{s+t+1}})\in C_{s+t+1}(C_G(z), k)$$
where $$I_z:=\{g\in G | h_1\cdots h_t g^{-1}x g h_t^{-1}\cdots h_1^{-1}y =\phi(g)z\phi(g)^{-1} \ \mbox{for some $\phi(g)\in G$}\}$$ and $k_{i_1}, \cdots, k_{i_{s+t+1}}\in\overline{C_G(z)}$ are uniquely determined by the following equations
$$\phi(g)\in C_G(z)\gamma_{i, z}, \quad  \gamma_{i, z}h_1=k_{i_1}\gamma_{s_i^1, z}, \quad  \gamma_{s_i^1, z}h_2=k_{i_2}\gamma_{s_i^2, z}, \quad \cdots, \quad \gamma_{s_i^{t-1}, z} h_t=k_{i_t} \gamma_{s_i^t, z}, $$
$$\gamma_{s_i^t, z}g^{-1}g_{s}^{-1}\cdots g_1^{-1} x=k_{i_{t+1}} \gamma_{s_i^{t+1}, z}, \quad \gamma_{s_i^{t+1}, z} g_1=k_{i_{t+2}} \gamma_{s_i^{t+2}, z},\quad \cdots, \quad \gamma_{s_i^{t+s}, z}g_s=k_{i_{t+s+1}}\gamma_{s_i^{t+s+1}, z}.$$
By  Remark \ref{duality-and-additive-decomposition} and   Equation (\ref{equation-cyclic-pairing}), it is not difficult to obtain the cup product formula for the other cases between $\mathcal D^{<0}(kG,kG)$ and $\mathcal D^{\geq 0}(kG, kG)$. The details are left  to the reader.  
\end{Rem}


\bigskip
\section{The $\widehat{\Delta}$-operator formula}\label{section-delta-operator}

We have defined the BV-operator $\widehat{\Delta}$ in $\widehat{\HH}^*(kG,kG)$ at the complex level (cf. Section 2). In this section, we determine the behavior of the operator $\widehat{\Delta}$ under the additive decomposition. There are three cases to be considered:

\medskip
{\it The first case:} $\widehat{\Delta}$ in $\widehat{\HH}^{> 0}(kG,kG)$. At the complex level, for any $n> 0$, $$\widehat{\Delta}: \Map(\overline{G}^{\times n},kG)\rightarrow \Map(\overline{G}^{\times (n-1)},kG)$$ maps any $\alpha: \overline{G}^{\times n}\rightarrow kG$ to $\widehat{\Delta}(\alpha):  \overline{G}^{\times (n-1)}\rightarrow kG$ such that
$$\widehat{\Delta}(\alpha)(g_{1,n-1})=\sum_{g_n\in G}\sum_{i=1}^n(-1)^{i(n-1)}\langle \alpha(g_{i,n},g_{1,i-1}),1\rangle g_n^{-1}.$$

{\it The second case:} $\widehat{\Delta}$ in $\widehat{\HH}^{\leq -1}(kG,kG)$. At the complex level, for any $n\leq -1$ (let $s=-n-1\geq 0$), $$\widehat{\Delta}: k[G\times\overline{G}^{\times s}]\rightarrow k[G\times\overline{G}^{\times {s+1}}]$$ is given by
$$\alpha=(g_0,g_{1,s})\mapsto \widehat{\Delta}(\alpha)=\sum_{i=0}^s(-1)^{is}(1,g_{i,s},g_{0,i-1}).$$

{\it The third case:} $\widehat{\Delta}: \widehat{\HH}^0(kG,kG)\rightarrow \widehat{\HH}^{-1}(kG,kG)$. In this case, $\widehat{\Delta}: kG\rightarrow kG$ is zero.

Since the last case is trivial, we deal with the first two cases. In the first case, $\widehat{\Delta}$ is the BV-operator $\Delta$ in the Hochschild cohomology $\HH^*(kG,kG)$, and its behavior under the additive decomposition has been determined in \cite{LZ2015}. Let us briefly recall the results there. As in Section 4, we fix a complete set $X$ of
representatives of the conjugacy classes in the finite group $G$.
For $x\in X$, $C_x=\{gxg^{-1}|g\in G\}$ is the conjugacy class
corresponding to $x$ and $C_G(x)=\{g\in G|gxg^{-1}=x\}$ is the
centralizer subgroup. For each $x\in X$,
$\mathcal{H}_x^*=\bigoplus_{n\geq 0}\mathcal{H}_x^n$ is a subcomplex
of the Hochschild cochain complex $C^*(kG, kG)=\mathcal{H}^*$, where
$$\mathcal{H}_x^n=\{\varphi: \overline{G}^{\times
n}\longrightarrow kG|\varphi(g_1,\cdots, g_{n})\in k[g_1\cdots
g_{n}C_x]\subset kG, \forall g_1,\cdots, g_n\in \overline{G}\}.$$

\begin{Lem}\label{restriction-Hochschild-cohomology}  $($\cite[Lemma 8.1]{LZ2015}$)$  For any $x\in X$ and $n\geq 1$, the BV-operator $\Delta: \mathcal{H}^n\longrightarrow
\mathcal{H}^{n-1}$ restricts to $\Delta_x:
\mathcal{H}_x^n\longrightarrow \mathcal{H}_x^{n-1}$.
\end{Lem}

We can define an operator $\widetilde{\Delta}_x$ by the following commutative diagram
$$\xymatrix{
\mathrm H^n(\mathcal{H}_x^{*})\ar[r]^{\Delta_x}\ar[d]_{\wr}  &
\mathrm H^{n-1}(\mathcal{H}_x^*)
\ar[d]_{\wr}\\
  \mathrm H^n(C_G(x),k)\ar[r]^{\widetilde{\Delta}_x}  &
\mathrm H^{n-1}(C_G(x),k),}$$
where the vertical
isomorphisms are given in Theorem \ref{realization-Hochschild-cohomology}.

\begin{Thm} \label{bigtriangleup-operator-Hochschild-cohomology}  $($\cite[Theorem 8.2]{LZ2015}$)$  Let
$\widetilde{\Delta}_x: \mathrm H^n(C_G(x),k)\longrightarrow
\mathrm H^{n-1}(C_G(x),k)$ be the map induced by the operator
$\Delta: \HH^n(kG,kG)\longrightarrow \HH^{n-1}(kG,kG)$. Then, at the complex level,
$\widetilde{\Delta}_x$ is defined as follows:
$$\widetilde{\Delta}_x(\psi)(h_{1,n-1})=\sum_{i=1}^n(-1)^{i(n-1)}\psi(h_{i, n-1}, h_{n-1}^{-1}\cdots h_1^{-1}x^{-1},
 h_{1,i-1})$$
 for $\psi:\overline{C_G(x)}^{\times
n}\longrightarrow k$ and for $h_1,\cdots, h_{n-1}\in
\overline{C_G(x)}$.

\end{Thm}

In the second case, $\widehat{\Delta}$ is the Connes' $B$-operator  in the Hochschild homology $\HH_*(kG,kG)$. Recall that for each $x\in X$,
$\mathcal{H}_{x,*}=\bigoplus_{s\geq 0}\mathcal{H}_{x,s}$ is a subcomplex
of the Hochschild chain complex $C_*(kG, kG)=\mathcal{H}_*$, where
$$\mathcal{H}_{x,s}=k[(g_s^{-1}\cdots g_1^{-1}u,g_{1,s})| u\in C_x, g_1,\cdots, g_s\in \overline{G}].$$

\begin{Lem}\label{restriction-Hochschild-homology}  For $x\in X$ and $s\geq 0$, the operator $B: \mathcal{H}_s\longrightarrow
\mathcal{H}_{s+1}$ restricts to $B_x:
\mathcal{H}_{x,s}\longrightarrow \mathcal{H}_{x,s+1}$.
\end{Lem}

\begin{proof} We need to show that $B(\alpha)\in \mathcal{H}_{x,s+1}$ for each $\alpha=(g_s^{-1}\cdots g_1^{-1}g_0^{-1}xg_0,g_{1,s})\in \mathcal{H}_{x,s}$, where $g_0\in G$, $g_1,\cdots, g_s\in \overline{G}$. This follows from the definition of the operator $B$:
$$B(g_s^{-1}\cdots g_1^{-1}g_0^{-1}xg_0,g_{1,s})=(1,g_s^{-1}\cdots g_1^{-1}g_0^{-1}xg_0, g_{1, s})+ \sum_{i=1}^s(-1)^{is}(1,g_{i,s},g_s^{-1}\cdots g_1^{-1}g_0^{-1}xg_0,g_{1,i-1})$$
and for each $0\leq i\leq s$, $$g_{i-1}^{-1}\cdots g_1^{-1}(g_s^{-1}\cdots g_1^{-1}g_0^{-1}xg_0)^{-1}g_s^{-1}\cdots g_i^{-1}\cdot (g_0\cdots g_{i-1})^{-1}x(g_0\cdots g_{i-1})=1.$$
\end{proof}

We can define an operator $\widetilde{B}_x$ by the following commutative diagram
$$\xymatrix{
H_s(\mathcal{H}_{x,*})\ar[r]^-{B_x}\ar[d]_{\wr}  &
H_{s+1}(\mathcal{H}_{x,*})
\ar[d]_{\wr}\\
  \mathrm H_s(C_G(x),k)\ar[r]^-{\widetilde{B}_x}  &
\mathrm H_{s+1}(C_G(x),k),}$$
where the vertical
isomorphisms are given in Theorem \ref{realization-Hochschild-homology}.

\begin{Thm} \label{bigtriangleup-operator-Hochschild-homology} Let
$\widetilde{B}_x: \mathrm H_s(C_G(x),k)\longrightarrow
\mathrm H_{s+1}(C_G(x),k)$ be the map induced by the operator
$B: \HH_s(kG,kG)\longrightarrow \HH_{s+1}(kG,kG)$. Then, at the complex level,
$\widetilde{B}_x$ is defined as follows:
$$\widetilde{B}_x(\gamma)=(1, h_{s}^{-1}\cdots h_1^{-1}x, h_{1, s})+\sum_{i=1}^s(-1)^{is}(1, h_{i,s},h_s^{-1}\cdots h_1^{-1}x,h_{1,i-1})$$
 for $\gamma=(h_{1,s})\in k[\overline{C_G(x)}^{\times s}]$.
\end{Thm}

\begin{proof} By Theorem \ref{realization-Hochschild-homology}, this is straightforward by chasing the above commutative diagram.
\end{proof}

Theorem \ref{theorem-A-infinity} shows that the natural inclusion  $\widehat{\iota}_{x=1}: \widehat{\mathrm H}^*(G, k)\hookrightarrow \widehat{\HH}^*(kG,kG)$ (cf. Remark \ref{remark-deformation-retract2}) is an inclusion of graded algebras. We now further prove that it is an inclusion of BV-algebras.

\begin{Cor} \label{BV-sub-algebra}  Let $k$ be a field and $G$ a finite group. Then  $\widehat{\iota}_{x=1}: \widehat{\mathrm H}^*(G, k)\hookrightarrow \widehat{\HH}^*(kG,kG)$  is an (unitary) embedding of BV-algebras.
\end{Cor}

\begin{proof} It follows from Theorem \ref{theorem-A-infinity} that the inclusion is an embedding of graded algebras. 
Theorems \ref{bigtriangleup-operator-Hochschild-cohomology} and \ref{bigtriangleup-operator-Hochschild-homology} show that this inclusion preserves the operator $\widehat{\Delta}$. Since this operator together with the cup product $\cup$ generates the Lie bracket $[\cdot, \cdot]$ on  $\widehat{\HH}^*(kG,kG)$,  we deduce that the Lie bracket $[\cdot, \cdot]$
restricts to $\widehat{\mathrm H}^*(G, k)=\widehat{\mathrm H}^*(C_G(1), k)$. This proves the corollary. 
\end{proof}

\begin{Rem}
We remark that $\widehat{\Delta}$ restricts to zero on  $\widehat{\mathrm H}^*(G,k)$ due to the fact that the Connes' $B$-operator is trivial in the group homology $\mathrm H_*(G,k)$. In Appendix A, we shall provide a proof of this non-trivial result, which is known and only implicit in the literature. 

Let $G$ be a finite abelian group. By Corollary \ref{cor-abelian-A}, we have an isomorphism of graded algebras
$$\widehat{\HH}^*(kG,kG)\simeq kG\otimes_k\widehat{\mathrm H}^*(G,k).$$ Since the Lie bracket on $\widehat{\HH}^*(kG,kG)$ is in general  nontrivial (see e.g.  \cite[Corollary 4.2]{LZ}) and the Lie bracket on $\widehat{\mathrm H}^*(G, k)$ is always trivial,   the above isomorphism is not an isomorphism of BV-algebras.
\end{Rem}

\begin{Rem} \label{BV-infinity-algebra} Notice that the restrictions of $\widehat{\Delta}$ to other summands $\widehat{\mathrm H}^*(C_G(x),k)$ (where $x\neq 1$) in the additive decomposition are non-trivial in general. Notice also that although the $\widehat{\Delta}$-operator is trivial on the Tate  cohomology $\widehat{\mathrm H}^*(G,k)$, it is not trivial at the complex level.  We conjecture that the Tate-Hochschild cochain complex $\mathcal{D}^*(kG, kG)$ is a BV$_{\infty}$-algebra and the Tate cochain complex $\widehat{C}^*(G, k)$ is a BV$_{\infty}$ subalgebra.  Equivalently, we conjecture that the operad of the frame little $2$-discs acts on $\mathcal D^*(kG, kG)$ and this action restricts to the subcomplex $\widehat{C}^*(G, k)$. 
\end{Rem}


\bigskip
Let us  consider  the stable Hochschild homology $\HH^{st}_{*}(kG, kG)$ which has been studied in \cite{HHKM2005} \cite{LiuZhouZimmermann2012}. 
From Remark \ref{duality-in-Tate-Hochschild-cohomology}, we have that $\HH^{st}_{m}(kG, kG)\cong \widehat{\HH}^{-m-1}(kG, kG)$ for $m\geq 0$.  Hence $\HH_*^{st}(kG, kG)$  is computed by the following truncated (at degree $-1$) complex of $\mathcal{D}^*(kG, kG)$, 
 $$\widetilde{C}_{*}(kG, kG): \quad \cdots\xrightarrow{\partial'_p} C_{p-1}(kG, kG) \xrightarrow{\partial'_{p-1}} \cdots \xrightarrow{\partial'_2}C_1(kG, kG) \xrightarrow{\partial'_1} \Ker(\tau) \rightarrow 0$$ 
where $\widetilde{C}_{-1}(kG, kG)=\Ker(\tau)$ and $\widetilde{C}_{-p-1}(kG, kG)=C_p(kG, kG)$ for $p>0$; and we recall that $\partial_p'=(-1)^{-p-1}\partial_p$ (cf. Remark \ref{remark-signs-change}).   Note that the restriction of the cup product $\cup$ to $\widetilde{C}_{*}(kG, kG)$  is strictly associative (since $m_3=0$ when restricted to $\mathcal D^{<0}(kG, kG)$) and compatible with the differential $\partial'$. 
\begin{Rem}\label{remark-restriction-negative}
In general, the restriction of $\cup$ to   the whole negative part $\mathcal D^{<0}(kG, kG)=C_*(kG, kG)$, is not compatible with $\partial'$: 
For $g_0, h_0\in D^{-1}(kG, kG)=C_0(kG, kG), $ we have that  $g_0\cup h_0\in C_1(kG, kG)$ and $$\partial'_1(g_0\cup h_0)=\sum_{g\in G} \partial_1((gh_0, g^{-1} g_0))=\sum_{g\in G} (gh_0g^{-1}g_0-g^{-1}g_0gh_0),$$ which is not zero in general,  but $\partial'_0(g_0)=0=\partial'_0(h_0)$ in $C_*(kG, kG)$. Hence $\cup$ is not well-defined on the whole $H^{-*-1}(\mathcal D^{<0}(kG, kG))=\HH_*(kG, kG)$, but  it is well-defined on the subspace $\HH^{st}_*(kG, kG)\subset \HH_*(kG, kG)$. 
\end{Rem}

Analogously, let us denote by $\widetilde{C}_{*}(G, k)$  the truncated (at degree $-1$) complex  of the Tate cochain complex $\widehat{C}^*(G, k)$. The cohomology of this complex is denoted by $\mathrm{H}^{st}_{-*-1}(G, k)$, namely $\mathrm H^{st}_m(G, k)=H^{-m-1}(\widetilde C_{*}(G, k))$ for $m\geq 0$.   Then the  homotopy deformation retract in Remark \ref{remark-deformation-retract2}
 induces the following  additive decomposition
$$\HH^{st}_{*}(kG, kG)\simeq \bigoplus_{x\in X}\mathrm H^{st}_{*}(C_G(x),k).$$ As a consequence, we have the following result. 

\begin{Thm}\label{theorem-stable-Hochschild}
The stable Hochschild homology $\HH_{-*-1}^{st}(kG, kG)$, equipped with the Connes' $B$-operator and the cup product $\cup$, is a BV-algebra (without unit). Moreover, $\mathrm H^{st}_{-*-1}(G, k)$ is a BV subalgebra of $\HH_{-*-1}^{st}(kG, kG)$. \end{Thm}
\begin{proof}
This follows from Corollary \ref{BV-sub-algebra}. \end{proof}

Denote by $BG$ the classifying space of a finite group $G$. There is a well-known isomorphism between the Hochschild homology $\HH_*(kG, kG)$ and the singular homology $\mathrm H_*(LBG, k)$ of the free loop space $LBG:=\Map(S^1, BG)$ of $BG$ (cf. \cite[7.3.13 Corollary]{Lod}).  Under this isomorphism, the Connes' $B$-operator on $\HH_*(kG, kG)$ corresponds to the $S^1$-action  on $\mathrm H_*(LBG,k)$ (cf. \cite{Lod}).  We denote by $\mathrm H^{st}_*(LBG, k)$ the subspace of $\mathrm H_*(LBG, k)$ corresponding to $\HH_*^{st}(kG, kG)$  under the above isomorphism. Transferring the cup product on $\HH_*^{st}(kG, kG)$ to  $\mathrm H^{st}_*(LBG, k)$, we obtain the following result. 

\begin{Cor}\label{cor-bv-free-loop}
Let $G$ be a finite group and $k$ be a field. Then   $\mathrm H^{st}_{-*-1}(LBG, k)$ equipped with the  $S^1$-action  and the transferred product, is a BV-algebra (without unit).  
\end{Cor}
\begin{proof}
This follows from  Theorem \ref{theorem-stable-Hochschild} and the above analysis.  
\end{proof}

\begin{Rem}
Clearly, $\mathrm H_m^{st}(LBG, k)=\mathrm H_m(LBG, k)$ for $m>0$ and $\HH_0^{st}(kG, kG)\cong \mathrm H_0^{st}(LBG, k)\subset \mathrm H_0(LBG, k)$. It would be interesting to give a topological construction of the transferred  product on $\mathrm H_{-*-1}^{st}(LBG, k)$. 
\end{Rem}

\section{The symmetric group of degree $3$}

In this section, we use our results to compute the BV structure of the Tate-Hochschild cohomology
for symmetric group of degree $3$ over a field $k$ of characteristic $3$. For convenience, we write the BV-operator $\widehat{\Delta}$ in $\widehat{\HH}^*(kG, kG)$ as $\Delta$ in this section. 

Recall that in a BV-algebra, there is the following equation (see
\cite{Getzler1994}; here we have changed the original equation
according to the sign convention in Remark \ref{sign-convention} and
 we write   $\delta\gamma$ instead of $\delta\cup \gamma$):
$$\Delta(\alpha\beta\gamma)
=(-1)^{|\alpha||\beta||\gamma|}[(-1)^{|\gamma|}\Delta(\alpha\beta)\gamma
+ \alpha\Delta(\beta\gamma) +
(-1)^{|\alpha||\beta|}\beta\Delta(\alpha\gamma)$$
$$ - (-1)^{|\alpha|}
\Delta(\alpha)\beta\gamma -
(-1)^{|\alpha|+|\beta|-|\alpha||\gamma|}\alpha(\Delta(\beta))\gamma
-
(-1)^{|\alpha|+|\beta|+|\gamma|}\alpha\beta\Delta(\gamma)],$$
where $\alpha, \beta, \gamma$ are homogeneous elements.
So in order to compute the $\Delta$-operator in $\widehat{\HH}^*(kG,kG)$,
it suffices to find the value of $\Delta$ on each generator
and the value of $\Delta$ on the cup product of every two generators. Also recall that we can use the cup product formula, the
$\Delta$-operator formula and the following formulas to
compute the Lie bracket in a BV-algebra:
$$[\alpha,\beta]=-(-1)^{(|\alpha|-1)|\beta|}(\Delta(\alpha\cup\beta)-\Delta(\alpha)\cup\beta-(-1)^{|\alpha|}\alpha\cup\Delta(\beta)),$$
$$[\alpha,\beta]=-(-1)^{(|\alpha|-1)(|\beta|-1)}[\beta,\alpha].$$

Notice that the associative algebra structure of the positive part $\widehat{\HH}^{\geq 0}(kS_3,kS_3)$ has been determined by Siegel and
Witherspoon in \cite{SW1999} and the associative algebra structure of the whole algebra $\widehat{\HH}^*(kG,kG)$ has been determined by Nguyen in \cite{Nguyen2012}. Moreover, the $\Delta$ operator and the Lie bracket of the positive part $\widehat{\HH}^{\geq 0}(kS_3,kS_3)$ has been computed by the first and the third named authors in \cite{LZ2015}.

Let
$G=S_3=\langle a,b | a^3=1=b^2, bab=a^{-1}\rangle$. Choose the
conjugacy class representatives as $1,a,b$. The corresponding
centralizers are $H_1=G, H_2=\langle a\rangle$ and $H_3=\langle
b\rangle$. So $\widehat{\HH}^*(kS_3)\simeq \widehat{H}^*(S_3)\oplus
\widehat{\mathrm H}^*(\langle a\rangle)\oplus \widehat{\mathrm H}^*(\langle b\rangle)$. The algebra
structures of $\widehat{\mathrm H}^*(S_3)$, of $\widehat{\mathrm H}^*(\langle a\rangle)$, and of
$\widehat{\mathrm H}^*(\langle b\rangle)$ are known (see \cite{Nguyen2012}). $\widehat{\mathrm H}^*(S_3)$ is of the form $k[x]/(x^2)\otimes_kk[z,z^{-1}]$, where $x,z,z^{-1}$
are of degrees $3,4,-4$, respectively, subject to the
graded-commutative relations. $\mathrm H^*(\langle
a\rangle)$ is of the form $k[w_1]/({w_1}^2)\otimes_kk[w_2,w_2^{-1}]$, where $w_1,w_2,w_2^{-1}$
are of degrees $1,2,-2$, respectively, subject to the
graded-commutative relations. $\widehat{\mathrm H}^*(\langle
b\rangle)=0$, since $k\langle b\rangle$ is
semisimple. Identify the elements $x, z$ with their images under $\gamma_1$ in
$\widehat{\HH}^*(kG,kG)$ and denote by $W_1, W_2, W_1^{-1}, W_2^{-1}$ the images of the
elements (resp.) $w_1, w_2, w_1^{-1}, w_2^{-1}$ under $\gamma_2$, and put $E_i:=\gamma_i(1)$ $(i=1,2)$ and $C:=E_2+E_1=E_2+1$. Then
Nguyen proved in \cite{Nguyen2012} the following
presentation for the Tate-Hochschild cohomology algebra
$\widehat{\HH}^*(kG,kG)$: it is generated as an
algebra by elements $x,z,z^{-1},C,W_1,W_2$, and $W_2^{-1}$ of
degrees (resp.) $3,4,-4,0,1,2,$ and $-2$, subject to the relations
$$xW_1=0,\quad xW_2=zW_1,\quad z^{-1}W_1=(xz^{-1})W_2^{-1},$$
$$
C^2=CW_2^{-1}=CW_i=0\quad (i=1,2),$$
$$
W_2^2=zC,\quad W_2^{-2}=z^{-1}C,\quad
W_1W_2=xC,\quad W_1W_2^{-1}=xz^{-1}C,$$ together with the graded commutative relations. Observe that although both $w_2^3$ and $w_2^{-3}$ are nonzero in $\widehat{\mathrm H}^*(\langle a\rangle)$, we have that $W_2^3=W_2^2W_2=zCW_2=0$ and $W_2^{-3}=W_2^{-2}W_2^{-1}=z^{-1}CW_2^{-1}=0$ in $\widehat{\HH}^*(kG,kG)$. Moreover, $w_2w_2^{-1}=1$ in $\widehat{\mathrm H}^*(\langle a\rangle)$ but $W_2W_2^{-1}=C\neq 1$ in $\widehat{\HH}^*(kG,kG)$.

By Section 7, the
operator $\Delta: \widehat{\HH}^n(kS_3)\longrightarrow
\widehat{\HH}^{n-1}(kS_3)$ restricts to the operators
$\widehat{\Delta}_b: \widehat{\mathrm H}^n(\langle b\rangle)\longrightarrow
\widehat{\mathrm H}^{n-1}(\langle b\rangle)$, $\widehat{\Delta}_a: \widehat{\mathrm H}^n(\langle
a\rangle)\longrightarrow \widehat{\mathrm H}^{n-1}(\langle a\rangle)$, and
$\widehat{\Delta}_1: \widehat{\mathrm H}^n(S_3)\longrightarrow \widehat{\mathrm H}^{n-1}(S_3)$.
Both $\widehat{\Delta}_1$ and $\widehat{\Delta}_b$ are zero maps  and we only need to consider $\widehat{\Delta}_a: \widehat{\mathrm H}^n(\langle
a\rangle)\longrightarrow \widehat{\mathrm H}^{n-1}(\langle a\rangle)$.    In \cite{LZ2015}, we have computed $\widehat{\Delta}_a$ for the positive part $\widehat{\mathrm H}^{> 0}(\langle
a\rangle)$ up to degree $4$: $\widehat{\Delta}_a(w_2^2)=0$, $\widehat{\Delta}_a(w_1w_2)=-w_2$, $\widehat{\Delta}_a(w_2)=0$, $\widehat{\Delta}_a(w_1)=-1$. By the duality mentioned in Remark \ref{duality-in-Tate-Hochschild-cohomology}, we get the values of $\widehat{\Delta}_a$ for the negative part $\widehat{\mathrm H}^{< 0}(\langle
a\rangle)$ up to degree $-4$: $\widehat{\Delta}_a(w_1w_2^{-1})=-w_2^{-1}$, $\widehat{\Delta}_a(w_2^{-1})=0$, $\widehat{\Delta}_a(w_1w_2^{-2})=-w_2^{-2}$, $\widehat{\Delta}_a(w_2^{-2})=0$. Moreover, in degree $0$, we have $\widehat{\Delta}_a(1)=0$. From these results we can compute the values of $\Delta$ on the elements of degrees between $4$ and $-4$ in $\widehat{\HH}^*(kS_3)$. For example, $\Delta(x)=0$ since $x\in \widehat{\mathrm H}^*(S_3)$ and $\widehat{\Delta}_1$ is trivial; $\Delta(W_1W_2)=-W_2$, the reason is as follows: under the additive decomposition, $W_1W_2$ corresponds to the element $x+w_1w_2$, $\widehat{\Delta}_1(x)=0$, $\widehat{\Delta}_a(w_1w_2)=-w_2$; $\Delta(W_2^{-1})=0$, the reason is as follows: $W_2^{-1}$ is an element of degree $-2$, under the additive decomposition, it corresponds to an element $\lambda w_2^{-1}$ with some $\lambda\in k$, but $\widehat{\Delta}_a(w_2^{-1})=0$; $\Delta(W_1)=-E_2=1-C$ (here $1=E_1$ denotes the unit element of $\widehat{\HH}^*(kS_3)$) since $\widehat{\Delta}_a(w_1)=-1$ (here $1$ denotes the unit element of $\widehat{\mathrm H}^*(\langle
a\rangle)$); etc.

We now compute the Lie brackets. Since we have the following Poisson rule: $[\alpha\cup
\beta, \gamma]=[\alpha, \gamma]\cup \beta +
(-1)^{|\alpha|(|\gamma|-1)}\alpha\cup[\beta, \gamma]$, it suffices
to write down the Lie brackets between generators in
$\widehat{\HH}^*(kS_3)$. There are $49$ cases, we list them explicitly. In the following computations, we shall freely use the three formulas mentioned at the beginning of this section.

(1) $[x,x]=0$. (2) $[x,z]=0$. (3) $[z,x]=0$. (4) $[x,z^{-1}]=0$. (5) $[z^{-1},x]=0$.

The cases (1) to (5) can be seen from the facts that $\widehat{\Delta}_1$ and therefore the Lie brackets are trivial over $\widehat{\mathrm H}^*(S_3)$.

(6) $[x,C]=-(\Delta(xC)-\Delta(x)C+x\Delta(C))=-\Delta(W_1W_2)=W_2$, since $\Delta(x)=\Delta(C)=0$.

(7) $[C,x]=-[x,C]=-W_2$.

(8) $[x,W_1]=-(\Delta(xW_1)-\Delta(x)W_1+x\Delta(W_1))=-x(-E_2)=x(C-1)$, since $xW_1=0, \Delta(x)=0$ and $\Delta(W_1)=-E_2$.

(9) $[W_1,x]=-[x,W_1]=x(1-C).$

(10) $[x,W_2]=-(\Delta(xW_2)-\Delta(x)W_2+x\Delta(W_2))=-\Delta(xW_2)=0$. In the last step, we use the fact that $\Delta(xW_2)=0$. The reason is as follows:  $xW_2$ is an element of degree $5$,
under the additive decomposition, it corresponds to an element in
$\mathrm H^*(\langle a\rangle)$ and has the form $\lambda w_1{w_2}^2$ with some $\lambda\in k$. Using the formula for $\Delta$-operator it is easy to show that $\widehat{\Delta}_a(w_1{w_2}^2)=0.$

(11) $[W_2,x]=0.$

(12) $[x,W_2^{-1}]=-(\Delta(xW_2^{-1})-\Delta(x)W_2^{-1}+x\Delta(W_2^{-1}))=-\Delta(xW_2^{-1})=-\Delta(W_1)=C-1$, since $xW_2^{-1}=xzx^{-1}z^{-1}W_1=zxx^{-1}z^{-1}W_1=W_1$.

(13) $[W_2^{-1},x]=-[x,W_2^{-1}]=1-C$.

(14) $[x,x]=0$. (15) $[x,z]=0$. (16) $[z,x]=0$. (Since the Lie brackets are trivial over $\widehat{\mathrm H}^*(S_3)$.)

(17) $[z,C]=-(\Delta(zC)-\Delta(z)C-z\Delta(C))=-\Delta(zC)=-\Delta(W_2^2)=0$.

(18) $[C,z]=0$.

(19) $[z,W_1]=\Delta(zW_1)-\Delta(z)W_1-z\Delta(W_1)=\Delta(xW_2)-z(1-C)=-z(1-C)=z(C-1)$, since in the above we have computed that $\Delta(xW_2)=0$.

(20) $[W_1,z]=-[z,W_1]=z(1-C)$.

(21) $[z,W_2]=-(\Delta(zW_2)-\Delta(z)W_2-z\Delta(W_2))=-\Delta(zW_2)=0$. The reason for the last step is as follows: $zW_2$ is an element of degree $6$,
under the additive decomposition, it corresponds to an element in
$\mathrm H^*(\langle a\rangle)$ and has the form $\lambda {w_2}^3$ with some $\lambda\in k$. Using the formula for $\Delta$-operator it is easy to show that $\widehat{\Delta}_a({w_2}^3)=0.$

(22) $[W_2,z]=0.$

(23) $[z,W_2^{-1}]=-(\Delta(zW_2^{-1})-\Delta(z)W_2^{-1}-z\Delta(W_2^{-1}))=0$. Notice that $zW_2^{-1}$ is an element of degree $2$, under the additive decomposition, it corresponds to an element in $\mathrm H^*(\langle a\rangle)$ and has the form $\lambda {w_2}$ with some $\lambda\in k$. However, $\widehat{\Delta}_a({w_2})=0$.

(24) $[W_2^{-1},z]=0$.

(25) $[z^{-1},z^{-1}]=0$.

(26) $[z^{-1},C]=-(\Delta(z^{-1}C)-\Delta(z^{-1})C-z^{-1}\Delta(C))=0$. The reason is as follows: $z^{-1}C=W_2^{-2}$ is an element of degree $-4$, under the additive decomposition, it corresponds to an element $z^{-1}+\lambda w_2^{-2}$ with some $\lambda\in k$, $\widehat{\Delta}_1(z^{-1})=0$, $\widehat{\Delta}_a(w_2^{-2})=0$.

(27) $[C,z^{-1}]=0$.

(28) $[z^{-1},W_1]=\Delta(z^{-1}W_1)-\Delta(z^{-1})W_1-z^{-1}\Delta(W_1)=-z^{-1}\Delta(W_1)=z^{-1}(C-1)$. Notice that $\Delta(z^{-1}W_1)=0$, since $z^{-1}W_1$ is an element of degree $-3$, under the additive decomposition, it corresponds to the element $w_1w_2^{-2}$ and $\widehat{\Delta}_a(w_1w_2^{-2})=0$.

(29) $[W_1,z^{-1}]=-[z^{-1},W_1]=z^{-1}(1-C)$.

(30) $[z^{-1},W_2]=-(\Delta(z^{-1}W_2)-\Delta(z^{-1})W_2-z^{-1}\Delta(W_2))=-\Delta(z^{-1}W_2)=0$. The reason for the last step is as follows: $z^{-1}W_2$ is an element of degree $-2$, under the additive decomposition, it corresponds to an element $\lambda w_2^{-1}$  with some $\lambda\in k$, but $\widehat{\Delta}_a(w_2^{-1})=0$.

(31) $[W_2,z^{-1}]=0$.

(32) $[z^{-1},W_2^{-1}]=-(\Delta(z^{-1}W_2^{-1})-\Delta(z^{-1})W_2^{-1}-z^{-1}\Delta(W_2^{-1}))=-\Delta(z^{-1}W_2^{-1})=0$. The reason for the last step is as follows: $z^{-1}W_2^{-1}$ is an element of degree $-6$, under the additive decomposition, it corresponds to an element $\lambda w_2^{-3}$  with some $\lambda\in k$, but $\widehat{\Delta}_a(w_2^{-3})=0$.

(33) $[W_2,z^{-1}]=0$.

(34) $[C,C]=-(\Delta(C^2)-\Delta(C)C-C\Delta(C))=0$.

(35) $[C,W_1]=-(\Delta(CW_1)-\Delta(C)W_1-C\Delta(W_1))=C\Delta(W_1)=C(1-C)=C$.

(36) $[W_1,C]=[C,W_1]=-C$.

(37) $[C,W_2]=-(\Delta(CW_2)-\Delta(C)W_2-C\Delta(W_2))=0$.

(38) $[W_2,C]=0$.

(39) $[C,W_2^{-1}]=-(\Delta(CW_2^{-1})-\Delta(C)W_2^{-1}-C\Delta(W_2^{-1}))=0$.

(40) $[W_2^{-1},C]=0$.

(41) $[W_1,W_1]=-[W_1,W_1]=0$.

(42) $[W_1,W_2]=-(\Delta(W_1W_2)-\Delta(W_1)W_2+W_1\Delta(W_2))=W_2+(1-C)W_2=-W_2$.

(43) $[W_2,W_1]=-[W_1,W_2]=W_2$.

(44) $[W_1,W_2^{-1}]=-(\Delta(W_1W_2^{-1})-\Delta(W_1)W_2^{-1}+W_1\Delta(W_2^{-1}))=-\Delta(z^{-1}x)+(1-C)W_2^{-1}=W_2^{-1}$.

(45) $[W_2^{-1},W_1]=-[W_1,W_2^{-1}]=-W_2^{-1}$.

(46) $[W_2,W_2]=-(\Delta(W_2^{2})-\Delta(W_2)W_2-W_2\Delta(W_2))=0$. The reason that $\Delta(W_2^{2})=0$ is as follows: $W_2^{2}=zC$ is an element of degree $4$, under the additive decomposition, it corresponds to an element $z+\lambda w_2^{2}$ with some $\lambda\in k$, $\widehat{\Delta}_1(z)=0$, $\widehat{\Delta}_a(w_2^{2})=0$.

(47) $[W_2,W_2^{-1}]=-(\Delta(W_2W_2^{-1})-\Delta(W_2)W_2^{-1}-W_2\Delta(W_2^{-1}))=-\Delta(W_2W_2^{-1})=0$, since $W_2W_2^{-1}$ is an element of degree $0$.

(48) $[W_2^{-1},W_2]=0$.

(49) $[W_2^{-1},W_2^{-1}]=-(\Delta(W_2^{-2})-\Delta(W_2^{-1})W_2^{-1}-W_2^{-1}\Delta(W_2^{-1}))=0$. The reason that $\Delta(W_2^{-2})=0$ is as follows: $W_2^{-2}=z^{-1}C$ is an element of degree $-4$, under the additive decomposition, it corresponds to an element $z^{-1}+\lambda w_2^{-2}$ with some $\lambda\in k$, $\widehat{\Delta}_1(z^{-1})=0$, $\widehat{\Delta}_a(w_2^{-2})=0$.

\begin{Rem} In \cite{LZ2015}, our computations for Lie brackets contain some minor errors, which are caused by the same reason:   The equality $\widehat{\Delta}_a(w_1)=-1$  on $\mathrm H^*(\langle a\rangle)$ should correspond to the equality $\Delta(X_1)=-E_2=1-C_1$ in $\HH^*(kS_3)$, but we   used $\Delta(X_1)=-1$  in \cite{LZ2015}.  We list all the corrections in \cite{LZ2015} as follows: 
$$[u,X_1]=u(C_1-1)=-[X_1, u],\quad [v,X_1]=v(1-C_1)=-[X_1, v];$$
$$[C_1,X_1]=C_1(C_1-1)=-[X_1, C_1],\quad [C_2,X_1]=C_2(C_1-1)=-[X_1, C_2].$$
 But the following phenomena in our examples is still true both in Hochschild cohomology and in Tate-Hochschild cohomology: The Lie bracket of any two generators in even degrees vanishes.
\end{Rem}

\bigskip


\appendix

\section{A proof of $B=0$ in $\mathrm H_*(G, k)$}

In this appendix, we denote $kG$ by $A$ and use the unnormalized bar resolution $\Barr_*(A)$ of $A$. Recall that the cyclic bicomplex $\{CC_{p, q}(A)\mid
p, q\geq 0\}$ is the  double complex (cf. \cite[Section 2.1]{Lod}):
$$\xymatrix@R=1.3pc{
\ar[d]^{b} & \ar[d]^{-b'}  & \ar[d]^{b} \\
A^{\otimes 3}\ar[d]^{b} & \ar[l]_{1-t} \ar[d]^{-b'} A^{\otimes 3}  & A^{\otimes 3} \ar[l]_{N} \ar[d]^{b} & \ar[l]_{1-t}\\
A^{\otimes 2}\ar[d]^{b} & \ar[l]_{1-t} \ar[d]^{-b'} A^{\otimes 2}  & A^{\otimes 2} \ar[l]_{N} \ar[d]^{b} & \ar[l]_{1-t}\\
A   & \ar[l]_{1-t}   A  & A \ar[l]_{N}  & \ar[l]_{1-t}\\
}$$
where 
\begin{itemize}
\item $CC_{p, q}(A)=A^{\otimes (q+1)}$, 
\item $b'(a_0\otimes \cdots \otimes a_p)=\sum_{i=0}^{p-1} (-1)^i a_0\otimes \cdots \otimes a_ia_{i+1} \otimes \cdots \otimes a_p,$
\item $b(a_0\otimes \cdots \otimes a_p)=b'(a_0\otimes \cdots \otimes a_p)+(-1)^pa_pa_0\otimes a_1\otimes \cdots \otimes a_{p-1},$
\item $t_p(a_0\otimes \cdots \otimes a_p)=(-1)^{p} a_p\otimes a_0\otimes \cdots \otimes a_{p-1},$
\item $N_p=\sum_{i=0}^{p} t^i$, where $t^i$ denotes the $i$-th power of the map $t$. 
\end{itemize}
It is well known (cf.,  e.g.,  \cite[Section 2.1]{Wei}) that the cyclic homology $\HC_*(A)$ of $A$ is defined as the homology of the total
complex of $CC_{*, *}(A)$. Note that the odd columns of $CC_{*, *}(A)$ are exact, as they are exactly  $\Barr_*(A)[1]$  (the shift $[1]$ of the bar resolution $\Barr_*(A)$),  and the
even columns are the Hochschild chain complex $C_*(A, A)$. We
define, for $p\geq 0$, a map $s: A^{\otimes (p+1)}\rightarrow
A^{\otimes (p+2)}$ by
$$s(a_0\otimes \cdots \otimes a_p)=1\otimes a_0\otimes \cdots \otimes a_p.$$
Now the Connes' $B$-operator is defined
to be $B=(1-t)sN:
A^{\otimes (p+1)}\rightarrow A^{\otimes (p+2)}.$ Since it can be shown
that $B^2=Bb+bB=0$,  $B$ induces a map $\HH_p(A)\rightarrow
\HH_{p+1}(A)$, still denoted by $B$. Note that if we use the
normalized Hochschild chain complex as in the main text of the  present paper, then
$B$ is equal to $sN$ at the complex level. Consider the following
short exact sequence of double complexes:
$$0\rightarrow CC_{<2,  *}(A) \xrightarrow{I} CC_{*, *}(A)  \xrightarrow{S} CC_{*,*}(A)[2]\rightarrow 0$$
where $CC_{<2, *}(A)$ is the double complex formed by the first two
columns of $CC_{*,*}(A)$, $I$ is the natural embedding, $S$ is the
quotient map, and where for the double complex $CC_{*, *}(A)[2]$,
$$(CC_{*, *}(A)[2])_{p, q}=CC_{p-2,q}(A).$$
 By a standard fact (Killing contractible complexes) from homological algebra (see \cite[2.1.6 Lemma]{Lod}), $C_{<2, *}(A)$ has a total complex which is quasi-isomorphic to the Hochschild chain complex $C_*(A, A)$. Thus we get a long exact sequence
$$\cdots \rightarrow \HH_n(A, A)\xrightarrow{I} \HC_n(A) \xrightarrow{S} \HC_{n-2}(A) \xrightarrow{B} \HH_{n-1}(A, A)\rightarrow \cdots.$$
Note that the composition of the maps (still denoted by $B$)
$$\HH_{n}(A, A) \xrightarrow{I} \HC_{n}(A) \xrightarrow{B} \HH_{n+1}(A, A)$$
is exactly the Connes' $B$-operator (cf. \cite[Exercise 9.8.2]{Wei}).

Now we consider another double complex $\{C_{p, q}(G)\mid
p, q\geq 0\}$:
$$\xymatrix@R=1.4pc{
\ar[d]^{\partial} & \ar[d]^{-{\partial}'}  & \ar[d]^{\partial} \\
A^{\otimes 2}\ar[d]^{\partial} & \ar[l]_{1-t'} \ar[d]^{-{\partial}'} A^{\otimes 2}  & A^{\otimes 2} \ar[l]_{N'} \ar[d]^{\partial} & \ar[l]_{1-t'}\\
A\ar[d]^{\partial} & \ar[l]_{1-t'} \ar[d]^{-{\partial}'} A  & A \ar[l]_{N'} \ar[d]^{\partial} & \ar[l]_{1-t'}\\
k   & \ar[l]_{1-t'}   k  & k \ar[l]_{N}  & \ar[l]_{1-t'},\\
}$$
where \begin{itemize}
\item $C_{p, q}(G)=A^{\otimes q}$,
\item  $\partial'(g_1\otimes \cdots \otimes g_n)=g_2\otimes g_{3, n}+\sum_{i=1}^{n-1} (-1)^i g_{1, i-1}\otimes g_ig_{i+1}\otimes g_{i+2, n},$
\item  $\partial(g_1\otimes \cdots \otimes g_n)=\partial'(g_1\otimes \cdots\otimes g_n)+(-1)^ng_{1,n-1},$
 \item $t'_n(g_1\otimes \cdots \otimes g_{n-1})=(-1)^{n-1}
g_{n-1}^{-1}\cdots g_1^{-1}\otimes g_1\otimes \cdots \otimes g_{n-2},$
  \item $N'_n=\sum_{i=1}^{n} t'^i$.\end{itemize}
   The odd columns of ${C}_{*, *}(G)$ are
exact since they are  isomorphic to $(k\otimes_{kG} \mathrm{Bar}_*(kG))[1]$,  the projective resolution (shifted by $[1]$) of the trivial module $k$ as right $A$-modules.
The even columns are obtained by shift $[1]$ on the group homology chain complex $C_*(G,k)$.
 Here we adapt the notation from Karoubi \cite{Kar} and denote by
$\HC_*(G)$ the homology of the total complex of $C_{*, *}(G)$.
Similarly, the double complex $C_{<2,*}(G)$ has a total complex whose
$n$-th homology is naturally isomorphic to the group homology
$\mathrm H_n(G, k)$. Observe that there is a split injection (which is clearly compatible with the additive decomposition map in Theorem \ref{realization-Hochschild-homology}):
$$
\begin{array}{rcl}
\iota: C_{*, *}(G) &\hookrightarrow &CC_{*,*}(A),\\
A^{\otimes n}&\longrightarrow&
A\otimes A^{\otimes n},\\
g_{1,n}& \longmapsto& (g_n^{-1}\cdots g_1^{-1},g_{1,n}),
\end{array}
$$
whose retraction is given by
$$
\begin{array}{rcl}
r: CC_{*, *}(A) &\twoheadrightarrow &C_{*, *}(G),\\
A\otimes A^{\otimes n}&\longrightarrow&
A^{\otimes n},\\
g_{0,n}& \longmapsto & g_{1,n}\mbox{ if } g_0g_1\cdots g_n=1,\\
\mbox{ or }\ g_{0,n}& \longmapsto& 0 \quad \mbox{ otherwise}.
\end{array}
$$
Thus we have the following commutative diagram between long exact sequences:
$$\xymatrix@R=15px{
\cdots \ar[r] & \HH_n(A, A) \ar[r]^-{I} & \HC_n(A) \ar[r]^-{S} & \HC_{n-2}(A) \ar[r]^B & \HH_{n-1}(A, A) \ar[r] & \cdots\\
\cdots \ar[r] & H_n(G, k) \ar[r]^-{I}\ar@{_(->}[u]& \HC_n(G) \ar[r]^-{S}\ar@{_(->}[u] & \HC_{n-2}(G) \ar[r]^B \ar@{_(->}[u]& \mathrm H_{n-1}(G, k) \ar@{_(->}[u]\ar[r] & \cdots\\
}$$ Note that the restriction of the Connes' $B$-operator to $\mathrm H_{*}(G, k)$ is the
composition of maps
$$\mathrm H_{n}(G, k)\xrightarrow{I}\HC_{n}(G) \xrightarrow{B} \mathrm H_{n+1}(G, k).$$ So in order to prove that the restriction of the Connes' $B$-operator vanishes,
it is sufficient to prove that $B: \HC_{n}(G)\rightarrow \mathrm H_{n+1}(G,
k)$ vanishes for $n\geq 0$. The following result is due to Karoubi.
\begin{Thm}[\cite{Kar}]\label{kar}
We have $$\HC_n(G)\simeq \bigoplus_{i\in \Z_{\geq 0}} \mathrm H_{n-2i}(G,
k)$$ and the map $S: \HC_n(G)\rightarrow \HC_{n-2}(G)$ is a projection with kernel
$\mathrm H_n(G, k)$.
\end{Thm}
As a consequence, the long exact sequence in the second row of the above commutative diagram splits as short exact sequences
$$0\rightarrow \mathrm H_n(G, k) \xrightarrow{I} \HC_n(G) \xrightarrow{S} \HC_{n-2}(G) \rightarrow 0.$$ In particular, the map $B: \HC_{n}(G)\rightarrow\mathrm  H_{n+1}(G, k)$ is zero.

For the reader's convenience, we include a proof of Theorem \ref{kar}.
\begin{proof}[Proof of Theorem \ref{kar}]
Note that $C_{*, *}(G)\simeq k\otimes_{kG} \widetilde{C}_{*, *}(G)$, where the double complex $\widetilde{C}_{*, *}(G)$ is defined as follows,
$$\xymatrix@R=1.3pc{
\ar[d]^{\widetilde b} & \ar[d]^{-b'}  & \ar[d]^{\widetilde b} \\
A^{\otimes 3}\ar[d]^{\widetilde b} & \ar[l]_{1-\widetilde{t}} \ar[d]^{-b'} A^{\otimes 3}  & A^{\otimes 3} \ar[l]_{\widetilde{N}} \ar[d]^{\widetilde b} & \ar[l]_{1-\widetilde{t}}\\
A^{\otimes 2}\ar[d]^{\widetilde b} & \ar[l]_{1-\widetilde{t}} \ar[d]^{-b'} A^{\otimes 2}  & A^{\otimes 2} \ar[l]_{\widetilde{N}} \ar[d]^{\widetilde b} & \ar[l]_{1-\widetilde{t}}\\
A   & \ar[l]_{1-\widetilde{t}}   A  & A \ar[l]_{\widetilde{N}}  & \ar[l]_{1-\widetilde{t}}
}$$
where $$\widetilde{t}_n(g_0\otimes \cdots\otimes g_{n-1})=(-1)^{n-1} g_0\cdots g_{n-1}\otimes (g_1\cdots g_{n-1})^{-1}\otimes g_1\otimes \cdots \otimes g_{n-2}$$ and $\widetilde{N}_n=\sum_{i=1}^n \widetilde{t}^i$; the odd columns are $\Barr_*(A)[1]$; the even columns are the complex $\widetilde{\Barr}(A)\otimes_Ak$ with the differential $\widetilde{b}=b'\otimes_A \id_k$, where   $\widetilde{\Barr}(A)$ is the deleted bar resolution.
Thus we have  $$\HC_*(G)=H_*(k\otimes_{kG} \Tot
(\widetilde{C}_{*, *}(G))\simeq H_*(k\otimes_{kG}^{\mathbb{L}} \Tot
(\widetilde{C}_{*, *}(G)).$$
By a spectral sequence argument, we have that  $\Tot
(\widetilde{C}_{*, *}(G))$ is quasi-isomorphic, as complexes of
$kG$-modules, to the following complex
$$K_*: k\leftarrow 0\leftarrow k\leftarrow 0\leftarrow k\leftarrow \cdots.$$
So we have a quasi-isomorphism
$$k\otimes_{kG}^{\mathbb{L}} K_* \simeq k\otimes^{\mathbb{L}}_{kG} \Tot(\widetilde{C}_{*, *}(G)).$$
Taking the following projective resolution $\widetilde{K}_{*, *}$ of $K_*$,
$$\xymatrix@R=1.3pc{
\ar[d]^-{\widetilde{b}} & \ar[d]  & \ar[d]^-{\widetilde{b}} \\
A^{\otimes 3}\ar[d]^-{\widetilde{b}} & \ar[l] 0\ar[d]  & A^{\otimes 3} \ar[l] \ar[d]^{\widetilde{b}} & \ar[l]\\
A^{\otimes 2}\ar[d]^-{\widetilde{b}} & \ar[l]\ar[d]0 & A^{\otimes 2} \ar[l] \ar[d]^{\widetilde{b}} & \ar[l]\\
A          &     \ar[l] 0                      &       \ar[l]A      & \ar[l]
}$$
 we get that
$$\HC_n(G)\simeq H_n(k\otimes_{kG}^{\mathbb{L}}K_*)\simeq H_n(k\otimes_{kG}^{\mathbb{L}}\Tot(\widetilde{K}_{*, *}))\simeq H_n(k\otimes_{kG}\Tot(\widetilde{K}_{*, *})) \simeq \bigoplus_{i\in\Z_{\geq 0}} \mathrm H_{n-2i}(G, k),$$
where the forth isomorphism comes from the fact that $k\otimes_{kG} \widetilde{\Barr}_*(A)\otimes_{kG} k\simeq C_*(G, k)$.  
We have the following commutative diagram,
$$\xymatrix{
\HC_n(G) \ar[r]^-{S}\ar[d]^-{\simeq} & \HC_{n-2}(G)\ar[d]^-{\simeq}\\
\bigoplus_{i\in\Z_{\geq 0}} \mathrm H_{n-2i}(G, k)\ar[r]^-{\widetilde{S}} &
\bigoplus_{i\in\Z_{\geq 0}} \mathrm H_{n-2-2i}(G, k)}$$ where
$\widetilde{S}$ is induced by the identity morphisms from
$\mathrm H_{n-2i}(G, k)$ to $\mathrm H_{n-2i}(G, k)$ for $i>0$. So the kernel of $S$
is $\mathrm H_n(G, k)$.
\end{proof}



\end{document}